\documentclass[12pt]{amsart}

\usepackage{amsmath,amsfonts,amssymb,amsthm, a4wide, url, tikz-cd, mathscinet,bm,mathtools,mathdots,amsrefs,doi,comment,tikz-cd,hyperref}

\newtheorem{theorem}{Theorem}[section]
\newtheorem{lemma}[theorem]{Lemma}

\newtheorem{definition}[theorem]{Definition}

\newtheorem{proposition}[theorem]{Proposition}%[subsection]
\newtheorem{corollary}[theorem]{Corollary}%[subsection]
\theoremstyle{remark}
\newtheorem{remark}[theorem]{Remark}%[subsection]
\newtheorem{example}[theorem]{Example}%[subsection]
\numberwithin{equation}{section}

\newcommand{\A}{\mathbb{A}}                                                                          % adeles
\newcommand{\aaa}{\mathfrak{a}}                                                                      % vector space of Cartan
\newcommand{\Lorb}{\mathfrak{o}}                                                                     % nilpotent orbit
\newcommand{\bs}{\backslash}                                                                         % backslash
\newcommand{\R}{\mathbb{R}}                                                                          % real numbers
\newcommand{\C}{\mathbb{C}}                                                                          % complex numbers
\newcommand{\Z}{\mathbb{Z}}                                                                          % integers
\newcommand{\Q}{\mathbb{Q}}                                                                          % rational numbers
\newcommand{\PVS}{PVS}                                                                               % Prehomogeneous vector space
\newcommand{\rk}{\operatorname{rk}}                                                                  % rank
\newcommand{\centder}[1]{{#1}^\star}                                                                 % group generated by stabilizer and derived group
\newcommand{\fundchr}{\mathfrak{X}}                                                                  % fundamental characters
\newcommand{\reducd}[1]{\overline{#1}}                                                               % vector space generated by fundamental characters
\newcommand{\Sym}{\operatorname{Sym}}                                                                % symmetric matrices / symmetric power
\newcommand{\Skew}{\operatorname{Skew}}                                                              % skew symmetric matrices
\newcommand{\Pf}{\operatorname{Pf}}                                                                  % Pfaffian
\newcommand{\Env}{\mathbf{E}}                                                                        % parabolic group generated by ``effective'' simple roots
\newcommand{\excp}{\operatorname{excp}}                                                              % exceptional subspaces
\newcommand{\Stab}{\operatorname{Stab}}                                                              % stabilizer
\newcommand{\iii}{\mathrm{i}}                                                                        % square-root of -1
\newcommand{\PD}{\operatorname{PD}}                                                                  % Pontryagin dual
\newcommand{\Exp}{\operatorname{Exp}}                                                                % exponentional of a character
\newcommand{\Vreg}{X}                                                                                % homogenous space of a PVS
\newcommand{\Hom}{\operatorname{Hom}}                                                                % morphisms
\newcommand{\nsr}{isotropic}                                                                         % isotropic - centralizer contains split central torus
\newcommand{\sr}{non-isotropic}                                                                      % non-isotropic
\newcommand{\rd}{\operatorname{rd}}                                                                  % rapid decay
\newcommand{\vol}{\operatorname{vol}}                                                                % volume
\newcommand{\Ad}{\operatorname{Ad}}                                                                  % adjoint representation
\newcommand{\indx}{\mathcal{I}}                                                                      % indexing set for a basis of Lie group
\newcommand{\jndx}{\mathcal{J}}                                                                      % indexing set for a basis of V
\newcommand{\rest}{\big|}                                                                            % restriction
\newcommand{\Sub}{\Gamma}                                                                            % subspaces that are some of T_0-eigenspaces
\newcommand{\BASIC}{basic}                                                                           % centralizer not contained in any parabolic
\newcommand{\reg}{\operatorname{reg}}                                                                % regular part of V
\newcommand{\sreg}{\operatorname{niso}}                                                              % non-isotropic part
\newcommand{\Minset}{\Sub_{\reg}}                                                                    % subspaces containing regular elements
\newcommand{\SPCL}{special}
\newcommand{\Spcl}{\Sub_{\operatorname{spcl}}}                                                       % ``special'' subspaces
\newcommand{\Ker}{\operatorname{Ker}}                                                                % kernel
\newcommand{\rad}[1]{N_{#1}}                                                                        % unipotent radical
\newcommand{\tor}{\operatorname{tor}}                                                               % the toric part of a group
\newcommand{\ssm}{\operatorname{ss}}                                                                 % the semisimple part
\newcommand{\ssu}{\operatorname{ssu}}                                                               % semisimple + unipotent part
\newcommand{\red}{\operatorname{red}}                                                                % reductive part
\newcommand{\levi}{\operatorname{lev}}                                                                % Levi part
\newcommand{\mdchr}{\Delta}                                                                          % modulus character (algebraic)
\newcommand{\supp}{\operatorname{supp}}                                                              % support  
\newcommand{\GL}{\operatorname{GL}}                                                                  % general linear group
\newcommand{\SL}{\operatorname{SL}}                                                                  % special linear group
\newcommand{\PGL}{\operatorname{PGL}}                                                                % projective linear group
\newcommand{\Orth}{\operatorname{O}}                                                                 % orthogonal group
\newcommand{\SO}{\operatorname{SO}}                                                                  % special orthogonal group
\newcommand{\Sp}{\operatorname{Sp}}                                                                  % symplectic group
\newcommand{\GSp}{\operatorname{GSp}}
\newcommand{\Spin}{\operatorname{Spin}}
\newcommand{\GSpin}{\operatorname{GSpin}}
\newcommand{\GE}{\operatorname{GE}}
\newcommand{\spin}{\operatorname{spin}}
\newcommand{\Mat}{\operatorname{Mat}}
\newcommand{\diag}{\operatorname{diag}}
\newcommand{\RI}{\operatorname{RI}}                                                                  % relative invariant
\newcommand{\Lie}{\operatorname{Lie}}
\newcommand{\Liealg}[1]{\mathfrak{\lowercase{#1}}}
\newcommand{\sprod}[2]{\left\langle{#1},{#2}\right\rangle}
\newcommand{\abs}[1]{\left|{#1}\right|}
\newcommand{\swrz}{\mathcal{S}}
\newcommand{\Std}{\operatorname{Std}}                                                                 % standard representation
\newcommand{\Siegel}{\mathfrak{S}}
\newcommand{\basis}{\mathcal{B}}
\newcommand{\rts}{\Phi}
\newcommand{\wgts}{\Psi}
\newcommand{\cone}{\mathcal{C}}
\newcommand{\face}{\mathcal{F}}
\newcommand{\Ind}{\operatorname{Ind}}                                                                % induction of nilpotent orbits
\newcommand{\pr}{\operatorname{pr}}
\newcommand{\IFD}{IFD}                                                                               % induced filtration data
\newcommand{\data}{\mathfrak{d}}                                                                     % IFD item
\newcommand{\Indata}{\mathfrak{D}}                                                                   % set of IFDs
\newcommand{\Cls}{\mathfrak{C}}                                                                      % orbit of IFDs
\newcommand{\filt}{\mathcal{F}}                                                                      % filtration
\newcommand{\der}{\operatorname{der}}                                                                % derived group
\newcommand{\ab}{\operatorname{ab}}                                                                  % abelianization
\newcommand{\Gal}{\operatorname{Gal}}                                                                % Galois group
\newcommand{\fin}{\operatorname{fin}}
\newcommand{\cmpl}{\mathfrak{b}}                                                                     % a complement
\newcommand{\id}{\operatorname{id}}                                                                  % identity
\newcommand{\vlum}{\kappa}                                                                           % for volume
\newcommand{\sm}[4]{{\bigl(\begin{smallmatrix}{#1}&{#2}\\{#3}&{#4}\end{smallmatrix}\bigr)}}
\newcommand{\norm}[1]{\lVert#1\rVert}
\newcommand{\B}{\=}                                                                                  % Needed for Sato in biliography
\newcommand{\Pic}{\operatorname{Pic}}                                                                % Picard group
\newcommand{\Cl}{\operatorname{Cl}}                                                                  % divisor class group
\newcommand{\Spec}{\operatorname{Spec}}                                                              % Spectrum
\newcommand{\Mlt}{\mathbf{G}_m}                                                                      % multiplicative group
\newcommand{\OOO}{\mathcal{O}}                                                                       % structure sheaf
\newcommand{\Eroot}[2]{\begin{psmallmatrix*}[l]{#1}\\\phantom{0}\phantom{0}{#2}\end{psmallmatrix*}}
\newcommand{\Rich}{\operatorname{Rich}}                                                              % Richardson orbit
\newcommand{\Domain}{\mathcal{D}}                                                                    % a domain

\newcommand{\one}{\mathbf{1}}
\newcommand{\CF}{CF}                                                                                 % completely factorizable regular PVS
\newcommand{\fs}{\operatorname{fs}}                                                                  % full support (for decomposition of Theta)
\newcommand{\nfs}{\operatorname{nfs}}                                                                % non-full suport
\newcommand{\PS}{positive envelope}

\begin{document}

\title[Zeta functions of PVSs]{On the convergence of zeta functions of prehomogeneous vector spaces}

\author{Tobias Finis}
\address{Institute of Mathematics, University of Leipzig, Leipzig, Germany}
\email{Tobias.Finis@math.uni-leipzig.de}
\author{Erez M. Lapid}
\address{Department of Mathematics, Weizmann Institute of Science, Rehovot Israel}
\email{Erez.M.Lapid@gmail.com}

\date{\today}

\begin{abstract}
We prove a general convergence result for zeta functions of prehomogeneous vector spaces extending results of H. Saito, F. Sat\B{o} and Yukie.
Our analysis points to certain subspaces which yield boundary terms. We study it further in the setup arising from nilpotent orbits. 
In certain cases we determine the residue at the rightmost pole of the zeta function.
\end{abstract}

\maketitle
\setcounter{tocdepth}{1}
\tableofcontents

\section{Introduction}

The theory of prehomogeneous vector spaces ({\PVS}s) and their zeta functions was initiated by M. Sato and Shintani \cites{MR296079, MR1086566}.
(See \cite{MR1944442} for history and references.)
One of the early motivations was to study asymptotics of arithmetic invariants \cite{MR0292755},
an endeavor followed up by a number of authors (e.g., \cites{MR1267735, MR1185585, MR936994, MR2138067}).
For these applications, the alternative approach of Bhargava using innovative Geometry of Numbers techniques was very successful (see \cite{MR2275598} and the references therein).
Nonetheless, zeta functions of {\PVS}s are still useful for these applications \cite{2107.12819}.
Our motivation for the study of the zeta functions comes from their appearance in Arthur's trace formula \cites{MR903631, MR3675167, MR3843151, MR3969874, MR3625130},
although this will not play a direct role in the current paper.

The purpose of this paper is to provide a general basic convergence result for zeta functions of regular {\PVS}s over numbers fields. 
It generalizes earlier results in this direction by F. Sat\B{o}, H. Saito, Yukie and others.
More importantly, we hope that the method will enhance our understanding of the singularities
of the zeta function and ultimately will be incorporated to Arthur's trace formula.
As a first result in this direction, we determine the structure of the first pole (at the edge of the domain of absolute convergence)
for a restricted, but natural class of {\PVS}s.

In order to state our main result, let $V$ be a reductive {\PVS} over a number field $F$, i.e, an $F$-rational representation $\rho$
of a reductive group $G$ over $F$ on $V$ with an open orbit $\Vreg$. Let $H=G_v$ be the generic stabilizer
and let $\chi_1,\dots,\chi_r\in X^*(G)$ be the fundamental $F$-rational characters of $G$.
Let $\A$ be the ring of adeles of $F$.
For any rapidly decreasing continuous function $\phi$ on $V(\A)$ let $\theta_\phi$ be the automorphic function
\[
\theta_\phi(g)=\sum_{v\in \Vreg(F)}\phi(\rho(g)^{-1}v),\ \ g\in G(F)\bs G(\A).
\]
Roughly speaking, the zeta function is the Mellin transform of $\theta_\phi$.
In the simplest case, our result is the following.
\begin{theorem} \label{thm: mainsimp}
Suppose that the connected component $H^\circ$ of $H$ is semisimple. Then, the zeta function
\[
Z(\phi,(s_1,\dots,s_r))=\int_{G(F)\bs G(\A)}\theta_\phi(g)\abs{\det\rho(g)}^{-1}\prod_{i=1}^r\abs{\chi_i(g)}^{-s_i}\ dg
\]
converges absolutely for $\Re s_i>0$, $i=1,\dots,r$.
\end{theorem}

Theorem \ref{thm: mainsimp} cannot hold as stated for a general {\PVS}. The first issue is that assuming convergence, as a function of $x\in G(\A)$ the integral
\begin{equation} \label{eq: prezeta}
\int_{G(F)\bs G(\A)^1}\theta_\phi(gx)\ dg
\end{equation}
is invariant not just under $G(\A)^1$ but also under the adelic points of the group $\centder{G}=HG^{\der}$. Therefore, in the case where $X^*(H)$ is infinite
(or, equivalently $r<\dim X^*(G)$) we need to modify the definition of the zeta function as follows:
\begin{equation} \label{eq: zetanonisotropic}
Z(\phi,(s_1,\dots,s_r))=\int_{G(\A)^1\centder{G}(\A)\bs G(\A)}\int_{G(F)\bs G(\A)^1}\theta_\phi(ga)\ dg\ \abs{\det\rho(a)}^{-1}\prod_{i=1}^r\abs{\chi_i(a)}^{-s_i}\ da.
\end{equation}
We will show that this integral converges absolutely for $\Re s_i>0$, $i=1,\dots,r$ provided that $H$ is connected reductive.

In the non-connected case, there is yet a more fundamental difficulty, namely that already the integral \eqref{eq: prezeta} may diverge.
This happens if $\vol(G_v(F)\bs (G_v(\A)\cap G(\A)^1))=\infty$, or equivalently, if
the group of connected components $G_v/G_v^\circ$ acts non-trivially on $X^*(G_v^\circ)$ for some $v\in \Vreg(F)$. We call such elements $v$ {\nsr}.
For instance, this is the case for the {\PVS} of binary quadratic forms (with respect to the action of $\GL_2$) in which the centralizer
(an orthogonal group in dimension $2$) may split over $F$. 

To remedy the situation, we modify $\theta_\phi$, and consequently the zeta function, by simply removing the sum over {\nsr} $v\in \Vreg(F)$. Our main result 
is that the modified zeta function converges for $\Re s_i>0$, $i=1,\dots,r$ provided that $H^\circ$ is reductive (see Theorem \ref{thm: convergence} below).

In fact, using different methods, the convergence of this modified zeta function in \emph{some} orthant $\Re s_i\gg0$ was proved by H. Saito
under the condition that $X^*(H)$ is finite \cite{MR1994885}.

We also remark that meromorphic continuation to $\C^r$ and functional equations for the zeta function 
\eqref{eq: zetanonisotropic} were established by F. Sat\B{o} 
\cite{MR676121} under the assumptions that $H^\circ$ is reductive, $X^*(H)$ is finite and that there are no {\nsr} elements in $X(F)$,
extending the classical work by M. Sato and Shintani \cite{MR0344230}.
For the {\PVS} of binary quadratic forms, meromorphic continuation of \eqref{eq: zetanonisotropic} to $\C$ and the functional equation
of a modified zeta function, which includes a contribution from the {\nsr} elements, were obtained by Yukie \cite{MR1149040}
(following the work of Shintani \cite{MR0384717}). We will not deal with these questions here.

The contents of the paper are as follows.
We start with recalling general facts, mostly well known, about {\PVS}s (\S\ref{sec: {\PVS}}).
We review examples of {\PVS}s, especially those arising from nilpotent orbits, in \S\ref{sec: nilPVS}.
The main result is stated in \S\ref{sec: statement}.
In \S\ref{sec: thetabnd} we give a general bound for the function $\theta_\phi$ (Proposition \ref{prop: bndtheta}).
The analysis is based on the examination of the support sets of elements in $\Vreg(F)$ with respect to the weights of a maximal $F$-split torus of $G$,
an idea which goes back to Yukie \cites{MR1267735, MR1369477}.
The linchpin of the argument is a certain non-negativity statement pertaining to these support sets (Corollary \ref{cor: lam(U)}).
Further analysis is carried out in \S\ref{sec: further} where the key notion of \emph{\SPCL} subspaces of a {\PVS} is introduced.
In \S\ref{sec: pfmain} we prove the main convergence result. An important aspect in the proof is that the existence of {\nsr} elements is closely
related to the existence of {\SPCL} subspaces with additional properties, which we call exceptional (Proposition \ref{prop: stronglyregular}).
The non-exceptional subspaces exhibit a crucial positivity property (Proposition \ref{prop: convregs}) which strengthens Corollary \ref{cor: lam(U)} \textit{supra} and
guarantees convergence. In \S\ref{SectionSimpleCase} we improve the convergence result for {\PVS}s without non-trivial {\SPCL} subspaces
(under an additional technical condition which is satisfied in the $F$-irreducible case) and establish meromorphic continuation to the left of the polar hyperplanes
$s_i = 0$, $i = 1, \ldots, r$ together with an explicit description of the residues.
Finally, in \S\ref{sec: examspcl} we give a general construction of {\SPCL} subspaces for regular PVSs coming from nilpotent orbits. 

\section{Prehomogeneous vector spaces} \label{sec: {\PVS}}

We start with basic facts about {\PVS}s.  (See \cite{MR1944442}*{Ch.~2} and \cite{MR676121}*{\S1} for more details.)

\subsection{General notation}
Throughout the paper $F$ is a field of characteristic $0$ with algebraic closure $\bar F$.
(From \S\ref{sec: statement} onward, $F$ will be a number field.)

We denote by $\Mlt$ the multiplicative group.

Whenever a group $G$ acts on a set $A$, we denote by $A^G$ the set of fixed points of $G$ in $A$.

We denote by $X_*(T)$ the lattice of $F$-rational cocharacters of an $F$-torus $T$.

For a linear algebraic group $G$ over $F$ we use the following notation.
\begin{itemize}
\item $G^\circ$ -- the identity component of $G$.
\item $Z(G)$ -- the center of $G$.
\item $G^{\der}$ -- the derived group of $G$.
\item $G^{\ab}=G/G^{\der}$ -- the abelianization of $G$.
\item $\rad{G}$ -- the unipotent radical of $G$.
\item $G^{\levi}=G/\rad{G}$ -- the Levi part of $G$, so that $G^{\red}:=(G^{\levi})^{\circ}=(G^\circ)^{\levi}$ is reductive.
\item $G^{\ssm}=(G^{\red})^{\der}$ -- the semisimple part of $G^\circ$.
\item $G^{\tor}=(G^{\red})^{\ab}$ -- the largest toric quotient of $G^\circ$.
\item $G^{\ssu}=\Ker (G^\circ\rightarrow G^{\tor})$ -- an extension of $G^{\ssm}$ by $\rad{G}$.
\item $\Liealg{G}=\Lie G$ -- the Lie algebra of $G$ (and similarly for other groups).
\item $\Ad=\Ad_G$ -- the adjoint representation of $G$ on $\Liealg{G}$.
\item $X^*(G)=X^*(G^{\ab})=X^*(G^{\levi})$ -- the finitely generated abelian group of $F$-rational characters of $G$.
\item $\mdchr_G=\det\Ad\in X^*(G)$ -- the modular character of $G$.
\item $T_G$ -- the maximal $F$-split torus of $Z(G^{\red})$.
\item $\aaa_G=X_*(T_G)\otimes_{\Z}\R$ -- the $\R$-vector space generated by $X_*(T_G)$.
\end{itemize}
We often view $G^{\levi}$ as an $F$-subgroup (the Levi subgroup) of $G$, uniquely defined up to conjugation by $\rad{G}(F)$ (\cite{MR0092928} or \cite{MR620024}*{Ch. VIII, Theorem 4.3}).

If $G$ is connected, then $X^*(G)=X^*(G^{\tor})$ is torsion-free, i.e., it is a lattice.

If $G$ is reductive, then $G^{\ssu}=G^{\ssm}=G^{\der}$.

If $T$ is an $F$-split torus, then $X^*(T)$ can be identified with the dual lattice of $X_*(T)$.
In general, if $T'$ is the maximal $F$-split subtorus of $T$, then we can identify $X_*(T)$ with $X_*(T')$,
while the restriction map $X^*(T)\rightarrow X^*(T')$ is injective with finite cokernel. Thus, we have isomorphisms
\[
X^*(T)\otimes_{\Z}\R\simeq X^*(T')\otimes_{\Z}\R\simeq \aaa_{T'}^*= \aaa_T^*.
\]
The $F$-split rank of $T$, i.e. the dimension of $T'$, is the rank of $X^*(T)$ (and $X_*(T)$).

More generally, if $G$ is connected, then the restriction map $X^*(G)=X^*(G^{\red})\rightarrow X^*(T_G)$
is injective with finite cokernel and hence yields an isomorphism
\[
X^*(G)\otimes_{\Z}\R\simeq X^*(T_G)\otimes\R=\aaa_G^*.
\]

For any finite subset $A$ of a real vector space we denote by $\R_{\ge0}A$ the polyhedral cone generated by $A$, i.e.,
the set of linear combinations $\sum_{a\in A}x_aa$, where $x_a\ge0$ for all $a\in A$.
Similarly, we use the notation $\R_{>0}A$, $\Z_{\ge0}A$, $\Z_{>0}A$ where we require the $x_a$ to be strictly positive (resp., non-negative integers,
strictly positive integers) for all $a$.

Let $G$ be an algebraic $F$-group and let $H$ be an $F$-subgroup of $G$.
We record a few basic results.

\begin{lemma} \label{lem: genHG}
Suppose that $G$ is connected. Then, the following two conditions are equivalent.
\begin{enumerate}
\item \label{part: cokGH} The restriction map $X^*(G)\rightarrow X^*(H^\circ)$ has finite cokernel.
\item \label{part: centm} The torus $(Z(H^{\red})\cap G^{\ssu})^\circ$ is $F$-anisotropic.
\end{enumerate}
\end{lemma}

\begin{proof}
We may assume that $H$ is connected. Fix a Levi subgroup $L\simeq H^{\red}$ of $H$. Then, the restriction map $X^*(H)\simeq X^*(L)\rightarrow X^*(T_L)$
is injective with finite cokernel. Therefore, we may rephrase condition \ref{part: cokGH} by saying that
the restriction map $X^*(G)=X^*(G^{\tor})\rightarrow X^*(T_L)$ has finite cokernel, or equivalently,
that $X_*(T_L)\rightarrow X_*(G/G^{\ssu})$ is injective.
In turn, this is equivalent to the triviality of $(T_L\cap G^{\ssu})^\circ$, which amounts to condition \ref{part: centm}.
\end{proof}

\begin{lemma} \label{lem: charconn}
Suppose that $H=G^\circ$ and let $\Gamma=G/H$.
Then, the kernel of the restriction map $X^*(G)\xrightarrow{r} X^*(H)$ is finite and its image is a finite index sublattice of $X^*(H)^G=X^*(H)^\Gamma$.
Thus, the cokernel of $r$ is finite if and only if $G$ acts trivially on $X^*(H)$.
\end{lemma}

\begin{proof}
The first part follows from the fact that the homomorphism $H/[G,H]\rightarrow G^{\ab}$ has finite kernel and cokernel.
In fact we have (a $\Gal(\bar F/F)$-equivariant) exact sequence
\[
H_2(\Gamma,\Z)\rightarrow H/[G,H]\rightarrow G^{\ab}\rightarrow\Gamma^{\ab}\rightarrow1
\]
\cite{MR1438546}*{Corollary VI.8.2}. The second part follows from the first part and the fact that if a finite group $\Gamma$ acts non-trivially on a lattice $L$, then
$L^\Gamma$ is of infinite index in $L$. 
\end{proof}

We denote the Picard group of a variety $X$ by $\Pic X=H^1(\Vreg,\OOO_X^*)$. If $X$ is smooth, then we can identify $\Pic X$ with the (Weil) divisor class group
$\Cl X$  of $X$. (See \cite{MR463157}*{\S II.6} for basic facts about divisors and the Picard group.)

\begin{lemma}[\cite{MR260758}*{Proposition VII.1.5}, cf.~also the proof of \cite{MR357399}*{Theorem 4}] \label{lem: pic}
Suppose that $G$ is connected. Then, there is an exact sequence
\[
X^*(G)\rightarrow X^*(H)\rightarrow \Pic (G/H).
\]
\end{lemma}

\subsection{}
Recall that a \emph{prehomogeneous vector space} ({\PVS}) over $F$ is a finite-dimensional $F$-representation
\[
\rho:G\rightarrow\GL(V)
\]
of a connected linear algebraic $F$-group $G$ on a finite-dimensional $F$-vector space $V$, for which there exists
a (necessarily unique) open orbit $\Vreg$.
Note that $\Vreg$ is defined over $F$ and $\Vreg(F)$ is $G(F)$-invariant.
If $F$ is algebraically closed, then $\Vreg(F)$ is a $G(F)$-orbit and the stabilizers $G_v$, $v\in \Vreg(F)$ are conjugate by $G(F)$.
In general, given $u\in \Vreg(F)$, the $G(F)$-orbits in $\Vreg(F)$ are in bijection with the elements of the kernel of 
$H^1(F,G_u)\rightarrow H^1(F,G)$ (a map of pointed sets) \cite{MR1466966}*{I \S5.4, Corollary 1 of Proposition 36}.
Correspondingly, the $F$-groups $G_v$, $v\in \Vreg(F)$ are pure inner forms of $G_u$.

Clearly, a quotient of a {\PVS} by an invariant subspace is also a {\PVS}.

We will write $H=G_v$ (the ``generic stabilizer'') when $v\in \Vreg(F)$ is immaterial, e.g., when we refer to the group $X^*(H)$ or to properties pertaining to $H(\bar F)$.
Thus, we can identify $\Vreg$ with $G/H$. The $F$-subgroup $\centder{G}=G^{\der} H$ of $G$ is well defined.
Note that in general, $H$ may be disconnected and the groups $G_v^\circ$, $v\in \Vreg(F)$ are \emph{not} necessarily in the same inner class.
In fact, the lattice $X^*(G_v^\circ)$ may depend on the choice of $v\in \Vreg(F)$. For this reason we will refrain from using the notation $X^*(H^\circ)$.

For $v\in\Vreg(F)$, the linear map
\begin{equation} \label{def: Dv}
D_v:\Liealg{G}\rightarrow V,\ \  x\mapsto d\rho(x)v
\end{equation}
is surjective, since it is the differential of the dominant map $G\rightarrow V$, $g\mapsto\rho(g)v$.
The kernel $\Liealg{G}_v$ of $D_v$ is the Lie algebra of $G_v$ \cite{MR1102012}*{Lemma II.7.4}.
We thus get a $G_v$-equivariant isomorphism $\Liealg{G}/\Liealg{G}_v\rightarrow V$.
In particular,
\begin{equation} \label{eq: mdchr}
\mdchr_G\det\rho^{-1}\rest_H=\mdchr_H.
\end{equation}

A \emph{relative invariant} with respect to a character $\chi\in X^*(G)$ is an $F$-rational function $f\not\equiv0$ on $V$ such that
$f\circ\rho(g)=\chi(g)f$ for all $g\in G$. We will write $f=\RI_\chi$.

The following basic properties hold.
\begin{enumerate}
\item $\RI_\chi$, if it exists, is unique up to a non-zero scalar and it is regular and non-vanishing on $\Vreg$.
\item The sublattice $X^*(G/\centder{G})$, which is the kernel of the restriction map $X^*(G)\rightarrow X^*(H)$,
consists of the characters with respect to which there exists a relative invariant.
\item Let $\chi_1,\dots,\chi_r$ be a basis of $X^*(G/\centder{G})$. Then $\RI_{\chi_1},\dots,\RI_{\chi_r}$ are algebraically independent.
\item The $F$-irreducible polynomials that are nowhere vanishing on $\Vreg$ (if any)
are relative invariants. They are called the \emph{fundamental relative invariants}.
The corresponding characters are called \emph{fundamental characters}.
They form a basis for $X^*(G/\centder{G})$.
Denote the set of all fundamental characters by $\fundchr$.
\end{enumerate}

If $G$ is reductive, then the following conditions are equivalent.
\begin{enumerate}
\item $H^\circ$ is reductive.
\item $\Vreg$ is the complement of a hypersurface.
\item $\Vreg$ is the complement of the hypersurface defined by $\RI_{\chi_1}\cdots\RI_{\chi_r}$ where $\fundchr=\{\chi_1,\dots,\chi_r\}$.
\item $\Vreg$ is an affine variety.
\item $V$ is regular (see \cite{MR1944442}*{Definition 2.14}).
\end{enumerate}
In this case, $(\det\rho)^2\in X^*(G/\centder{G})$.
(As far as we know, it is unknown whether $\RI_{\det\rho^2}$ is a polynomial in general.)

We give some more basic properties of {\PVS}s.

\begin{lemma} \label{lem: restggv}
There is a short exact sequence
\[
0 \rightarrow X^*(G/\centder{G}) \rightarrow X^* (G)\rightarrow X^* (H) \rightarrow 0.
\]
Thus, $\rk X^*(G)-r=\rk X^*(H)$.
\end{lemma}

\begin{proof}
Clearly, $X^*(G/\centder{G})$ is the kernel of the restriction map $X^*(G)\rightarrow X^*(H)$.
The surjectivity of this map follows from Lemma \ref{lem: pic} and the vanishing of $\Pic\Vreg$. 
More generally, this is true for any open subvariety $\Vreg$ of $Y=\Spec A$ where $A$ is a Noetherian UFD since
$\Cl Y=1$ and $\Cl Y\rightarrow\Cl \Vreg$ is onto.
\end{proof}

Let $\reducd{\aaa}_G^*$ be the subspace $X^*(G/\centder{G})\otimes\R$ of $\aaa_G^*$. Thus, $\fundchr$ forms a basis for $\reducd{\aaa}_G^*$.

For completeness, we also note the following fact, which can also be proved by a topological argument for $F=\C$.

\begin{lemma} \label{lem: speconn}
Suppose that $V$ has no relative invariants. Then $H$ is connected.
\end{lemma}

\begin{proof}
We may assume that $F$ is algebraically closed.
The condition on $V$ means that $V\setminus \Vreg$ has codimension $>1$ in $V$.
It follows from purity of the branch locus that $\Vreg$ is simply connected \cite{MR2017446}*{Exp.~X, Corollaire 3.3}, i.e., there is no non-trivial \'etale covering of $\Vreg$.
On the other hand,
\[
G/H^\circ\rightarrow \Vreg
\]
is an \'etale covering. Thus, $H$ is connected.
\end{proof}

\begin{lemma} \label{lem: isot}
The following conditions on $v\in\Vreg(F)$ are equivalent.
\begin{enumerate}
\item \label{part: cok1} The cokernel of the restriction map $X^*(G)\rightarrow X^*(G_v^\circ)$ is infinite.
\item \label{part: cok2} The cokernel of the restriction map $X^*(G_v)\rightarrow X^*(G_v^\circ)$ is infinite.
\item \label{part: splt} There is a non-trivial $F$-split torus in $Z(G_v^{\red})\cap G^{\ssu}$.
\item \label{part: nontact} The action of $G_v$ (via the component group $G_v/G_v^\circ$) on $X^*(G_v^\circ)$ is nontrivial.
\end{enumerate}
\end{lemma}

\begin{proof}

\ref{part: cok1}$\iff$\ref{part: cok2}: Lemma \ref{lem: restggv}.

\ref{part: cok1}$\iff$\ref{part: splt}: Lemma \ref{lem: genHG}.

\ref{part: cok2}$\iff$\ref{part: nontact}: Lemma \ref{lem: charconn}.
\end{proof}

\begin{definition} \label{def: sr}
An element $v\in\Vreg(F)$ is \emph{\nsr} if it satisfies the conditions of Lemma \ref{lem: isot}. Otherwise, $v$ is \emph{\sr}. 

We denote the subset of {\sr} elements in $\Vreg(F)$ by $\Vreg(F)^{\sreg}$. It is invariant under $G(F)$.

A {\PVS} $V$ is {\nsr} if it contains {\nsr} elements. Otherwise, $V$ is {\sr}.
\end{definition}

By Lemma \ref{lem: isot}, if $H$ is connected, then $V$ is {\sr}.

The basic example of an {\nsr} {\PVS} is the space $V$ of binary quadratic forms with respect to the action of $\GL_2$.
Here, $H$ is an orthogonal group in dimension $2$; $v\in\Vreg(F)$ is {\nsr} if and only it is isotropic in the usual sense of quadratic forms.
We caution however, that the {\PVS} of quadratic forms in $n>2$ variables (with respect to the $\GL_n$-action) is {\sr}.

In general, it is possible that all elements of $\Vreg(F)$ are {\nsr} (for instance, if $V$ is {\nsr} and $F=\bar F$),
although we expect that this is never the case if $F$ is a number field.

Assume now that $G$ is reductive, so that every representation of $G$ is completely reducible.
We single out two important classes of {\PVS}s in the opposite extremes.

\begin{definition} \label{def: basic}
Let $V$ be a reductive {\PVS}. We say that
\begin{enumerate}
\item $V$ is \emph{{\BASIC}} if for every $v\in \Vreg(F)$, $G_v^\circ$ is not contained in any proper parabolic $F$-subgroup of $G$.\footnote{This condition
depends on $v$. A strictly stronger condition would be that $H^\circ$ is not contained in any proper parabolic subgroup of $G$ (not necessarily defined over $F$).}
\item $V$ is \emph{distinguished} if $H/\Ker\rho$ is finite.\footnote{The terminology comes from nilpotent orbits -- see \S \ref{sec: nilPVS} below.}
\end{enumerate}
\end{definition}

Note that every {\BASIC} {\PVS} is regular, since any connected non-reductive $F$-subgroup of $G$ is contained in a proper parabolic $F$-subgroup of $G$ \cite{MR0294349}.
Moreover, a {\BASIC} {\PVS} is necessarily {\sr}.
It is also clear that every subrepresentation of a {\BASIC} {\PVS} is also {\BASIC}.

A distinguished {\PVS} is regular and {\sr}.

\section{{\PVS}s of DK-type, examples} \label{sec: nilPVS}

A main source of {\PVS}s are gradations on reductive Lie algebras and nilpotent orbits.

\subsection{}
Let $G'$ be a reductive group over $F$ and let $\lambda:\Mlt\rightarrow G'$ be a one-parameter subgroup.
Let $h=(d\lambda)(1)\in\Liealg{G}'$. Decompose
\[
\Liealg{G}'=\oplus_{i\in\Z}\Liealg{G}'_i,\ \ \text{where }\Liealg{G}'_i=\{x\in\Liealg{G}'\mid\Ad(\lambda(t))x=t^ix\ \forall t\in\Mlt\}=
\{x\in\Liealg{G}'\mid [h,x]=ix\}.
\]
This is a $\Z$-gradation of $\Liealg{G}'$. (If $G'$ is semisimple and adjoint, then every gradation of $\Liealg{G'}$ is of this form.)
Let $\filt_i=\oplus_{j\ge i}\Liealg{G}'_j$, $i\in\Z$ be the corresponding (decreasing) filtration of $\Liealg{G}'$.
The stabilizer $P$ of $\filt$ in $G$ under $\Ad$ is the parabolic $F$-subgroup of $G$ whose Lie algebra is $\filt_0$.
The Lie algebra of the unipotent radical of $P$ is $\filt_1$.
The centralizer $G$ of $\lambda$ is a Levi subgroup of $P$ whose Lie algebra is $\Liealg{G}'_0$.
For every $i\ne0$, the vector space $\Liealg{G}'_i$, which consists of nilpotent elements,
comprises finitely many (geometric) orbits with respect to the adjoint action of $G$ \cite{MR0430168}.
In particular it is a {\PVS}. In the terminology of \cite{MR1125214}, these are {\PVS}s of V-type (for Vinberg). They are not necessarily regular. We also 
say that $\filt$ is a filtration of $\Liealg{G}'$ of V-type.
(Note that we may restrict here to the case $i=1$, since the general case follows from it by considering the
graded Lie subalgebra $\oplus_j\Liealg{G}'_{ji}$, which is reductive.)

As a special case, let $P$ be a parabolic subgroup of $G'$ with Levi factor $G$ and $\Liealg{N}=\Lie\rad{P}$.
Then, $V=\Liealg{N}/[\Liealg{N},\Liealg{N}]$ is a {\PVS} with respect to the adjoint action of $G$. 
In fact, $\Liealg{N}$ is a {\PVS} with respect to the adjoint action of $P$ by a well known result of Richardson. We say that $V$ is a {\PVS} of R-type.
We denote the nilpotent orbit containing the open orbit of $P$ on $\Liealg{N}$ (the Richardson orbit with
respect to $P$) by $\Rich(P)$.
The number of irreducible components of $V$ is the $F$-corank of $P$ in $G$.

Other important special cases, this time of regular {\PVS}s, arise from the Dynkin--Kostant classification of nilpotent orbits.
In more detail, let $(e,h,f)$ be an $\Liealg{SL}_2$-triple in $[\Liealg{G}',\Liealg{G}']$.
(By the Jacobson--Morozov theorem, every nilpotent $e\in\Liealg{G}'$ is a part of an $\Liealg{SL}_2$-triple.)
It corresponds to an $F$-homomorphism $\SL_2\rightarrow G'^{\der}$.
Let $\Liealg{G}'=\oplus_i\Liealg{G}'_i$ be the grading induced by $h$.
The corresponding (decreasing) filtration $\filt_i=\oplus_{j\ge i}\Liealg{G}'_j$, $i\in\Z$ of $\Liealg{G}'$ depends only on $e$ and in fact only on the $P$-orbit
of $e$, where $P=P(e)$ is the stabilizer of $\filt=(\filt_i)_{i\in\Z}$ in $G'$ under the adjoint action.
Recall that $P$ is the parabolic subgroup of $G'$
with Lie algebra $\filt_0$. The Lie algebra of the unipotent radical of $P$ is $\filt_1$. (That $\filt$ depends only on $e$ follows from the fact that two
$\Liealg{SL}_2$-triples $(e,h,f)$ and $(e,h',f')$ with the same first element $e$ are conjugate by an element of $N_P(F) \cap G'_e (F)$ \cite{MR2109105}*{Lemma VIII.11.1.4}.)

Let $G$ be the Levi quotient  
of $P$, so that $\Liealg{G}=\filt_0/\filt_1\simeq\Liealg{G}'_0$.
Let $V=\filt_2/\filt_3\simeq\Liealg{G}'_2$ with the $G$-action $\rho$ induced by the adjoint representation. 
Let $\Lorb$ be the geometric $G'$-orbit of $e$ and let $\Lorb_\filt=\Lorb\cap\filt_2$.

\begin{theorem} \label{thm: stdunip}
With the notation above:
\begin{enumerate}
\item $V$ is a regular {\PVS} with respect to $\rho$.
\item $\Lorb_\filt=\Vreg+\filt_3$.
\item $\Lorb_\filt$ is a $P$-orbit.
\item \label{part: moveo} For every $g\in G'$, $\Ad(g)\Lorb_\filt\cap\Lorb_\filt\ne\emptyset$ if and only if $g\in P$.
In particular, $G'_v\subset P$ for any $v\in\Lorb_\filt$.
\end{enumerate}
\end{theorem}

There are the {\PVS}s of \emph{DK-type} (Dynkin--Kostant) in the language of \cite{MR1125214}. 
(We quickly sketch how to derive these assertions from the literature. Note first that $\Lorb\cap \Liealg{G}'_2$ is the open $G$-orbit $\Vreg$ of $V$ by \cite{MR2109105}*{Proposition VIII.11.3.6}.
$V$ is regular by \cite{MR575790}, see below. In addition, the $N_P$-orbit of $e$ is $e + \filt_3$ \cite{MR0320232}*{Lemma 1}. This shows the second and third assertions. 
The fourth assertion easily reduces to the special case $G'_e\subset P$. Let $g \in G'_e$. Then $(e, \Ad (g) h, \Ad (g) f)$ is an $\Liealg{SL}_2$-triple
with the same first element as $(e,h,f)$, and therefore conjugate to $(e,h,f)$ by an element $n \in N_P \cap G'_e$. Since $ng$ centralizes $h$, we have $ng \in G$, and therefore $g \in P$.)

If we fix in addition to $e$ the $\Liealg{SL}_2$-triple $(e,h,f)$, then we can identify $G$ with the centralizer of $h$ in $G'$ and $V$ with the subspace $\Liealg{G}'_2$ of $\Liealg{G}'$.
We have then $\Vreg=V\cap\Lorb$. We may identify $\Liealg{G}'_{-2}$ with the dual of $V=\Liealg{G}'_2$ via the Killing form.
Fixing a vector space isomorphism between $\Liealg{G}'_{-2}$ and $\Liealg{G}'_2$, $\RI_{(\det\rho)^2}$ is given by the determinant of the linear map
$\operatorname{ad} (e)\circ\operatorname{ad} (e):\Liealg{G}'_{-2}\rightarrow\Liealg{G}'_2$,
a polynomial in $e\in V$ whose nonzero locus is $\Vreg$ \cite{MR575790}*{Proposition 1.1, 1.2, 2.1}.

Fix a minimal parabolic $F$-subgroup $P_0'$ of $G'$ and a maximal split $F$-torus $T_0'$ contained in $P_0'$.
Then, we may choose $(e,h,f)$ in its $G'(F)$-orbit so that $h\in\Liealg{T}_0'$ is dominant and $P=P(e)$ is standard.
This is the canonical parabolic subgroup pertaining to $\Lorb$.
$G$ is then identified with the standard Levi subgroup of $P$.

An orbit $\Lorb$ is called \emph{even} if $\Liealg{G}'_i=0$ for all odd $i$ (or equivalently, if $\Liealg{G}'_1=0$). For even orbits clearly $\Lorb=\Rich(P)$.
A particular case is given by \emph{distinguished} orbits, where the connected stabilizer $(G'_e)^{\circ}$, $e\in\Lorb$ is modulo the center of $G'$ a unipotent group.
In this case the {\PVS} $V$ is distinguished (and the converse holds for even orbits $\Lorb$).
In general, not every standard parabolic subgroup is the canonical parabolic subgroup of some nilpotent orbit, let alone an even one.
(By \cite{MR575790}*{Proposition 2.1} and  \cite{MR1125214}*{Theorem 2.4},
the Richardson orbit of a standard parabolic subgroup $P$ has $P$ as its canonical parabolic subgroup if and only if the 
associated {\PVS} $V=\Liealg{N}/[\Liealg{N},\Liealg{N}]$ is regular and $\delta_{\Liealg{N}^{(i)}/\Liealg{N}^{(i+1)}} \in \reducd{\aaa}_G^*$ for all $i \ge 1$,
where $\Liealg{N}=\Liealg{N}^{(0)}\supset\Liealg{N}^{(1)}\supset\dots$ denotes the descending central series of $\Liealg{N}$.)

Any {\PVS} of DK-type arises from an even nilpotent orbit for the reductive Lie algebra $\oplus_i\Liealg{G}'_{2i}$. 
Hence, any {\PVS} of DK-type is also of R-type, although not necessarily with the same $G'$.

\subsection{}
Let us give some concrete examples of DK-type arising from even nilpotent orbits.
For each case we provide the following data:
\begin{itemize}
\item The numbering in the Sato--Kimura classification (appendix of \cite{MR1944442}) in the irreducible case,
\item the group $G'$,
\item the (even) nilpotent orbit $\Lorb$ in $\Liealg{G}'$ (either the corresponding partition for classical groups
or the notation in the Bala--Carter classification),
\item the canonical Levi $G$,
\item the regular {\PVS} $V=\Liealg{n}/[\Liealg{n},\Liealg{n}]$,
\item the representation $\rho$ of $G$ on $V$,
\item the fundamental relative invariant polynomials (FRIPs),
\item the generic stabilizer $H$.
\item If $V$ is {\BASIC}, distinguished, or {\nsr}, we indicate that.
\end{itemize}

For simplicity we only consider split $G'$.

In the following $m,n,k$ are positive integers (possibly with some restrictions).
We denote by $\Mat_{m,n}$ the space of $m\times n$ matrices and by $\Mat_n$ (resp., $\Sym_n$, $\Skew_n$) the spaces of square
(resp., symmetric, skew-symmetric) $n\times n$-matrices. We denote by $\Pf$ the Pfaffian of a skew-symmetric matrix.

\begin{example} \ \label{ex: GL} (\#1 for $k=2$)
\begin{itemize}
\item $G'=\GL_{n_1+\dots+n_k}$, $k>1$, $n_i=n_{k+1-i}$ for all $i$, $n_i\le n_{i+1}$ for $i\le l:=\lfloor\frac{k-1}{2}\rfloor$.
\item $\Lorb$ corresponds to the partition
\[
(\overbrace{k,\dots,k}^{n_1},\overbrace{k-2,\dots,k-2}^{n_2-n_1},\dots,\overbrace{k-2l}^{n_{l+1}-n_l}).
\]
\item $G=\GL_{n_1}\times\dots\times\GL_{n_k}$.
\item $V=\Mat_{n_1,n_2}\oplus\dots\oplus\Mat_{n_{k-1},n_k}$.
\item $\rho(g_1,\dots,g_k)(x_1,\dots,x_{k-1})=(g_1x_1g_2^{-1},\dots,g_{k-1}x_{k-1}g_k^{-1})$.
\item FRIPs: $\begin{cases}\det x_i&\text{if }n_i=n_{i+1}\\\det (x_i\dots x_{k-i})&\text{if }n_i<n_{i+1}\end{cases}$
\item $H=\GL_{n_1}\times\GL_{n_2-n_1}\times\dots\times\GL_{n_{l+1}-n_l}$ embedded as
\[
(g_1,\dots,g_{l+1})\mapsto (g_1,\diag(g_1,g_2),\dots,\diag(g_1,\dots,g_{l+1}),\dots,\diag(g_1,g_2),g_1).
\]
\item Identifying $X^*(G)$ with $\Z^k$ (and consequently $\aaa_G^*$ with $\R^k$) by $(\lambda_1,\dots,\lambda_k)\mapsto\prod_{i=1}^k\det g_i^{\lambda_i}$,
the subspace $\reducd{\aaa}_G^*$ is given by the equations $\lambda_{i+1}+\dots+\lambda_{k-i}=0$ whenever $n_i<n_{i+1}$ or $i=0$.
\item $V$ is {\BASIC} if and only if $n_1=\dots=n_k$. $V$ is distinguished if and only if $n_1=\dots=n_k = 1$.
\end{itemize}
\end{example}

\begin{example} \ \label{ex: Spnodd} (\#13 for $k=1$, up to castling)
\begin{itemize}
\item $G'=\Sp_{2(n_1+\dots+n_k)+n_{k+1}}$, $n_1\le\dots\le n_{k+1}$ even. Set $m_i=n_i-n_{i-1}$ (with $n_0=0$).
\item $\Lorb$ corresponds to the partition
\[
(\overbrace{2k+1,\dots,2k+1}^{m_1},\overbrace{2k-1,\dots,2k-1}^{m_2},\dots,\dots,\overbrace{1,\dots,1}^{m_{k+1}}).
\]
\item $G=\GL_{n_1}\times\dots\times\GL_{n_k}\times\Sp_{n_{k+1}}$.
\item $V=\Mat_{n_1,n_2}\oplus\dots\oplus\Mat_{n_k,n_{k+1}}$.
\item $\rho(g_1,\dots,g_{k+1})(x_1,\dots,x_k)=(g_1x_1g_2^{-1},\dots,g_kx_kg_{k+1}^{-1})$.
\item FRIPs: $\begin{cases}\det x_i&\text{if }n_i=n_{i+1},\\
\Pf(x_i\dots x_{k+1}J_{n_{k+1}}x_{k+1}^t\dots x_i^t)&\text{otherwise},\end{cases}$, $i=1,\dots,k+1$
where $J_{n_{k+1}}\in\Skew_{n_{k+1}}$ defines $\Sp_{n_{k+1}}$.
\item $H=\Sp_{m_1}\times\dots\times\Sp_{m_{k+1}}$ embedded as
\[
(h_1,\dots,h_{k+1})\mapsto (h_1,\iota_2(h_1,h_2),\dots,\iota_{k+1}(h_1,\dots,h_{k+1}))
\]
where $\iota_j$ is an embedding $\Sp_{m_1}\times\dots\times\Sp_{m_j}\hookrightarrow\Sp_{n_j}$.
\item $V$ is {\BASIC} if and only if $n_1=\dots=n_k$.
\end{itemize}
\end{example}

\begin{example} \ \label{ex: Speven} (\#2 for $k=1$)
\begin{itemize}
\item $G'=\Sp_{2(n_1+\dots+n_k)}$, $n_1\le\dots\le n_k$. Set $m_i=n_i-n_{i-1}$ (with $n_0=0$).
\item $\Lorb$ corresponds to the partition
\[
(\overbrace{2k,\dots,2k}^{m_1},\overbrace{2k-2,\dots,2k-2}^{m_2},\dots,\dots,\overbrace{2,\dots,2}^{m_k}).
\]
\item $G=\GL_{n_1}\times\dots\times\GL_{n_k}$.
\item $V=\Mat_{n_1,n_2}\oplus\dots\oplus\Mat_{n_{k-1},n_k}\oplus\Sym_{n_k}$.
\item $\rho(g_1,\dots,g_{k+1})(x_1,\dots,x_k)=(g_1x_1g_2^{-1},\dots,g_{k-1}x_{k-1}g_k^{-1},g_kx_kg_k^t)$.
\item FRIPs: $\begin{cases}\det x_i&\text{if } n_i=n_{i+1}\text{ or }i=k,\\
\det x_i\cdots x_kx_{k-1}^t\cdots x_i^t&\text{otherwise},\end{cases}$, $i=1,\dots,k$.
\item $H=\Orth_{m_1}\times\dots\times\Orth_{m_k}$ embedded as
\[
(h_1,\dots,h_k)\mapsto (h_1,\iota_2(h_1,h_2),\dots,\iota_k(h_1,\dots,h_k))
\]
where $\iota_j$ is an embedding $\Orth_{m_1}\times\dots\times\Orth_{m_j}\hookrightarrow\Orth_{n_j}$.
\item $V$ is {\BASIC} if and only if $n_1=\dots=n_k\ne2$. $V$ is {\nsr} if and only if $m_i=2$ for some $i$.
$V$ is distinguished if and only if $m_i=0$ or $1$ for all $i$.
\end{itemize}
\end{example}

\begin{example} \ \label{ex: SOodd} (\#15 for $k=1$, up to castling)
\begin{itemize}
\item $G'=\SO_{2(n_1+\dots+n_k)+n_{k+1}}$, $n_1\le\dots\le n_{k+1}$. Set $m_i=n_i-n_{i-1}$ (with $n_0=0$).
\item $\Lorb$ corresponds to the partition
\[
(\overbrace{2k+1,\dots,2k+1}^{m_1},\overbrace{2k-1,\dots,2k-1}^{m_2},\dots,\dots,\overbrace{1,\dots,1}^{m_{k+1}}).
\]
\item $G=\GL_{n_1}\times\dots\times\GL_{n_k}\times\SO_{n_{k+1}}$.
\item $V=\Mat_{n_1,n_2}\oplus\dots\oplus\Mat_{n_k,n_{k+1}}$.
\item $\rho(g_1,\dots,g_{k+1})(x_1,\dots,x_k)=(g_1x_1g_2^{-1},\dots,g_kx_kg_{k+1}^{-1})$.
\item FRIPs: $\begin{cases}\det x_i&\text{if }n_i=n_{i+1},\\
\det(x_i\dots x_{k+1}J_{n_{k+1}}x_{k+1}^t\dots x_i^t)&\text{otherwise},\end{cases}$, $i=1,\dots,k+1$
where $J_{n_{k+1}}\in\Sym_{n_{k+1}}$ defines $\SO_{n_{k+1}}$.
\item $H=S(\Orth_{m_1}\times\dots\times\Orth_{m_{k+1}})$ embedded as
\[
(h_1,\dots,h_{k+1})\mapsto (h_1,\iota_2(h_1,h_2),\dots,\iota_{k+1}(h_1,\dots,h_{k+1}))
\]
where $\iota_j$ is an embedding $\Orth_{m_1}\times\dots\times\Orth_{m_j}\hookrightarrow\Orth_{n_j}$.
\item $V$ is {\BASIC} if and only if $n_1=\dots=n_k\ne2$ and $m_{k+1}\ne2$. $V$ is {\nsr} if and only if
$m_i=2$ for some $i$. $V$ is distinguished if and only if $m_i=0$ or $1$ for all $i$.
\end{itemize}
\end{example}

\begin{example} \ \label{ex: SOeven} (\#3 for $k=1$)
\begin{itemize}
\item $G'=\SO_{2(n_1+\dots+n_k)}$, $n_1\le\dots\le n_k$ even. Set $m_i=n_i-n_{i-1}$ (with $n_0=0$).
\item $\Lorb$ corresponds to the partition
\[
(\overbrace{2k,\dots,2k}^{m_1},\overbrace{2k-2,\dots,2k-2}^{m_2},\dots,\dots,\overbrace{2,\dots,2}^{m_k}).
\]
\item $G=\GL_{n_1}\times\dots\times\GL_{n_k}$.
\item $V=\Mat_{n_1,n_2}\oplus\dots\oplus\Mat_{n_{k-1},n_k}\oplus\Skew_{n_k}$.
\item $\rho(g_1,\dots,g_{k+1})(x_1,\dots,x_k)=(g_1x_1g_2^{-1},\dots,g_{k-1}x_{k-1}g_k^{-1},g_kx_kg_k^t)$.
\item FRIPs: $\begin{cases}\det x_i&\text{if } n_i=n_{i+1},\\
\Pf(x_i\cdots x_kx_{k-1}^t\cdots x_i^t)&\text{otherwise},\end{cases}$, $i=1,\dots,k$.
\item $H=\Sp_{m_1}\times\dots\times\Sp_{m_k}$ embedded as
\[
(h_1,\dots,h_k)\mapsto (h_1,\iota_2(h_1,h_2),\dots,\iota_k(h_1,\dots,h_k))
\]
where $\iota_j$ is an embedding $\Sp_{m_1}\times\dots\times\Sp_{m_j}\hookrightarrow\Sp_{n_j}$.
\item $V$ is {\BASIC} if and only if $n_1=\dots=n_k$.
\end{itemize}
\end{example}

\begin{example} \ \label{ex: G2} (\#4)
\begin{itemize}
\item $G'=G_2$.
\item $\Lorb$ is the subregular (10-dimensional) orbit.
\item $G=\GL_2$, with $\Delta_0^G$ consisting of the short simple root.
\item $V$ is the space of binary cubic forms
\item $\rho$ is the symmetric cube representation twisted by $\det^{-1}$.
\item FRIP: the discriminant.
\item $H$ is (an $F$-form of) the symmetric group $S_3$ (with its faithful two-dimensional representation).
\item $V$ is distinguished.
\end{itemize}
\end{example}

\begin{example} \ \label{ex: F4} (\#8)
\begin{itemize}
\item $G'=F_4$, $\Delta_0=\{\alpha_1,\alpha_2,\alpha_3,\alpha_4\}$, $\alpha_1,\alpha_2$ long, $\alpha_3,\alpha_4$ short.
\item $\Lorb=F_4(a_3)$ (40-dimensional)
\item $G=(\GL_2\times\GL_3)/\{(\lambda^2 I_2,\lambda^{-1}I_3)\}$, $\Delta_0^G=\{\alpha_1,\alpha_3,\alpha_4\}$.
\item $V=\Hom_F(F^2,\Sym_3)$.
\item $\rho(g_1,g_2)(A)(\xi)=g_2A(\xi g_1)g_2^t$
\item FRIP: the discriminant of the binary cubic form $\xi\mapsto\det A(\xi)$.
\item $H$ is (an $F$-form of) of the symmetric group $S_4$, embedded in $G$ by its faithful three-dimensional representation on the $\GL_3$-factor 
and by its irreducible two-dimensional representation (which factors through the surjection $S_4 \to S_3$) on the $\GL_2$-factor.
\item $V$ is distinguished.
\end{itemize}
\end{example}

\begin{example} \ \label{ex: E6} (\#12)
\begin{itemize}
\item $G'=E_6$.
\item $\Lorb=D_4(a_1)$ (58-dimensional)
\item $G=\SL_3\times\SL_3\times\GL_2$.
\item $V=\Hom_F(F^2,\Mat_3)$.
\item $\rho(g_1,g_2,g_3)(A)(\xi)=g_1A(\xi g_3)g_2^{-1}$, $\xi\in F^2$.
\item FRIP: the discriminant of the binary cubic form $\xi\mapsto\det A(\xi)$.
\item $H$ is the normalizer of a maximal torus in $\SL_3$, embedded diagonally into the $\SL_3$-factors and mapped to the $\GL_2$-factor via the faithful
two-dimensional representation of its component group $S_3$.
\item $V$ is {\nsr}.
\end{itemize}
\end{example}

Other examples of {\BASIC} {\PVS}s of DK-type are summarized in Table \ref{tab: excbasic}
which refers to the numbering in the list in the appendix of \cite{MR1944442} whenever $\rho$ is reduced. (See also \cite{MR910424}.)
\begin{center}
\begin{table}
\caption{Irreducible {\BASIC} {\PVS}s of DK-type for exceptional groups.}  \label{tab: excbasic}
\begin{tabular}{ |c|c|c|c|c|c| } 
 \hline
\# & $G'$  & $G$ & $\rho$ & $H^\circ$ & $H^\circ\hookrightarrow G$ \\
 \hline
16 & $F_4$ & $\GSpin_7$ & $\spin(8)$ & $G_2$ & $\phi$\\
6  & $E_7$ & $\GL_7$ & $\wedge^3$ & $G_2$ & $\pi$ \\ 
10 & $E_7$ & $\SL_5\times\GL_3$ & $\wedge^2\otimes\Std$ & $\SL_2$ & $\Sym^4\times\Sym^2$ \\ 
27 & $E_7$ & $\GE_6$ & $27$-dim. & $F_4$ & $\sigma$\\
20 & $E_7$ & $\Spin_{10}\times\GL_2$ & half-spin(16) $\otimes\Std$ & $G_2\times\SL_2$ & $\iota_1$ \\
   & $E_7$ & $\GL_2\times\SL_3\times\SL_4$ & $\Std\otimes\Std\otimes\Std$ & $\SL_2$ & $\id\times\Sym^2\times\Sym^3$ \\
7  & $E_8$ & $\GL_8$ & $\wedge^3$ & $\PGL_3$ & adjoint rep. \\ 
21 & $E_8$ & $\Spin_{10}\times\GL_3$ & half-spin(16) $\otimes\Std$ & $\SO_4$ & $\iota_2$ \\
24 & $E_8$ & $\GSpin_{14}$ & half-spin(64) & $G_2\times G_2$ & $\iota_3$ \\
 \hline
\end{tabular}
\end{table}
\end{center}

Here, $\pi$ is the embedding $G_2\hookrightarrow\SO_7\hookrightarrow\GL_7$ and $\phi$ is its lifting to an embedding $G_2\hookrightarrow\Spin_7$.
Also, $\sigma$ is the embedding of $F_4$ as the fixed point subgroup of the outer involution of $E_6$. Finally,
\begin{align*}
\iota_1(x,y)&=(\psi_{7,3}(\phi(x),y),y),\\
\iota_2(x,y)&=(\psi_{6,4}(\Sym^3(y),(x,y)),\Sym^2(y)),\\
\iota_3(x,y)&=\psi_{7,7}(\phi(x),\phi(y)),
\end{align*}
where $\psi_{m,n}$ is the homomorphism $\Spin_m\times\Spin_n\rightarrow\Spin_{m+n}$ and we identify $\Spin_3\simeq\SL_2$,
$\Spin_6\simeq\SL_4$, $\Spin_4\simeq\SL_2\times\SL_2$, $\SO_4\simeq(\SL_2\times\SL_2)/\{\pm1\}$.

\subsection{} \label{sec: magic square}
Another interesting set of examples of irreducible {\PVS}s of DK-type arises from the Freudenthal-Tits magic square
(see Table \ref{tab: FTMS}). 
In these cases, $G$ is a maximal Levi subgroup of $G'$ corresponding to a simple root $\alpha$ which is a leaf in the Dynkin diagram
(it is the unique simple root which is not orthogonal to the highest root $\tilde\alpha$ of $G'$).
The highest weight $\mu$ of $\rho$ is $\tilde\alpha-\alpha$ and its restriction to $G^{\der}$ is the fundamental weight corresponding
to the simple root $\beta$ adjacent to $\alpha$ (and of the same length) in the Dynkin diagram. 
Let $Q=LN'$ be the maximal parabolic subgroup of $G^{\der}$ corresponding to $\beta$.
Then, there exists a non-trivial element $w_0$ in the normalizer of $L$ in $G^{\der}$.
We have $\alpha=w_0\mu$, which is the lowest weight of $V$.
Clearly, $L$ stabilizes the root spaces $V_\alpha$ and $V_\mu$.
Remarkably, any $v\in V_\alpha+V_\mu$ outside $V_\alpha\cup V_\mu$ is regular. We have $H^\circ=L^{\der}$ 
and $[H:H^\circ]=2$ with $H$ containing a representative of $w_0$. In particular, $V$ is not {\BASIC}.
Since $V\simeq\Liealg{G}/\Liealg{H}$ as a representation of $H$, we see that the restriction of $\rho$ to $H$ is
$\Ind_{H^\circ}^H(\rho'\oplus 1)$ where $\rho'$ is the representation of $L$ on $\Liealg{N'}$. It turns out that $N'$ is abelian.
Therefore, $(L,\Liealg{N'},\rho')$ is a {\PVS} and in fact it is {\BASIC}: the generic stabilizer
(whose identity component is denoted by $K$ in the table) is the centralizer of (a suitable choice of) $w_0$ in $L$,
i.e., the stabilizer of an outer involution of $L$ corresponding to $w_0$. See drawings below.

The fundamental relative invariant polynomials in these cases (which are of degree 4) are called Freudenthal quartics (see \cites{MR63358, MR0447140, MR1928799}).
They arise by considering $\operatorname{ad} (x)^4$, $x \in V \subset \Liealg{G}'$ as a map between the one-dimensional spaces $\Liealg{G}'_{-\tilde\alpha}$ and 
$\Liealg{G}'_{\tilde\alpha}$ (cf. \cite{MR1125214}*{Table II}).
The orbit structure of $V$ is described in \cite{MR1247502}.

What is more, the Richardson orbit with respect to the maximal parabolic subgroup corresponding to $\beta$ is also even.
These orbits give rise to the {\PVS}s \#8 (Example \ref{ex: F4}), \#12 (Example \ref{ex: E6}), \#9 and \#28 respectively in the list of the appendix of \cite{MR1944442}.
They all have a fundamental relative invariant polynomial of degree 12, obtained as the determinant of $\operatorname{ad} (x)^6: \Liealg{G}'_{-6} \to \Liealg{G}'_{6}$, $x \in V$, where
$\Liealg{G}'_{6} = \Liealg{G}'_{\tilde\alpha-\alpha} \oplus \Liealg{G}'_{\tilde\alpha}$ is two-dimensional.
Alternatively, the {\PVS} $(L,\Liealg{N'},\rho')$ mentioned above has a cubic fundamental relative invariant polynomial, and the {\PVS} $V$ associated to $\beta$
is the direct sum of $\Liealg{N'}$ and of its isomorphic image under the simple reflection $w_\alpha$, which enables one to construct the fundamental relative invariant of $V$
as the discriminant of the cubic form associated to pairs of elements of $\Liealg{N'}$.

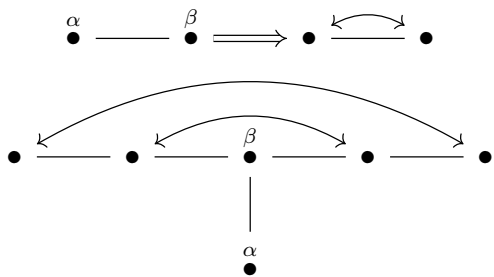
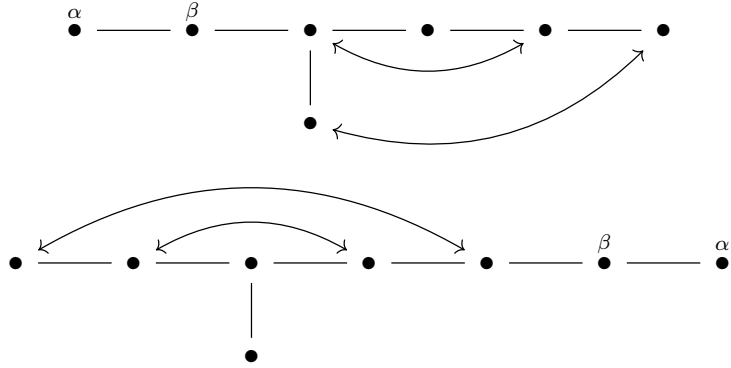
\begin{center}
\begin{table}
\caption{{\PVS}s of DK-type pertaining to the Freudenthal-Tits magic square.}  \label{tab: FTMS}
\begin{tabular}{|c||c|c|c|c|c|c|}
\hline 
\# & $G'$  & $G$           & $L$          & $K$      & $\rho$           & $\rho'$ \\
\hline
14 & $F_4$ & $\GSp_6$      & $\GL_3$            & $\PGL_2$  & 14-dim.          & $\Sym^2$ \\
5  & $E_6$ & $\GL_6$       & $(\GL_3\times\GL_3)\cap\SL_6$ & $\SL_3$  & $\wedge^3$       & $\Std\otimes\Std$ \\
23 & $E_7$ & $\GSpin_{12}$ & $\GL_6$            & $\Sp_6$  & half-spin(32)    & $\wedge^2$ \\
29 & $E_8$ & $\GE_7$        & $\GE_6$              & $F_4$    & 56-dim.          & 27-dim. \\
\hline
\end{tabular}
\end{table}
\end{center}

\begin{center}
\begin{tikzcd}
\overset{\alpha}{\bullet} \arrow[r,dash] & \overset{\beta}{\bullet} \arrow[r,Rightarrow] & \bullet \arrow[dash,r] \arrow[r,leftrightarrow,bend left] & \bullet
\end{tikzcd}

\begin{tikzcd}
\bullet \arrow[r,dash] \arrow[rrrr,leftrightarrow,bend left] & \bullet \arrow[r,dash] \arrow[rr,leftrightarrow,bend left] &
\overset{\beta}{\bullet} \arrow[r,dash] \arrow[d,dash] & \bullet \arrow[dash,r]  & \bullet \\
& & \overset{\alpha}{\bullet}
\end{tikzcd}

\begin{tikzcd}
\overset{\alpha}{\bullet} \arrow[r,dash] & \overset{\beta}{\bullet} \arrow[r,dash] & \bullet \arrow[r,dash] \arrow[d,dash] \arrow[rr,leftrightarrow,bend right] &
\bullet \arrow[r,dash] & \bullet \arrow[r,dash] & \bullet\\
& & \bullet \arrow[urrr,leftrightarrow,bend right]
\end{tikzcd}

\begin{tikzcd}
\bullet \arrow[r,dash] \arrow[rrrr,leftrightarrow,bend left] & \bullet \arrow[r,dash] \arrow[rr,leftrightarrow,bend left] &
\bullet \arrow[r,dash] \arrow[d,dash] & \bullet \arrow[r,dash] & \bullet \arrow[r,dash] & \overset{\beta}{\bullet} \arrow[r,dash] & \overset{\alpha}{\bullet}\\
& & \bullet
\end{tikzcd}
\end{center}

Sato and Kimura classified the irreducible {\PVS}s over $\bar F$ up to a suitable notion of equivalence weaker than isomorphism \cite{MR430336}.

There are 29 essential cases of regular {\PVS}s, listed in the appendix of \cite{MR1944442}.\footnote{The general case is obtained from these by castling transforms, see \cite{MR1944442}*{\S7.1}.
These preserve irreducibility and the generic stabilizer but may yield {\PVS}s which are far removed from nilpotent orbits.}
All these examples are either of DK-type or restrictions thereof to suitable subgroups.\footnote{The examples not of DK-type are \# 17, \# 18, \# 19, \# 22, \# 25 and \# 26.
Example \# 1 can be of DK-type or not.}
Most of these examples are {\sr}. The exceptions are:
\begin{itemize}
\item The symmetric square representation of $\GL_2$ (Example \ref{ex: Speven} for $k=1$, $n_1=2$).
\item Example \ref{ex: E6}.
\item The representation $\GL_2\times\SO_{k+2}$ on $\Mat_{2,k+2}$ (Example \ref{ex: SOodd} for $k=1$, $n_1=2$).
\item The restriction of the latter to $\GL_2\times\Spin_7$ (resp., $\GL_2\times G_2$) for $k=6$ (resp., $k=5$)
via the Spin representation $\Spin_7\hookrightarrow\SO_8$ (resp., the $7$-dimensional representation
$G_2\hookrightarrow\SO_7$).
\end{itemize}

The non-regular irreducible {\PVS}s over $\bar F$ are (up to castling equivalence) mostly of R-type or restrictions thereof.
The notable exception is the representation $\rho$ of $\SL_2\times\SL_n$ on $\Hom_F(F^2,\Skew_n)$ given by $\rho(g_1,g_2)A(\xi)=g_2A(\xi g_1)g_2^t$,
which is a {\PVS} for odd $n$ but not of R-type for $n>7$.
Another curious example is the restriction of the representation of $\GL_3\times\Sp_{2n}$ (of R-type) on $\Mat_{3,2n}$, $n>1$ given by $\rho(g_1,g_2)A=g_1Ag_2^{-1}$
to the image of the symmetric square representation of $\GL_2$ times $\Sp_{2n}$. This irreducible, non-regular {\PVS} has a relative invariant.

\section{Statement of main result} \label{sec: statement}
From now on assume that $F$ is a number field.
Let $\A = F_\infty \times \A_{\fin}$ be the ring of adeles of $F$.
For any connected $F$-group $G$ we write
\[
G(\A)^1=\cap_{\chi\in X^*(G)}\Ker\abs{\chi}
\]
where $\abs{\chi}:G(\A)\rightarrow\R_{>0}$ is the composition $G(\A)\xrightarrow{\chi}\A^*\xrightarrow{\abs{\cdot}}\R_{>0}$.
The automorphic space $G(F)\bs G(\A)^1$ is of finite volume.

Let
\[
\aaa_{G,\C}^*=\aaa_G^*\otimes_{\R}\C=X^*(G)\otimes_{\Z}\C.
\]
Every $\lambda\in\aaa_{G,\C}^*$ defines a (topological) quasi-character $g\mapsto g^\lambda$ of $G(\A)$ by
\[
g^{\sum c_i\chi_i}=\prod_i\abs{\chi_i(g)}^{c_i},\ \ \chi_i\in X^*(G),\ c_i\in\C.
\]
All quasi-characters of $G(\A)/G(\A)^1$ are obtained this way. In particular, we get an isomorphism of topological groups
\[
\iii\aaa_G^*\xrightarrow{\Exp_G}\PD(G(\A)/G(\A)^1),\ \ \lambda\mapsto (g\mapsto g^\lambda)
\]
where $\PD$ denotes the Pontryagin dual.

The following is a supplement to Lemma \ref{lem: genHG}.

\begin{lemma} \label{lem: finvolcrit}
Suppose that $G$ is connected and $H$ is an $F$-subgroup of $G$. Write $H(\A)^{(1)}=H(\A)\cap G(\A)^1$.
Then, the following two conditions are equivalent.\footnote{Note that both conditions imply that $H(\A)\cap G(\A)^1$ is unimodular.}
\begin{enumerate}
\item The restriction map $X^*(G)\xrightarrow{r} X^*(H^\circ)$ has finite cokernel.
\item The volume of the quotient $H(F)\bs H(\A)^{(1)}$ is finite.
\end{enumerate}
\end{lemma}

\begin{proof}
We may assume without loss of generality that $H$ is connected, since $[H(F):H^\circ(F)]$ is finite and
$H(\A)/H^\circ(\A)$ is compact. Now, $\vol(H(F)\bs H(\A)^1)<\infty$. Therefore,  $\vol(H(F)\bs H(\A)^{(1)})<\infty$
if and only if $Y=H(\A)^{(1)}/H(\A)^1$ is compact. Note that $Y$ is the kernel of the homomorphism
\[
H(\A)/H(\A)^1\rightarrow G(\A)/G(\A)^1.
\]
Passing to the Pontryagin dual, the lemma follows from the commutative diagram
\[
  \begin{tikzcd}
    \iii\aaa_G^* = \iii (X^*(G)\otimes\R) \arrow{d}{r\otimes\id}  \arrow[r,"\Exp_G"] &  PD(G(\A)/G(\A)^1) \arrow{d}\\
    \iii\aaa_H^* = \iii (X^*(H)\otimes\R)  \arrow[r,"\Exp_H"] &  PD(H(\A)/H(\A)^1)  
  \end{tikzcd}
\]
In fact, $Y$ is compact if and only if $Y=1$ if and only if the restriction map $\aaa_G^*\rightarrow\aaa_H^*$ is surjective.
\end{proof}

Henceforth, $(G,V,\rho)$ is a reductive {\PVS} (i.e., $G$ is reductive).
We apply the criterion of Lemma \ref{lem: finvolcrit} to the stabilizer $G_v$ of a vector $v \in \Vreg (F)$. Recall Definition \ref{def: sr}.

\begin{lemma} \label{lem: finitevolume} Let $v \in \Vreg (F)$. Then,
\begin{enumerate} \item
The group $G_v(\A)^{(1)}=G_v(\A)\cap G(\A)^1$ is unimodular. 
\item The volume
$\vlum_v=\vol(G_v(F)\bs G_v(\A)^{(1)})$ is finite if and only if $v$ is {\sr}. 
\end{enumerate} 
\end{lemma}

\begin{proof}
The first assertion follows from the fact that modular character of $G_v$ is the restriction of a character of $G$ by \eqref{eq: mdchr}.
The second assertion follows from Lemma \ref{lem: finvolcrit} and Lemma \ref{lem: isot}.
\end{proof}

We say that a continuous function $\phi$ on $V (\A)$ is rapidly decreasing, if it is supported in $V (F_\infty) K$ for some 
compact subgroup $K$ of $V(\A_{\fin})$ and 
$\sup_{v \in V (\A)}\abs{f (v)}(1+\norm{v_\infty})^N < \infty$ for every nonnegative integer $N$, where $\norm{\cdot}$ denotes a norm on the real vector space $V(F_\infty)$.
Let $C_{\rd}(V(\A))$ be the space of continuous rapidly decreasing functions on $V(\A)$.
It is a (strict) LF-space, i.e., the strict (locally convex) inductive limit of Fr\'echet spaces.

For any $\phi\in C_{\rd}(V(\A))$ we form the theta function
\[
\theta_\phi(g)=\theta_\phi^V(g)=\sum_{v\in \Vreg(F)}\phi(\rho(g)^{-1}v),\ \ g\in G(\A).
\]
The sum is locally uniformly absolutely convergent and defines a continuous function on $G(F)\bs G(\A)$.
We would like to study this function and the corresponding zeta function which is roughly speaking the Mellin transform of $\theta_\phi$.

Let $\reducd{\aaa}_{G,\C}^* = \reducd{\aaa}_G^* \otimes\C=X^*(G/\centder{G})\otimes\C$.
Ideally, we would like to study the double integral
\begin{equation} \label{eq: naivezeta}
\int_{G(\A)^1\centder{G}(\A)\bs G(\A)}\int_{G(F)\bs G(\A)^1}\theta_\phi(ga)\ dg\ a^{-\lambda}\abs{\det\rho(a)}^{-1}\ da
\end{equation}
for $\lambda$ in a suitable region of convergence in $\reducd{\aaa}_{G,\C}^*$.
Note that the outer integral formally makes sense, since the inner integral can be written as the sum over the $G(F)$-orbits in $\Vreg(F)$,
and the contribution of an orbit $\rho(G(F))v$ is
\[
\int_{G_v(F)\bs G(\A)^1}\phi(\rho(ga)^{-1}v)\ dg=
\vlum_v\int_{(G_v(\A)\cap G(\A)^1)\bs G(\A)^1}\phi(\rho(ga)^{-1}v)\ dg
\]
where $\vlum_v=\vol(G_v(F)\bs (G_v(\A)\cap G(\A)^1))$ as above.
Note that the images of $G_v(\A)$ and $\centder{G}(\A)$ in $G(\A)^1\bs G(\A)=\aaa_G=X_*(G^{\ab})\otimes\R$ coincide -- they are both equal to the image of
$X_*(\centder{G}/G^{\der})\otimes\R$ in $\aaa_G$.
Therefore, we can write the above as
\[
\vlum_v\int_{G_v(\A)\bs G(\A)^1\centder{G}(\A)}\phi(\rho(ga)^{-1}v)\abs{\det\rho(g)}^{-1}\ dg.
\]
This also explains why the more straightforward integral
$\int_{G(F)\bs G(\A)}\theta_\phi(g)g^{-\lambda}\ dg$ would not converge (for any $\lambda$) if $\reducd{\aaa}_G^* \neq \aaa_G^*$
(i.e., if $\rk X^*(H)>0$ by Lemma \ref{lem: restggv}).

A necessary condition for the convergence of this orbital integral for all $\phi$ is that $\vlum_v$ is finite, which by Lemma \ref{lem: finitevolume} is equivalent to $v$ being {\sr}.
Define
\[
\theta_\phi^{\sreg}(g)=\sum_{v\in \Vreg(F)^{\sreg}}\phi(\rho(g)^{-1}v),\ \ g\in G(\A).
\]
Thus, $\theta_\phi^{\sreg}$ omits the $G(F)$-orbits in $\Vreg(F)$ whose contributions to \eqref{eq: naivezeta} clearly diverge.

Our main result, which will be ultimately proved in \S\ref{sec: pfmain},  is the following.
\begin{theorem} \label{thm: convergence}
Assume that $(G,V,\rho)$ is a regular \PVS. Then, the double integral
\[
Z(\phi,\lambda)=\int_{G(\A)^1\centder{G}(\A)\bs G(\A)}\int_{G(F) \bs G (\A)^1} \theta_\phi^{\sreg} (ga) \, dg\  a^{-\lambda}\abs{\det\rho(a)}^{-1} \, da
\]
is absolutely convergent and holomorphic for any $\lambda\in\reducd{\aaa}_{G,\C}^*$ such that $\Re\lambda\in\R_{>0}\fundchr$.
Moreover, the integral is continuous in $\phi\in C_{\rd}(V(\A))$.
\end{theorem}

Roughly speaking, Theorem \ref{thm: convergence} means that the only problematic terms in \eqref{eq: naivezeta} are those arising
from the orbits $\rho(G(F))v$ where $\vlum_v=\infty$.

The proof of Theorem \ref{thm: convergence} is based on a direct estimation of the theta function $\theta_\phi^{\sreg}$ using standard lattice-point counting methods.
In fact, this idea has already been used previously by Yukie \cites{MR1267735, MR1369477} to obtain partial results in the direction of Theorem \ref{thm: convergence},
mostly via a case-by-case analysis, for absolutely irreducible representations $\rho$.\footnote{Yukie proved a general qualitative convergence result for distinguished,
absolutely irreducible {\PVS}s.}
Our approach is completely general and avoids any classification result. Also, the possibility to define a zeta function in the case where $\reducd{\aaa}_G^* \neq \aaa_G^*$
does not seem to have been noticed in the literature. 

An alternative classification-free method was developed by F. Sat\B{o} \cite{MR695661} and H. Saito \cite{MR1994885}, among others,
using the decomposition of $\Vreg(F)$ into $G(F)$-orbits (Galois cohomology) and an analysis of the resulting Euler products.
Saito proves Theorem \ref{thm: convergence} under the additional assumption that $H^\circ$ is semisimple and the component group $H/H^\circ$
is abelian. Under the weaker assumption $\reducd{\aaa}_G^* = \aaa_G^*$, i.e., $X^*(H)$ finite, 
he proves convergence for all $\Re\lambda$ in some cone of the form 
$\lambda_0+\R_{>0}\fundchr$.\footnote{Strictly speaking, Saito assumes that there exists $v\in\Vreg(F)$ such that $X^*(G_v^\circ)=1$.
By Lemma \ref{lem: charconn} and Lemma \ref{lem: isot} this implies that $v\in\Vreg(F)^{\sreg}$ and $X^*(H)$ is finite.
Conversely, if $X^*(H)$ is finite, then $\Vreg(F)^{\sreg}=\{v\in\Vreg(F)\mid X^*(G_v^\circ)=1\}$.
Also, he considers a more general situation where $G$ is a connected linear algebraic group and $\Vreg$ the complement of a hypersurface.}

We expect Theorem \ref{thm: convergence} to hold without the regularity assumption on $V$.
In the case where there are no relative invariants (in which $V$ is called \emph{special}), the set $G(\A)\Vreg(F)$ is open in $V(\A)$,
and the integral $\int_{G(F)\bs G(\A)^1}\theta_\phi(g)\ dg$ converges to a constant multiple of $\int_{G(\A)\Vreg(F)}\phi(x)\ dx$.
(See \cites{MR0230726, MR1413009} and the discussion in \cite{MR3675167}*{\S2.4}. The set $G(\A)\Vreg(F)$ is described in \cite{MR1731463}*{Proposition 1.7}.)

\section{Bounding the theta function} \label{sec: thetabnd}
We start with a general bound on $\theta_\phi(g)$ (Proposition \ref{prop: bndtheta} below) which is valid without the regularity assumption,
but is by itself insufficient to prove convergence of the zeta integral of $\theta_\phi^{\sreg}(g)$. 
To state it, we need to set some more notation which will be used throughout the paper.

\subsection{}
Recall that $G$ is a reductive group over a number field $F$.
Fix a maximal $F$-split torus $T_0$ of $G$ and a minimal parabolic $F$-subgroup $P_0$ of $G$ containing $T_0$.
We have a Levi decomposition $P_0=M_0\ltimes N_0$ where $M_0$ is the centralizer of $T_0$ in $G$ and $N_0=\rad{P_0}$,
which is a maximal unipotent $F$-subgroup of $G$.

If $P$ is a standard parabolic subgroup of $G$, then we denote by $M_P$ its standard Levi subgroup.

For simplicity we write $\aaa_0=\aaa_{P_0}=\aaa_{M_0}\simeq\aaa_{T_0}$.
Let
\[
H_0:M_0(\A)\rightarrow\aaa_0
\]
be the group homomorphism defined by
\[
e^{\sprod{\chi}{H_0(m)}}=\abs{\chi(m)},\ \ m\in M_0(\A),\ \chi\in X^*(M_0).
\]

Let $\Delta_0\subset X^*(T_0)$ be the set of simple roots of $T_0$ on $\Lie N_0$.

Fix a maximal compact subgroup $K$ of $G(\A)$ which is in good position with respect to $P_0$.

For any $g\in G(\A)$, we write $m_0(g)\in M_0(\A)/(M_0(\A)\cap K)$ for the $M_0(\A)$-part in the Iwasawa decomposition
$G(\A)=N_0(\A)M_0(\A)K$ of $g$.

We extend $H_0$ to a left-$N_0(\A)$ and right-$K$-invariant function on $G(\A)$.
Thus,
\[
H_0(g)=H_0(m_0(g)).
\]

Fix a Siegel set $\Siegel$ in $G(\A)$ of the form \label{sec: Siegel}
\[
\Siegel=\{g\in G(\A)\mid\sprod{\alpha}{H_0(g)}>c_0\ \forall\alpha\in\Delta_0\}
\]
for some fixed constant $c_0$. Thus,
\begin{enumerate}
\item $\Siegel$ is left $P_0(F)N_0(\A)$-invariant.
\item For any $g\in G(\A)$, the set $\{\gamma\in P_0(F)\bs G(F)\mid \gamma g\in\Siegel\}$ is finite.
\item There exists a compact subset $\Omega\subset G(\A)$ such that
\begin{equation} \label{eq: SiegT}
\Siegel\subset P_0(F)(T_0(\A)\cap\Siegel)\Omega.
\end{equation}
\item $G(F)\Siegel=G(\A)$ provided that $c_0$ is sufficiently negative.
\end{enumerate}

Let $\delta_{P_0}$ be the modulus function of $P_0(\A)$.

\subsection{}
Recall that $V$ is a \PVS. (For now, we do not assume that $V$ is regular.)
The first order of business is the following result, which will be proved in the rest of this section.

\begin{proposition} \label{prop: bndtheta}
There exists a continuous seminorm $\nu$ on $C_{\rd}(V(\A))$ such that
\[
\abs{\theta_\phi(g)}\le\delta_{P_0}(m_0(g))\abs{\det\rho(g)}\nu(\phi)
\]
for all $\phi\in C_{\rd}(V(\A))$ and $g\in \Siegel$.
\end{proposition}

Note that $\abs{\det\rho}$ is the modulus function for the action of $G(\A)$ on $V(\A)$, that is,
\[
\int_{V(\A)}f(\rho(g)^{-1}v)\ dv=\abs{\det\rho(g)}\int_{V(\A)}f(v)\ dv
\]
for any $f\in L^1(V(\A))$ and $g\in G(\A)$.

\subsection{}

Let $\rts_G$ be the set of roots of $T_0$ on $G$ and let $\wgts_V$ be the set of weights of $T_0$ on $V$.
The set $\rts_G$ is a (not necessarily reduced) root system and the co-roots are in $X_*(T_0)$.
We decompose $\Liealg{G}$ and $V$ according to the roots and weights:
\begin{subequations}
\begin{align}
\Liealg{G}&=\oplus_{\alpha\in\rts_G\cup\{0\}}\Liealg{G}_\alpha,\\
V&=\oplus_{\beta\in\wgts_V}V_\beta. \label{eq: Vdecomp}
\end{align}
\end{subequations}
The vector spaces $\Liealg{G}_\alpha$, $\alpha\in\rts_G$, are not necessarily one-dimensional unless $G$ is split.
Moreover, $\Liealg{G}_\alpha$ is not necessarily a commutative Lie subalgebra (namely, if $2\alpha\in\rts_G$).
The spaces $\Liealg{G}_\alpha$ and $V_\beta$ are invariant under $M_0$. We write
\[
n_\beta=\dim V_\beta,\ \ \beta\in\wgts_V,
\]
and
\[
\delta_V=\det\rho\rest_{T_0}=\sum_{\beta\in\wgts_V}n_\beta\beta\in X^*(T_0).
\]
By convention, $V_\gamma=0$ if $\gamma\notin\wgts_V$.
Note that
\begin{equation} \label{eq: drhoab}
d\rho(\Liealg{G}_\alpha)(V_\beta)\subset V_{\alpha+\beta},\ \ \alpha\in\rts_G,\beta\in\wgts_V,
\end{equation}
where $d\rho:\Liealg{G}\rightarrow\Liealg{GL}(V)$ is the Lie algebra representation obtained by differentiating $\rho$.

Let
\[
\delta_0=\sum_{\alpha\in\rts_G^+}(\dim\Liealg{G}_\alpha)\alpha\in X^*(T_0),
\]
where $\rts_G^+\subset\rts_G$ is the subset of positive roots with respect to $P_0$, so that
\[
\delta_{P_0}(m)=e^{\sprod{\delta_0}{H_0(m)}},\ \ m\in M_0(\A).
\]

We enumerate the roots $\rts_G\cup\{0\}$ (including multiplicities) as $\alpha_i$, $i\in\indx_G$ for some index set $\indx_G$, and fix a basis
$\basis_G=(x_i)_{i\in\indx_G}$ for $\Liealg{G}$ such that $x_i\in\Liealg{G}_{\alpha_i}$ for all $i$.
Similarly, we enumerate the weights $\wgts_V$ including multiplicities as $\beta_j$, $j\in\jndx_V$, and fix a basis
$\basis_V=(y_j)_{j\in\jndx_V}$ for $V$ such that $y_j\in V_{\beta_j}$ for all $j$.

Let $\Sub(V)$ be the (finite) set of subspaces of $V$ that are sums of weight subspaces.
Thus, every $U\in\Sub(V)$ is of the form
\[
U=\oplus_{\beta\in\wgts_U}V_\beta,
\]
where $\wgts_U$ is a subset of $\wgts_V$. We will also write
\begin{equation} \label{def: jndxUV}
\jndx_{V/U}=\{j\in\jndx_V\mid\beta_j\notin\wgts_U\}.
\end{equation}

\begin{lemma} \label{lem: suppRI}
Let $\psi\in\Z_{\ge0}\fundchr$, so that $f=\RI_\psi$ is a polynomial. Let $U\in\Sub(V)$ be such that $f\rest_U\not\equiv0$.
Then,
\begin{enumerate}
\item $\psi\in\Z_{\ge0}\wgts_U$. 
\item There exists $U'\in\Sub(V)$, $U'\subset U$ such that $f\rest_{U'}\not\equiv0$ and $\psi\in\R_{>0}\wgts_{U'}$.
\end{enumerate}
\end{lemma}

\begin{proof}
Let $\basis_{V^*}=(y_j^*)_{j\in\jndx_V}$ be the basis of $V^*$ dual to $\basis_V$ and let $\basis_{F[V]}$ be the basis of $F[V]$ consisting of monomials in $\basis_{V^*}$.
By definition, the weight of a monomial $y_{j_1}^*\cdots y_{j_k}^*$ in $\basis_{F[V]}$ is the sum $\beta_{j_1}+\dots+\beta_{j_k}$.
We view $F[U]$ as a subalgebra of $F[V]$ using the unique projection $V\rightarrow U$ with kernel in $\Sub(V)$.
The basis $\basis_{F[V]}$ contains $\basis_{F[U]}$ (defined in the obvious way).

Let $h=f\rest_U$.
Expand $h$ in the basis $\basis_{F[U]}$.
Since $f=\RI_\psi$, the weight of any basis element with nonzero coefficient is $\psi$. The first part follows.

Let $U'$ be the smallest subspace in $\Sub(V)$ such that $h\in F[U']$ (viewed as a subspace of $F[U]$ via the projection $U\rightarrow U'$ with kernel in $\Sub(V)$).
Then, $\wgts_{U'}$ is the union of $\{\beta_{j_1},\dots,\beta_{j_k}\}$ as we range over all the monomials
$y_{j_1}^*\dots y_{j_k}^*$ in $\basis_{F(U)}$ with nonzero coefficients in the expansion of $h$.
Once again, since $f=\RI_\psi$, for any basis element with nonzero coefficient as above we have $\beta_{j_1}+\dots+\beta_{j_k}=\psi$.
Averaging over all these expressions of $\psi$ we deduce that $\psi\in\R_{>0}\wgts_{U'}$.
\end{proof}

Consider the set $\Minset(V)$ of subspaces in $\Sub(V)$ that contain a regular element:
\[
\Minset(V)=\{U\in\Sub(V)\mid U\cap \Vreg\ne\emptyset\}.
\]

\begin{corollary} \label{cor: fundchrS}
Let $U\in\Minset(V)$. Then,
\begin{enumerate}
\item \label{part: nonneg} $\fundchr\subset\Z_{\ge0}\wgts_U$.
\item \label{part: psipos} Assume that $V$ is regular.
Let $\psi\in \Z_{>0}\fundchr$.
Then there exists $U'\in\Minset(V)$, $U'\subset U$ such that $\psi\in\R_{>0}\wgts_{U'}$.
\end{enumerate}
\end{corollary}

\begin{proof}
Let $\psi\in\Z_{\ge0}\fundchr$ and let $f=\RI_\psi$.
Since $U\in\Minset(V)$, the restriction of $f$ to $U$ is not identically zero.
By Lemma \ref{lem: suppRI}, $\psi\in\Z_{\ge0}\wgts_U$, which implies the first part, and there exists $U'\subset U$, $U'\in\Sub(V)$ such that
$f\rest_{U'}\not\equiv0$ and $\psi\in\R_{>0}\wgts_{U'}$.

If $V$ is regular and $\psi\in \Z_{>0}\fundchr$, then $\Vreg=\{v\in V\mid f(v)\ne0\}$. Hence, $U'\cap \Vreg\ne\emptyset$.
The second part follows.
\end{proof}

A connected algebraic $F$-subgroup $R$ of $G$ is called \emph{$T_0$-saturated} if its Lie algebra
is of the form $\oplus_{\alpha\in S}\Liealg{G}_\alpha$ for some subset $S$ of $\rts_G\cup\{0\}$.
In this case we write $\rts_R = S \setminus \{0\}$.
Examples of $T_0$-saturated subgroups are parabolic $F$-subgroups containing $T_0$ and their unipotent radicals, and semistandard Levi subgroups. 
If $R$ is $T_0$-saturated, then we write
\begin{equation} \label{def: indxX}
\indx_R=\{i\in\indx_G\mid \Liealg{G}_{\alpha_i} \subset \Lie R \}.
\end{equation}

For any subspace $U\subset V$ denote by $\Stab(U)=\{g\in G\mid \rho(g)U=U\}$ the stabilizer of $U$. The Lie algebra of $\Stab(U)$ is
$\{x\in\Liealg{G}\mid d\rho(x)U\subset U\}$ \cite{MR1102012}*{Lemma II.7.4}.

Let $U\in\Sub(V)$ and let $R$ be a $T_0$-saturated subgroup of $G$. We denote by $R\star U$ the smallest $R$-stable space in $\Sub(V)$ containing $U$.

The following is an immediate consequence of \eqref{eq: drhoab}.
\begin{lemma} \label{lem: Hclos}
Suppose that $S$ is a subset of $\wgts_V$ containing $\wgts_U$ such that
for all $\beta\in S$ and $\alpha\in\rts_R$ with $\beta+\alpha\in\wgts_V$, we have $\beta+\alpha\in S$.
Then, $\wgts_{R\star U}\subset S$. In particular,
for every $\beta\in\wgts_{R\star U}$ there exist $\gamma\in \wgts_U$ and $\alpha_1,\dots,\alpha_k\in\rts_R$ such that $\beta=\gamma+\alpha_1+\dots+\alpha_k$.
\end{lemma}

\begin{remark}
It will follow from Lemma \ref{lem: sumofrts} below that $\wgts_{R\star U}$ is the smallest subset $S$ of $\wgts_V$ containing $\wgts_U$ with the property that
for every $\beta\in S$ and $\alpha\in\rts_R$ such that $\beta+\alpha\in\wgts_V$, we have $\beta+\alpha\in S$.
\end{remark}

\subsection{}
For any $U\in\Sub(V)$ define the gauge
\begin{equation} \label{def: lambda(U)}
\lambda(U)=\delta_0+\sum_{\beta\in\wgts_V\setminus\wgts_U}n_\beta\beta=\delta_0+\delta_V-\sum_{\beta\in\wgts_U}n_\beta\beta\in X^*(T_0).
\end{equation}
It will turn out to be a key parameter governing the growth of the subsum of $\theta_\phi$ pertaining to $U$.

The following combinatorial-geometric lemma is crucial.

\begin{lemma} \label{lem: suppSbnd}
Let $U\in\Minset(V)$. Then, there exists a one-to-one function
\[
\iota:\jndx_V\rightarrow\indx_G
\]
such that
\[
\beta_j+\alpha_{\iota(j)}\in\wgts_U
\]
for all $j\in\jndx_V$.
\end{lemma}

\begin{proof}
Let $v_0\in U\cap \Vreg$. Recall the surjection $D_{v_0}$ \eqref{def: Dv}.
Since $v_0\in U$, $D_{v_0}\Liealg{G}_\alpha\subset V_{\alpha+\wgts_U}$ for all $\alpha\in\rts_G\cup\{0\}$ by \eqref{eq: drhoab}.
Writing the matrix of $D_{v_0}$ with respect to the bases $\basis_G$ and $\basis_V$, a nonzero entry at a position $(j,i)$, $j\in\jndx_V$, $i\in\indx_G$,
can only occur if $\beta_j - \alpha_i \in\wgts_U$.
A non-zero minor of order $\dim V$ therefore gives rise to a function $\iota$ as required.
\end{proof}

Note that we used that $G$ is reductive in the proof, namely that $-\rts_G=\rts_G$.

\begin{corollary} \label{cor: lam(U)}
Let $U\in\Minset(V)$. Then, for any $J\subset\jndx_V$ we have
\[
\delta_0+\sum_{j\in J}\beta_j\in\R_{\ge0}(\wgts_U\cup\Delta_0).
\]
In particular, we have
\begin{equation} \label{eq: lambdaUincone}
\lambda(U)\in\R_{\ge0}(\wgts_U\cup\Delta_0).
\end{equation}
\end{corollary}

\begin{proof}
We write
\[
\delta_0+\sum_{j\in J}\beta_j=\delta_0-\sum_{j\in J}\alpha_{\iota(j)}+\sum_{j\in J} (\beta_j+\alpha_{\iota(j)})
\]
and observe that $\delta_0-\sum_{j\in J}\alpha_{\iota(j)}\in\R_{\ge0}\Delta_0$ since $\iota$ is injective.

Taking $J=\jndx_{V/U}$ (cf.~\eqref{def: jndxUV} and \eqref{def: lambda(U)}) we obtain \eqref{eq: lambdaUincone}.
\end{proof}

\subsection{}
The proof of Proposition \ref{prop: bndtheta} is based on a standard bound for lattice sums. 

For any $\xi\in V$ denote by $\xi_\beta\in V_\beta$, $\beta\in\wgts_V$ the components of $\xi$ with respect to the decomposition \eqref{eq: Vdecomp}.

We write $x_\pm=\max(0,\pm x)$ for the positive and negative parts of a real number $x$, so that $x=x_+-x_-$.

\label{sec: llNphi}

For simplicity, we write $A(\cdot,\phi)\ll_{N,\phi}^{\rd} B_N(\cdot)$ if for every $N\ge0$
there exists a continuous seminorm $\nu_N$ on $C_{\rd}(V(\A))$ such that $\abs{A(\cdot,\phi)}\le B_N(\cdot)\nu_N(\phi)$ for every $\phi\in C_{\rd}(V(\A))$.

\begin{lemma} \label{lem: Latticepoints}
Let $\Omega$ be a compact subset of $G(\A)$. Let $U\in\Sub(V)$.
Then,
\[
\sum_{\xi\in U: \, \xi_{\beta} \neq 0\, \forall\beta \in\wgts_U} \abs{\phi(\rho(g)^{-1}\xi)} \ll_{N,\phi}^{\rd}
e^{\sum_{\beta\in\wgts_U} (n_\beta\sprod{\beta}{H_0(g)}_+ - N\sprod{\beta}{H_0(g)}_-)},\ \ g \in T_0(\A)\Omega.
\]
Equivalently,
\[
e^{\sprod{\lambda(U)}{H_0(g)}}\abs{\det\rho(g)}^{-1}\sum_{\xi\in U: \, \xi_{\beta} \neq 0\, \forall\beta \in\wgts_U} \abs{\phi(\rho(g)^{-1}\xi)} \ll_{N,\phi}^{\rd}
\delta_{P_0}(m_0(g))e^{-N\sum_{\beta\in\wgts_U}\sprod{\beta}{H_0(g)}_-}.
\]
In particular, for $U=V$
\[
\abs{\det\rho(g)}^{-1}\sum_{\xi\in V: \, \xi_{\beta} \neq 0\, \forall\beta \in\wgts_V} \abs{\phi(\rho(g)^{-1}\xi)} \ll_{N,\phi}^{\rd}
\min(1,\abs{\det\rho(g)})^N.
\]
\end{lemma}

\begin{proof}
The first two statements are clearly equivalent by the definition \eqref{def: lambda(U)} of $\lambda(U)$.
Since $\rho$ acts continuously on $C_{\rd}(V(\A))$ we may assume without loss of generality that $g\in T_0(\A)$.
(In fact, since both sides are left $T_0(F)$-invariant, we may assume that $g\in T_0(\R)$ where $\R$ is viewed as a subring of $\A=\A_{\Q}\otimes F$
via $x\mapsto x\otimes 1$.)

Upon replacing $\phi$ by $\prod_{\beta\in\wgts_V}(1+\norm{v_{\beta,\infty}})^{-n}\one_{K_\beta}(v_{\beta,\fin})$, $n\gg0$
where $K_\beta$ is a compact open subgroup of $V_\beta(\A_{\fin})$ for every $\beta\in\wgts_V$,
we may also assume that $\phi\rest_{U (\A)}$ factors as a product $\phi (v) = \prod_{\beta \in\wgts_U} \phi_\beta (v_\beta)$,
$v \in U (\A)$, where each factor $\phi_\beta \in  C_{\rd}(V_\beta(\A))$ is nonnegative real.
Since $T_0$ acts by scalars on each space $V_\beta$, we are reduced to the assertion (for $t\in\A^\times$)
\[
\sum_{\xi_{\beta} \neq 0} \phi_\beta (t^{-1}\xi_{\beta})\ll_{N,\phi_\beta}^{\rd} \begin{cases}\abs{t}^{n_\beta}&\abs{t} \ge 1,\\
\abs{t}^{N}&\abs{t} \le 1.\end{cases}
\]
In turn, this immediately reduces to the following elementary statement for $\varphi\in C_{\rd}(\A)$:
\[
\sum_{\xi\in F^\times}\varphi(t^{-1}\xi)\ll_{N,\varphi}^{\rd}\begin{cases}\abs{t}&\abs{t}\ge1\\\abs{t}^N&\abs{t}\le1.\end{cases}
\]
To prove this, note that given a compact subset $K$ of $\A$, we have
\[
\#(F^\times\cap tK)\ll\abs{t}.
\]
In particular, $F^\times\cap tK=\emptyset$ if $\abs{t}$ is sufficiently small.
Thus, given a compact subset $K_{\fin}$ of $\A_{\fin}$,
\[
\#\{\xi\in F^\times\mid (t^{-1}\xi)_{\fin}\in K_{\fin}\text{ and }\norm{(t^{-1}\xi)_\infty}\le R\}\ll R^{[F:\Q]}\abs{t},\ \ R\ge0.
\]
The statement above follows easily from this estimate and the fact that $\varphi\in C_{\rd}(\A)$.
\end{proof}

\begin{proof}[Proof of Proposition \ref{prop: bndtheta}]
Define the support of an element $\xi\in V$ to be
\[
\supp\xi=\{\beta\in\wgts_V\mid\xi_\beta\ne0\},
\]
so that for any $U\in\Sub(V)$, $\xi\in U\iff\supp\xi\subset\wgts_U$.

We decompose $\theta_\phi$ according to the support of the summands, i.e., we write
\begin{equation} \label{eq: thetadecomp}
\theta_\phi(g)=\sum_{U\in\Sub(V)}\Theta_U(g)\text{  where }
\Theta_U(g)=\sum_{\xi\in \Vreg(F):\,\supp\xi=\wgts_U}\phi(\rho(g)^{-1}\xi).
\end{equation}
Of course, $\Theta_U$ is no longer $G(F)$-invariant or even $P_0(F)$-invariant, but it is $T_0(F)$-invariant.

At any rate, for the proof of Proposition \ref{prop: bndtheta} we may assume by \eqref{eq: SiegT} that $g\in (T_0(\A)\cap\Siegel)\Omega$, which we will do throughout.

Fix $U\in\Sub(V)$. By Lemma \ref{lem: Latticepoints}, we have
\begin{equation} \label{eq: BoundThetaU}
\Theta_U(g)\ll_{N,\phi}^{\rd} \delta_{P_0}(m_0(g))\abs{\det\rho(g)}e^{-\sprod{\lambda(U)}{H_0(g)} - N \sum_{\beta\in\wgts_U} \sprod{\beta}{H_0(g)}_-}.
\end{equation}
Moreover, we may assume that $U\in\Minset(V)$, for otherwise $\Theta_U(g)=0$. By \eqref{eq: lambdaUincone}, we have
\[
\lambda(U)=\sum_{\beta\in\wgts_U}c_\beta\beta+\sum_{\alpha\in\Delta_0}d_\alpha\alpha
\]
with suitable coefficients $c_\beta,d_\alpha\ge0$. Since $g\in (T_0(\A)\cap\Siegel)\Omega$, we deduce that
\[
\Theta_U(g)\ll_{N,\phi}^{\rd} \delta_{P_0}(m_0(g))\abs{\det\rho(g)}e^{-\sum_{\beta\in\wgts_U}(c_\beta \sprod{\beta}{H_0(g)}_+ +N\sprod{\beta}{H_0(g)}_-)}.
\]
In particular,
\begin{equation} \label{eq: thetacoarsebound}
\Theta_U(g)\ll_\phi^{\rd} \delta_{P_0}(m_0(g))\abs{\det\rho(g)}.
\end{equation}
The proposition follows.
\end{proof}

\section{Special subspaces} \label{sec: further}

In this section, we provide some geometric prerequisites to analyze further the behavior of $\theta_\phi$ on the Siegel domain.
A crucial role will be played by certain subspaces of $V$ that are {\PVS}s for parabolic subgroups of $G$ (cf. Definition \ref{def: special} below). 

\subsection{}

We start with the following elementary fact from convex geometry.

\begin{lemma} \label{lem: polycone}
Let $W$ be a finite-dimensional real vector space and let $v_i$, $i\in I$ be finitely many vectors in $W$.
Given $\lambda\in\R_{\ge 0}\{v_i, i\in I\}$ we write $\lambda=\sum_{i\in I}\lambda_i v_i$ with $\lambda_i\ge0$ for all $i$
and with maximal number of nonzero (i.e., positive) coefficients. Then,
\begin{enumerate}
\item $I':=\{i\in I\mid \lambda_i>0\}$ is uniquely determined by $\lambda$.
\item Suppose that $\lambda = \sum_{i\in I}\lambda'_i v_i$ with
$\lambda'_i \in \R$, $i \in I$, such that $\lambda'_i\ge0$ for all $i\notin I'$.
Then, $\lambda'_i=0$ for all $i\notin I'$.
\item \label{part: allin} In particular, for any $i\in I$, if $v_i$ is in the linear span of $\{v_j\}_{j\in I'}$ then $i\in I'$.
\end{enumerate}
\end{lemma}

\label{sec: PS}
In the situation of Lemma \ref{lem: polycone} we will say that $v_i$, $i\in I'$ is the \emph{\PS} of $\lambda$ with respect to $v_i$, $i\in I$.

Fix $U\in\Minset(V)$.
Recall that $\lambda(U) \in X^*(T_0)$ was defined in \eqref{def: lambda(U)} and that
$\lambda(U)\in\R_{\ge0}(\Delta_0\cup\wgts_U)$ by \eqref{eq: lambdaUincone}.
Let $\cone_U \subset \aaa_0$ be the cone dual to $\R_{\ge0}(\Delta_0\cup\wgts_U)$, that is,
\begin{equation} \label{def: CU}
\cone_U=\{x\in\aaa_{0,+}\mid\sprod{\beta}x\ge0\ \forall\beta\in\wgts_U\}.
\end{equation}
Let $\face_U$ be the face of $\cone_U$ defined by $\lambda(U)$, i.e.,
\[
\face_U=\{x\in\cone_U\mid\sprod{\lambda(U)}x=0\}.
\]
Roughly speaking, $\face_U$ describes the directions in $T_0(F)\bs (T_0(\A)\cap\Siegel)$ where the bounded function $\delta_{P_0}^{-1}\abs{\det\rho}^{-1}\Theta_U$ 
(see \eqref{eq: thetadecomp}) does not decay.
To analyze $\face_U$, let $A\subset\Delta_0\cup\wgts_U$ be the {\PS} of $\lambda(U)$ with respect to $\Delta_0\cup\wgts_U$. Let $I=A\cap\Delta_0$ and $J=A\cap\wgts_U$.

(Note that we do not assume at the outset that $\Delta_0$ and $\wgts_U$ are disjoint.)

\begin{definition}
Let $\Env(U)$ (the enveloping group of $U$) be the standard parabolic subgroup of $G$ such that $I=\Delta_0^{\Env(U)}$. 
\end{definition}

Then,
\[
\face_U=\{x\in\cone_U\mid\sprod{\alpha}x=\sprod{\beta}x=0\ \forall\alpha\in\Delta_0^{\Env(U)},\beta\in J\}.
\]

\begin{remark}
Let $U\in\Minset(V)$ and let $\iota$ be as in Lemma \ref{lem: suppSbnd}.
Then, the image under $\iota$ of $\jndx_{V/U}$ contains $\indx_{N_{\Env(U)}}$ and is contained in $\indx_{\Env(U)}$.

Indeed, as in the proof of Corollary \ref{cor: lam(U)} we write
\[
\lambda (U) =\sum_{j\in \jndx_{V/U}} (\beta_j+\alpha_{\iota(j)}) +
\delta_0-\sum_{j\in \jndx_{V/U}}\alpha_{\iota(j)}.
\]
It follows from the definition of $\Env(U)$ that $\delta_0-\sum_{j\in \jndx_{V/U}}\alpha_{\iota(j)}$
can be written as a linear combination of $\Delta_0^{\Env(U)}$. This precisely means that 
$\indx_{N_{\Env(U)}}\subset\iota(\jndx_{V/U})\subset\indx_{\Env(U)}$.
\end{remark}

\subsection{}
To go further, we need a follow-up of Lemma \ref{lem: suppSbnd}. It will play a key role.

\begin{proposition} \label{prop: cent}
Let $U\in\Minset(V)$. Assume that $P_0$ stabilizes $U$, so that the stabilizer $S=\Stab(U)$ of $U$ in $G$
is a standard $F$-parabolic subgroup of $G$. Then,
\begin{enumerate}
\item There exists a one-to-one function
\[
\iota:\jndx_{V/U}\rightarrow\indx_{N_S}
\]
such that $\beta_j+\alpha_{\iota(j)}\in\wgts_U$ for all $j\in\jndx_{V/U}$. In particular, $\dim V/U\le \dim G/S$.
\item $S\subset \Env(U)$.
\item The following conditions are equivalent. \label{part: special}
\begin{enumerate}
\item $\dim V/U=\dim G/S$. \label{part: dimeqlty}
\item $\iota$ is a bijection. \label{part: iotabij}
\item $\lambda(U)=\sum_{j\in\jndx_{V/U}}(\beta_j+\alpha_{\iota(j)})+\delta_0^S$. \label{part: lamexp}
\item For some (or every) $v\in U\cap \Vreg$, $\rho(S)v$ is dense in $U$ and $G_v^\circ\subset S$. \label{part: prehom}
\item For some (or every) $v\in U\cap \Vreg$, the variety \label{part: var}
\[
G_{v,U}=\{g\in G\mid\rho(g)^{-1}v\in U\}
\]
is the union of finitely many left cosets of $S$ in $G$.
\end{enumerate}
\item If the conditions in part \ref{part: special} are satisfied, then $U$ is a {\PVS} with respect to $S$
with open orbit $U^{\reg}=U\cap \Vreg$. Moreover, $G_{v,U}=G_vS$ for any $v\in U^{\reg}$.

On the other hand, if the conditions in part \ref{part: special} are not satisfied, then $\Env(U) \neq S$.
\end{enumerate}
\end{proposition}

\begin{proof}
Write $\Liealg{S}=\Lie S$ for the Lie algebra of $S$.
Let $N=N_S$, and let $\bar{N}$ be the unipotent radical of the parabolic subgroup of $G$ opposite to $S$.
Denote the corresponding Lie algebras by $\Liealg{N}=\Lie N$ and $\bar{\Liealg{N}}=\Lie\bar{N}$.

Let $v_0\in U\cap \Vreg$. The composition of the surjection $D_{v_0}$ \eqref{def: Dv} with the projection $V\rightarrow V/U$ factors as a surjective map
\[
\tilde{D}_{v_0}:\bar{\Liealg{N}}\simeq\Liealg{G}/\Liealg{S}\rightarrow V/U\simeq\oplus_{\beta\in\wgts_V\setminus\wgts_U}V_\beta.
\]
We have $\tilde D_{v_0}(\Liealg{G}_\alpha)\subset V_{(\alpha+\wgts_U)\setminus\wgts_U}$ for all $\alpha\in\rts_{\bar{N}}=-\rts_N$.
The existence of $\iota$ follows as in the proof of Lemma \ref{lem: suppSbnd}.
Moreover,
\[
\lambda(U)=\sum_{\beta\in\wgts_V\setminus\wgts_U}n_\beta\beta+\delta_0=
\sum_{j\in\jndx_{V/U}}(\beta_j+\alpha_{\iota(j)})+\delta_0-\sum_{j\in\jndx_{V/U}}\alpha_{\iota(j)}.
\]
Clearly,
\[
\delta_0-\sum_{j\in\jndx_{V/U}}\alpha_{\iota(j)}=
\delta_0-\sum_{\alpha\in\rts_N}(\dim\Liealg{G}_\alpha)\alpha+\mu=\delta_0^S+\mu,
\]
where $\mu=\sum_{i\in\indx_N\setminus \iota (\jndx_{V/U})}\alpha_i$ is a (possibly empty) sum of roots in $N$.
Therefore, $S\subset \Env(U)$, and the inclusion is proper if $\mu\ne0$.

Note that the following statements are equivalent: $\mu=0$; $\iota$ is a bijection; $\dim V/U=\dim G/S$; $\tilde D_{v_0}$ is an isomorphism of vector spaces.

Observe that the kernel of $\tilde{D}_{v_0}$ is $D_{v_0}^{-1}(U)/\Liealg{S}$ and we have a short exact sequence
\[
0\rightarrow\Ker D_{v_0}/( \Ker D_{v_0} \cap \Liealg{S}) \rightarrow\Ker\tilde{D}_{v_0}\rightarrow U/D_{v_0}(\Liealg{S})\rightarrow0.
\]
Thus, $\tilde{D}_{v_0}$ is an isomorphism if and only if $\Ker D_{v_0}\subset\Liealg{S}$ and $D_{v_0}(\Liealg{S})=U$.
Since $\Ker D_{v_0}=\Lie G_{v_0}$, the first condition means that $G_{v_0}^\circ\subset S$, while the second
condition means that the $S$-orbit of $v_0$ is dense in $U$.
This proves that conditions \eqref{part: dimeqlty}--\eqref{part: prehom} are equivalent and that the validity of \eqref{part: prehom} does not depend on the choice of $v\in U\cap \Vreg$.
Therefore, under these conditions, $U\cap \Vreg$ is a single $S$-orbit, since $U\cap \Vreg$ is irreducible.
Since $U\cap \Vreg=\rho(G_{v_0,U})^{-1}v_0$, this also implies that
$G_{v_0,U}=G_{v_0}S$, which is a finite union of left $S$-cosets since $G_{v_0}^\circ\subset S$.

Finally, if $G_{v_0,U}=\bigcup_{i=1}^n \eta_i S$ is the union of finitely many left $S$-cosets,
then
\[
U\cap \Vreg=\rho(G_{v_0,U})^{-1}v_0=\bigcup_{i=1}^n \rho(S)\rho(\eta_i)^{-1}v_0
\]
is the union of finitely many $S$-orbits, at least one of which must be dense, while
$G_{\rho(\eta_i)^{-1}v_0}^\circ\subset S$ for every $i = 1, \ldots, n$,
since $G_{\rho(\eta_i)^{-1}v_0}=\eta_i^{-1}G_{v_0}\eta_i\subset \eta_i^{-1}G_{v_0,U}$
is contained in a finite union of left $S$-cosets.
Thus, condition \eqref{part: var} is also equivalent to the other ones.
This concludes the proof of the proposition.
\end{proof}

As we shall see, the spaces considered in the proposition above deserve special attention.

\begin{definition} \label{def: special}
We denote by $\Spcl(V)$ the set of spaces satisfying the conditions of Proposition \ref{prop: cent} part \ref{part: special}.
Thus, $\Spcl(V)$ consists of the subspaces $U\in\Minset(V)$ that are stabilized by $P_0$ and satisfy $\dim V/U=\dim G/\Stab(U)$.

A subspace in $\Spcl(V)$ will be called \emph{{\SPCL}}.
\end{definition}

\begin{remark} \label{rem: basic}
Trivially, $V\in\Spcl(V)$.
It follows from Proposition \ref{prop: cent} that if $V$ is {\BASIC} (see Definition \ref{def: basic}), then $\Spcl(V)=\{V\}$.

The converse is not true: $V=\Mat_n$ is a {\PVS} with respect to $G=\GL_n$ acting by matrix multiplication. Here $H=1$,
but it is easy to see that $\Spcl(V)=\{V\}$.

In fact, the four examples in Table \ref{tab: FTMS} of \S\ref{sec: magic square} provide counterexamples for irreducible {\PVS} of DK-type.
Indeed, by Proposition \ref{prop: cent}, if $U\in\Spcl(V)$ then it is stable under the maximal standard subgroup parabolic $Q$ of $G$
whose Levi part $L$ satisfies $H^\circ=L^{\der}$. In particular, $U$ is stable under $L$.
However, $V$ decomposes into four irreducible subrepresentations of $L$, $V=U_1\oplus U_2\oplus U_3\oplus U_4$
where $U_1,U_2$ are one-dimensional, $\wgts_{U_1}$ consists of the highest weight of $V$, and $\dim U_3=\dim U_4=\dim G/Q$.
Since every non-trivial $P_0$-stable subspace of $V$ contains $U_1$, we cannot have $\dim V/U=\dim G/Q$. Hence $U=V$.

Consider the list of reduced, absolutely irreducible {\PVS}s of Sato and Kimura in the appendix to \cite{MR1944442}.
Of the 29 regular cases, 23 are of DK-type.\footnote{Case \# 1 can be of DK-type or not, but it is always basic.} Among these, we have $\Spcl(V) \neq \{V \}$ only in the following eight cases:
\# 2 (Example \ref{ex: Speven} for $k=1$ and $n_1=2$), \# 4 (Example \ref{ex: G2}), \# 8 (Example \ref{ex: F4}), \# 9, \# 11, \# 12 (Example \ref{ex: E6}),
\# 15 (Example \ref{ex: SOodd} for $k=1$ and $n_1=2$ or $n_2-2$), \# 28. (To verify the assertion $\Spcl(V) \neq \{V \}$ in the remaining cases \# 9, \# 11 and \# 28,
one may use the construction of \S\ref{sec: examspcl}.)
\end{remark}

\begin{remark}
In general, we may have $\Env(U)\supsetneq \Stab(U)$ even if $U\in\Spcl(V)$. See Examples \ref{ex: F4'} and \ref{ex: E6'} below.
\end{remark}

In principle, using Proposition \ref{prop: cent} one can enumerate $\Spcl(V)$, at least in small rank cases or when
$H$ is contained in only a few parabolic subgroups.

\begin{example} \label{ex: rank2}
Consider Example \ref{ex: Speven} for $k=1$, $n_1=2$ and Example \ref{ex: G2}. Write the simple roots of $G'$ as $\{\alpha,\beta\}$
where $\alpha$ is the short root.
Then $\Delta_0^G=\{\alpha\}$ and $\wgts_V=\{\beta,\beta+\alpha,\beta+2\alpha[,\beta+3\alpha]\}$ where root in square brackets occurs for $G'=G_2$.
It is easy to see that $\Spcl(V)=\{V,U\}$ where $\wgts_U=\wgts_V\setminus\{\beta\}$. $\Stab(U)=\Env(U)$ is the Borel subgroup of $G$.
\end{example}

\subsection{}

For $\alpha\in\rts_G$ denote by $N_{(\alpha)}$ the unipotent subgroup of $G$ with Lie algebra $\Liealg{G}_\alpha+\Liealg{G}_{2\alpha}$
if $2\alpha\in\rts_G$ and $\Liealg{G}_\alpha$ otherwise.

The following general fact about representations is probably well-known.

\begin{lemma} \label{lem: sumofrts}
Let $U\in\Sub(V)$ and $\alpha\in\rts_G$.
Suppose that $U$ is stable under $N_{(\alpha)}$.
Then, for every $\beta\in \wgts_U$ such that $\alpha+\beta\in\wgts_V$ we have $\alpha+\beta\in \wgts_U$.
\end{lemma}

This follows from the representation theory of $\SL_2$. To explain this, we need a standard consequence of
the Jacobson--Morozov Theorem.

\begin{lemma} \label{lem: SL_2}
Let $\alpha\in\rts_G$ and let $0\ne e\in\Liealg{G}_\alpha$.
Then, there exists a unique $F$-morphism $\varphi:\SL_2 \rightarrow G$ such that $d\varphi(\sm 0100)=e$,
the image of the diagonal torus under $\varphi$ is contained in $T_0$ and the co-character
$t\mapsto\varphi(\sm t{}{}{t^{-1}})$ is $\alpha^\vee$.
\end{lemma}

\begin{proof}
This is well-known. For convenience we include the proof.

For the uniqueness of $\varphi$ -- see \cite{MR1251060}*{Lemma 3.4.4}.

By passing to the subgroup generated by $N_{(\alpha)}$ and $N_{(-\alpha)}$, we may assume that $G$ is semisimple of $F$-rank one.
We may also assume without loss of generality that $\alpha>0$, so that $N_{(\alpha)}\subset N_0$.

By the Jacobson--Morozov Theorem, there exists an $F$-morphism $\varphi':\SL_2 \rightarrow G$ such that $d\varphi'(\sm 0100)=e$.
Let $a$ be the restriction of $\varphi'$ to the diagonal torus, i.e., $a(s)=\varphi'(\sm s{}{}{s^{-1}})$, so that $\Ad(a(s))e=s^2 e$.
Let $T$ be the image of $a$, which is a maximal split torus in $G$.
By the Bruhat decomposition, for any $g\notin P_0$ we have $gN_0g^{-1}\cap N_0=1$.
Thus, $T\subset P_0$. It follows that $T=nT_0n^{-1}$ for some element $n$ of $N_0(F)$.

We claim that $n$ centralizes $e$.
Indeed, for every $t\in T_0$, both $\Ad(t)e$ and $\Ad(ntn^{-1})e$ are scalar multiples of $e$.
Fix $t\in T_0$ such that its centralizer in $N_0$ is trivial. Write $ntn^{-1}=n't$ where $n'=[n,t]\in N_0$.
Then, $\Ad(n')e$ is a scalar multiple of $e$, and hence $\Ad(n')e=e$ since $n'$ is unipotent.
Note that the centralizer $N_{0,e}$ of $e$ in $N_0$ is invariant under conjugation by $T_0$ (since $\Ad(T_0)$ acts by scalars on $e$).
By \cite{MR1102012}*{Proposition III.9.3} the maps $N_0\rightarrow N_0$ and $N_{0,e}\rightarrow N_{0,e}$ given by $n\mapsto n'$ are bijective.
Since $n'\in N_{0,e}$, we infer that $n\in N_{0,e}$ as claimed.

We can now take $\varphi=\varphi'(n^{-1}\cdot n)$.
\end{proof}

\begin{proof}[Proof of Lemma \ref{lem: sumofrts}]
For any integer $k$ let $V_k\in\Sub(V)$ be such that
\[
\wgts_{V_k}=\{\beta\in\wgts_V\mid\sprod{\beta}{\alpha^\vee}=k\}.
\]
Fix $0\ne e\in\Liealg{G}_\alpha$.
Using Lemma \ref{lem: SL_2} and basic facts about representation theory of $\SL_2$, it follows that for any integer $k$, the map
\[
d\rho(e):V_{k-1}\rightarrow V_{k+1}\text{ is }
\begin{cases}\text{surjective}&\text{if }k\ge0,\\\text{injective}&\text{if }k\le0,\end{cases}
\]
Thus, for every $\beta\in\wgts_V$,
\[
d\rho(e):V_\beta\rightarrow V_{\beta+\alpha}\text{ is }
\begin{cases}\text{surjective}&\text{if }\sprod{\beta}{\alpha^\vee}+1\ge0,\\\text{injective}&\text{if }\sprod{\beta}{\alpha^\vee}+1\le0.\end{cases}
\]
In particular, $d\rho(e)(V_\beta)\ne0$ if $V_{\beta+\alpha}\ne0$. The lemma follows.
\end{proof}

\begin{lemma} \label{lem: intersections}
Suppose that $U_1,U_2\in\Spcl(V)$ and $U_1\cap U_2\in\Minset(V)$. Then, $U_1\cap U_2\in\Spcl(V)$ and $\Stab(U_1\cap U_2)=\Stab(U_1)\cap \Stab(U_2)$.
\end{lemma}

\begin{proof}
Let $U=U_1\cap U_2$, $S=\Stab(U)$, $N=N_S$, $S_k=\Stab(U_k)$, $N_k=N_{S_k}$, $M_k=M_{S_k}$, $k=1,2$.
Clearly, $\wgts_U=\wgts_{U_1}\cap\wgts_{U_2}$, $P_0$ stabilizes $U$ and $S\supset S_1\cap S_2$.
In particular, $\rts_N\subset\rts_{N_1}\cup\rts_{N_2}$.

By Proposition \ref{prop: cent}, there exists a one-to-one function
\[
\iota:\jndx_{V/U}\rightarrow\indx_N
\]
such that $\beta_j+\alpha_{\iota(j)}\in\wgts_U$ for all $j\in\jndx_{V/U}$.

We claim that if $\beta_j\in\wgts_V\setminus\wgts_{U_1}$, then $\alpha_{\iota(j)}\in\rts_{N_1}$.
Indeed, suppose otherwise. Then, $\alpha_{\iota(j)}\in\rts_{M_1}$ (since $\alpha_{\iota(j)}>0$ as $\alpha_{\iota(j)}\in\rts_N$).
Therefore, using Lemma \ref{lem: sumofrts}, $\beta_j=(\beta_j+\alpha_{\iota(j)})-\alpha_{\iota(j)}\in\wgts_{U_1}$ since $U_1$ is $S_1$-stable, in contradiction to
the assumption that $\beta_j\notin\wgts_{U_1}$.

Thus, the restriction of $\iota$ to $\jndx_{V/U_1}$ is a bijection with $\indx_{N_1}$ (since $U_1\in\Spcl(V)$).
Similarly, the restriction of $\iota$ to $\jndx_{V/U_2}$ is a bijection with $\indx_{N_2}$.

It follows that $\iota$ is onto and $\rts_N=\rts_{N_1}\cup\rts_{N_2}$.
Thus, $U \in\Spcl(V)$ and $N=N_1N_2$, which implies that $S=S_1\cap S_2$.
\end{proof}

\begin{lemma} \label{lem: admsets}
Let $U\in\Minset(V)$ and $\tilde U=\Env(U)\star U$. Then, $\tilde U\in\Spcl(V)$ and $\Env(\tilde U)=\Stab(\tilde U)=\Env(U)$.
\end{lemma}

\begin{proof}
Clearly $\tilde U\in\Minset(V)$, since $\tilde U\supset U$, and $\Stab(\tilde U)\supset \Env:=\Env(U)$.
In view of Proposition \ref{prop: cent}, in order to prove the lemma we only need to show that $\tilde \Env:=\Env(\tilde U)\subset \Env$.
We have $\wgts_{\tilde U}\supset \wgts_U$. Write
\[
\lambda(\tilde U)= \sum_{\beta\in\wgts_{\tilde U}} c_\beta \beta +\sum_{\alpha\in \Delta_0^{\tilde \Env}} d_\alpha\alpha
\]
with $c_\beta\ge0$ for all $\beta\in\wgts_{\tilde U}$ and $d_\alpha>0$ for all $\alpha\in\Delta_0^{\tilde \Env}$. Then,
\[
\lambda(U)= \sum_{\beta\in\wgts_U} c_\beta\beta +
\sum_{\beta\in\wgts_{\tilde U}\setminus\wgts_U} (c_\beta+n_\beta)\beta +
\sum_{\alpha\in \Delta_0^{\tilde \Env}} d_\alpha \alpha.
\]
Using Lemma \ref{lem: Hclos} for each $\beta\in\wgts_{\tilde U}\setminus\wgts_U$, we can rewrite the above relation as
\[
\lambda(U)= \sum_{\beta\in\wgts_U} c'_\beta \beta +\sum_{\alpha\in \Delta_0} d'_\alpha \alpha,
\]
where $c'_\beta \ge 0$ for all $\beta\in\wgts_U$, $d'_\alpha\ge0$ for $\alpha\in\Delta_0\setminus\Delta_0^{\Env}$
and $d'_\alpha\ge d_\alpha>0$ for all $\alpha\in\Delta_0^{\tilde \Env}\setminus\Delta_0^{\Env}$.
By Lemma \ref{lem: polycone} we deduce that $d'_\alpha=0$ for all $\alpha\in\Delta_0\setminus\Delta_0^{\Env}$.
This implies that $\tilde \Env\subset \Env$, as required.
\end{proof}

\begin{corollary} \ \label{cor: spcsets}
\begin{enumerate}
\item If $U\in\Minset(V)$ is stable under $\Env(U)$, then $U\in\Spcl(V)$.
\item If $U\in\Minset(V)$ and $\Env(U)\subsetneq G$, then there exists $\tilde U\in\Spcl(V)$, $\tilde U\ne V$,
containing $U$.
\item \label{part: spcl} In particular, if $\Spcl(V)=\{V\}$, then $\Env(U)=G$ for every $U\in\Minset(V)$.
\end{enumerate}
\end{corollary}

\begin{example} \label{ex: GL2'}
Consider Example \ref{ex: GL} with $k=3$ and $n_1<n_2$. Then, it is easy to see that $\Spcl(V)=\{V,V_1,V_2\}$ where
\begin{align*}
V_1&=\{(x_1,x_2)\mid\text{the first $n_2-n_1$ columns of $x_1$ vanish}\},\\
V_2&=\{(x_1,x_2)\mid\text{the last $n_2-n_3$ rows of $x_2$ vanish}\}.
\end{align*}

The general case of Example \ref{ex: GL} is far more complicated and will not be considered here.
\end{example}

\begin{example} \label{ex: F4'}
Consider Example \ref{ex: F4} pertaining to $F_4$.
Write for simplicity $(i_1i_2i_3i_4)$ instead of $i_1\alpha_1+i_2\alpha_2+i_3\alpha_3+i_4\alpha_4\in X^*(T_0)$.
 
We explicate $\wgts_V$ in a way that demonstrates its realization as pairs of $3\times3$-symmetric matrices as follows:
\[
\wgts_V=\begin{psmallmatrix}(0100)&(0110)&(0111)\\(0110)&(0120)&(0121)\\(0111)&(0121)&(0122)\end{psmallmatrix},
\begin{psmallmatrix}(1100)&(1110)&(1111)\\(1110)&(1120)&(1121)\\(1111)&(1121)&(1122)\end{psmallmatrix}.
\]

Consider the subspaces $U_1,U_2,U_3$ in $\Sub(V)$ given by
\begin{align*}
\wgts_{U_1}&=\wgts_V\setminus \{(0100),(1100)\},\\
\wgts_{U_2}&=\wgts_V\setminus \{(0100),(0110),(0120)\},\\
\wgts_{U_3}&=\wgts_V\setminus \{(0100),(0110),(0111)\}.
\end{align*}
In matrices,
\[
U_1=\begin{psmallmatrix} 0&*&* \\ *&*&* \\ *&*&* \end{psmallmatrix}, \begin{psmallmatrix} 0&*&* \\ *&*&* \\ *&*&* \end{psmallmatrix},\ \ 
U_2=\begin{psmallmatrix} 0&0&* \\ 0&0&* \\ *&*&* \end{psmallmatrix}, \begin{psmallmatrix} *&*&* \\ *&*&* \\ *&*&* \end{psmallmatrix},\ \ 
U_3=\begin{psmallmatrix} 0&0&0 \\ 0&*&* \\ 0&*&* \end{psmallmatrix}, \begin{psmallmatrix} *&*&* \\ *&*&* \\ *&*&* \end{psmallmatrix}.
\]
Let $U_4=U_1\cap U_2$, $U_5=U_2\cap U_3$ so that
\[
U_4=\begin{psmallmatrix} 0&0&* \\ 0&0&* \\ *&*&* \end{psmallmatrix}, \begin{psmallmatrix} 0&*&* \\ *&*&* \\ *&*&* \end{psmallmatrix},\ \ 
U_5=\begin{psmallmatrix} 0&0&0 \\ 0&0&* \\ 0&*&* \end{psmallmatrix}, \begin{psmallmatrix} *&*&* \\ *&*&* \\ *&*&* \end{psmallmatrix}.
\]
Then, it is not difficult to check that
\[
\Spcl(V)=\{V,U_1,U_2,U_3,U_4,U_5\}.
\]
Moreover,
\begin{gather*}
\Stab(U_1)=\Env(U_1)=(*00*),\ \Stab(U_2)=\Env(U_2)=(00{*}0),\ \Stab(U_3)=\Env(U_3)=(000*),\\
\Stab(U_4)=\Stab(U_5)=\Env(U_5)=(0000),\ \ \Env(U_4)=(00{*}0).
\end{gather*}
\end{example}

\section{Proof of Theorem \ref{thm: convergence}} \label{sec: pfmain}

In this section we prove Theorem \ref{thm: convergence} by providing an estimate on $\Theta_U$ which is finer than \eqref{eq: thetacoarsebound}
for spaces $U \in \Minset (V)$ containing {\sr} elements (see Definition \ref{def: sr}).

\subsection{} \label{sec: poscones}

Let $\cone$ be a closed cone in a finite-dimensional real vector space $W$.

Let $f$ be a continuous real-valued function on $W$ and let $d>0$.
We say that $f$ is positively-$d$-homogeneous if $f(rv)=r^df(v)$ for all $r>0$, $v\in W$. (In the cases at hand, $d=1$.)
In this case, $f$ is positive on $\cone\setminus\{0\}$ if (and course, only if) it is positive on the intersection of $\cone$ with the 
unit sphere of $W$ with respect to a prescribed norm.

The space of continuous positively-$d$-homogeneous functions on $W$ is a Banach space with respect to the supremum norm on the unit ball
and the positivity on $\cone\setminus\{0\}$ is an open condition.

By integration using spherical coordinates it is clear that if $f$ is positively-$d$-homogeneous and positive on $\cone\setminus\{0\}$, then the integral of $e^{-f}$ over $\cone$ converges.

Suppose that $f$ is positively-$1$-homogeneous and let $\lambda_j$, $j\in J$ be a finite set of vectors in $W^*$.
Then, the positively-$1$-homogeneous function $f+N\sum_{j\in J}\sprod{\lambda_j}v_-$ is positive on $\cone\setminus\{0\}$ for $N\gg0$ if and only if
$f$ is positive on the subcone $\{v\in\cone:\sprod{\lambda_j}v\ge0\ \forall j\in J\}\setminus\{0\}$.

These facts yield the following elementary convergence result.

\begin{lemma} \label{lem: convergence}
Let $\chi_i$, $i\in I$ and
$\lambda_j$, $j\in J$ be finite sets of vectors in $\aaa_0^*$. 
Let $W \subset \aaa_0$ be a subspace and $f$ a continuous positively-$1$-homogeneous function on $W$.
Assume that $f$ is positive on the cone
\[
\{x\in\aaa_{0,+} \cap W \mid\sprod{\lambda_j}x\ge0\ \forall j\in J\}\setminus\{0\}.
\]
Then, there exists $c>0$ such that for all $N\gg0$ the integral
\[
\int_{\aaa_{0,+} \cap W} e^{-f(H)-N\sum_{j\in J}\sprod{\lambda_j}{H}_- - \sum_{i\in I}\sigma_i \sprod{\chi_i}{H}} \ dH
\]
converges for $\abs{\sigma_i}<c$, $i\in I$. Moreover, if $\chi_i\in\R_{\ge0}\{\lambda_j\}_{j\in J}$ for all $i\in I$, then for
all $\sigma_i>-c$, $i\in I$ the integral converges for sufficiently large $N$ (depending on the maximum value of the $\sigma_i$'s).
\end{lemma}

\subsection{}

We now turn to the support sets of {\sr} elements.

\begin{definition}
We denote by $\Sub_{\sreg} (V) \subset \Minset (V)$ the set of subspaces in $\Sub(V)$ that contain a {\sr} element.
\end{definition}

\begin{definition}
Let $U\in\Spcl(V)$. We say that $U$ is \emph{exceptional}
if there exists $U'\in\Minset(V)$, $U'\subset U$ such that the linear span of $X^*(G)\cup\Delta^{\Stab(U)}_0\cup\wgts_{U'}$ is a proper subspace of $\aaa_0^*$.
In this case, we will also say that $(U,U')$ is an exceptional pair.

We denote by $\Sub_{\excp}(V)$ the set of exceptional subspaces of $V$.
\end{definition}

\begin{proposition} \label{prop: stronglyregular}
Let $U\in\Spcl(V)\cap \Sub_{\sreg}(V)$. Then $U$ is not exceptional.

Conversely, suppose that $V$ is regular. Then, for any {\nsr} $v\in \Vreg(F)$ there exists an exceptional pair $(U,U')$
such that $v\in\rho(G(F))U'$.

Thus, if $V$ is regular, then the set of {\nsr} elements in $\Vreg(F)$ is
\[
\cup_{U\in\Sub_{\excp}(V)}\rho(G(F))U\cap \Vreg(F)=\cup_{(U,U')}\rho(G(F))U'\cap \Vreg(F)
\]
where $(U,U')$ range over the exceptional pairs.
\end{proposition}

\begin{remark}
The first implication cannot be reversed. There is an example of a regular {\PVS} $V$ and a non-exceptional $U\in\Spcl(V)$ such that all elements of $U\cap\Vreg (F)$ are {\nsr} (see Example \ref{ex: E6'} below).
\end{remark}

\begin{proof}
For the first part, let $U\in\Spcl(V)$ and $S=\Stab(U)$.
Assume that there exists $U'\in\Minset(V)$, $U'\subset U$ such that 
the linear span of $X^*(G)\cup\Delta^S_0\cup\wgts_{U'}$ is a proper subspace of $\aaa_0^*$.
Equivalently, there exists a (split) torus $1\ne T\subset T_0 \cap Z(M_S) \cap G^{\der}$ fixing $U'$ pointwise
(namely, the identity component of $\cap_{\chi\in X^*(G)\cup\Delta_0^S\cup\wgts_{U'}}\Ker\chi$).
We will show that all elements of $U\cap \Vreg(F)$ are {\nsr}.

Fix a regular element $u$ in $U'$. By Proposition \ref{prop: cent}, we have $G_u^\circ\subset S$.
Since $T$ lies in the center of $M_S$, $[T,G_u\cap S]\subset [Z(M_S),S]\subset N_S$. On the other hand, $T\subset G_u^\circ$
(since $T$ stabilizes all elements of $U'$).
Therefore, $T$ centralizes $(G_u\cap S)/(G_u^\circ\cap N_S)$. In particular, the image of $T$ in $G_u^\circ/(G_u^\circ\cap N_S)$ is central.
Note that $G_u^\circ\cap N_S$ is contained in the unipotent radical of $G_u^\circ$.

Now, let $v\in U\cap \Vreg(F)$. By Proposition \ref{prop: cent}, $v$ is in the $S$-orbit (but not necessarily the $S(F)$-orbit) of $u$.
Let $x\in S (\bar{F})$ be such that $u=\rho(x)v$. Then, $x^{-1}Tx$ is a torus in $G^{\der}\cap G_v^\circ$.
Moreover, for any $\sigma \in\Gal(\bar{F}/F)$ we have $\sigma(x)x^{-1}\in G_u\cap S$.
Let $T'$ be the image of $x^{-1}Tx$ in $G_v^\circ/(G_v^\circ\cap N_S)$. Then, $T'$ is central.
Moreover, $T'$ is a split torus. In fact, since $T$ centralizes $\sigma(x)x^{-1}(G_u^\circ\cap N_S)$ for all $\sigma \in\Gal(\bar{F}/F)$,
the map $t\mapsto x^{-1}tx (G_v^\circ\cap N_S)$ is an $F$-isomorphism of $T$ with $T'$.
In particular, the image of $T'$ in $G_v^{\red}$ is a non-trivial $F$-split central torus.
Thus, $v$ is {\nsr} by Lemma \ref{lem: isot}. This proves the first part.

For the second part, suppose that $V$ is regular and $v\in \Vreg (F)$ is {\nsr}. Let $T\ne1$ be a split torus contained in $Z(G^\circ_v)\cap G^{\der}$.
The centralizer $L=C_G(T)$ contains $G_v^\circ$.
Upon replacing $v$ by an element in its $G(F)$-orbit, we can assume that $T$ is contained in $T_0  \cap G^{\der}$ and that $L$ is a standard Levi subgroup of $G$.
Thus, $X_*(T)\subset\aaa_L$ and in the notation of \cite{MR1361168}*{I.1.10} we have $X_*(T)\cap\aaa_L^{\reg}\ne\emptyset$.
Fix an element $h \in X_*(T)\cap\aaa_L^{\reg} \subset \aaa^G_0$. Upon conjugating by an element of the Weyl group of $T_0$, we may assume that $\sprod{\alpha}{h}$ is non-negative
precisely for the roots $\alpha$ of the (proper) standard parabolic subgroup $Q$ of $G$ with Levi subgroup $L$ \cite{MR1361168}*{Lemma I.1.10}.

Consider the subspaces $U,U'\in \Sub(V)$ given by
\[
\wgts_U = \{ \beta \in \wgts_V \mid\sprod{\beta}{h} \ge 0 \},\ \ \wgts_{U'}= \{ \beta \in \wgts_V \mid\sprod{\beta}{h} = 0 \} \subset\wgts_U.
\]
Clearly $U$ is stable under $Q$ and $U'\subset U$. Since $T$ stabilizes $v$, $\supp v\subset\wgts_{U'}$ so that $v \in U'$.
In particular, $U'\in \Minset (V)$, and a fortiori $U \in \Minset (V)$.
Recall the surjection \eqref{def: Dv}. It satisfies
\[
D_v(\Liealg{G}_\alpha)\subset V_{\alpha+\wgts_{U'}}\text{ for all }\alpha\in\rts_G\cup\{0\}
\]
by \eqref{eq: drhoab}. In particular, $D_v(\Liealg{Q})\subset U$ and $D_v(\bar{\Liealg{N}})\cap U=0$ where $\Liealg{Q}=\Lie Q$ and $\bar{\Liealg{N}}$ is the Lie algebra of
the unipotent radical of the parabolic subgroup of $G$ opposite to $Q$.
We claim that the preimage of $U$ under $D_v$ is $\Liealg{Q}$.
It suffices to show that $D_v^{-1}(U)\cap\bar{\Liealg{N}}=0$.
Suppose that $x\in\bar{\Liealg{N}}$ and $D_v(x)\in U$. Then, by the above, $D_v(x)=0$.
Hence, $x$ is in the Lie algebra of $G_v$, which is contained in $\Liealg{L}$, and thus $x=0$.

Hence, the induced map $\tilde D_v:\Liealg{G}/\Liealg{Q}\rightarrow V/U$ is an isomorphism.
It follows from Proposition \ref{prop: cent} that  $U\in\Spcl(V)$ and $\Stab(U)=Q$.

Since the nonzero vector $h \in \aaa^G_0$ annihilates $\wgts_{U'}\cup \Delta^Q_0$,
we obtain that $(U,U')$ is exceptional.

Finally, the last part follows from the first two parts.
\end{proof}

\begin{remark}
We do not know whether the second (and hence also the last) part of Proposition \ref{prop: stronglyregular} hold without the regularity assumption on $V$.
\end{remark}

\begin{example}
Proposition \ref{prop: stronglyregular} is best illustrated by the case
where $V$ is the {\PVS} of binary quadratic forms with respect to the action of $G=\GL_2$, i.e., the symmetric square representation.
(Example \ref{ex: Speven} with $k=1$, $n_1=2$.)
Thus, $\Vreg$ is the set of binary quadratic forms with nonzero discriminant.
In this case $\wgts_V = \{ \beta, \beta + \alpha, \beta + 2 \alpha \}$, where $\alpha$ is the simple root of $G$, and $\beta$ the lowest weight of $V$.
Note that $\beta + \alpha$ is the determinant character of $G$. Apart from $V$ itself, there is a single space $U$ in
$\Spcl(V)$, namely $U=V_{\beta+\alpha}+V_{\beta+2\alpha}$, and $\Env(U)=\Stab(U)$ is the standard Borel subgroup of $G$.
The space $U$ is exceptional. More precisely, $(U,U')$ is an exceptional pair where $U'=V_{\beta+\alpha}$.
Let $v\in \Vreg$. Then, $G_v$ is an orthogonal group so that $G_v^\circ$ is a one-dimensional torus.
Thus, Proposition \ref{prop: stronglyregular} amounts to the fact that $G_v^\circ$ is split if and only if $v$ can be represented
as $xy$ in suitable coordinates.
\end{example}

\begin{example} \label{ex: E6'}
Let us revisit Example \ref{ex: E6} pertaining to the simply connected group $G'=E_6$.

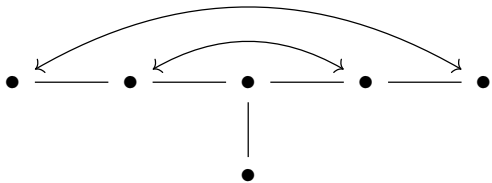
\begin{center}
\begin{tikzcd}
\bullet \arrow[r,dash] \arrow[rrrr,leftrightarrow,bend left]  & \bullet \arrow[r,dash] \arrow[rr,leftrightarrow,bend left] &
\bullet \arrow[r,dash] \arrow[d,dash] & \bullet \arrow[dash,r]  & \bullet \\ & & \bullet
\end{tikzcd}
\end{center}

For brevity we will write $\Eroot{abcde}f$ for the linear combinations of simple roots with coefficients $a,b,c,d,e,f$.

We will also write standard parabolic subgroups as $\Eroot{*0**0}*$ where $*$ indicates that the simple root is in the Levi subgroup
and $0$ otherwise.

Let $\theta$ be the outer involution of $G'$ acting on the simple roots as indicated.

Let $P$ be the maximal standard parabolic subgroup $\Eroot{**0**}*$.
The Levi subgroup $G$ has derived group $\SL_2 \times \SL_3 \times \SL_3$.

The unipotent radical $\rad{P}$, which has dimension $29$, has a three-step filtration, the first filtration step $V$ (the abelianization) has dimension $18$. 
We may write $\wgts_V$ explicitly to reflect the realization of $V$ as the space of pairs of $3\times 3$-matrices as follows.

\[
\wgts_V=\begin{pmatrix}\Eroot{11100}0,\Eroot{11110}0,\Eroot{11111}0\\\Eroot{01100}0,\Eroot{01110}0,\Eroot{01111}0\\\Eroot{00100}0,\Eroot{00110}0,\Eroot{00111}0\end{pmatrix},\ \ 
\begin{pmatrix}\Eroot{11100}1,\Eroot{11110}1,\Eroot{11111}1\\\Eroot{01100}1,\Eroot{01110}1,\Eroot{01111}1\\\Eroot{00100}1,\Eroot{00110}1,\Eroot{00111}1\end{pmatrix}
\]

We observe that for all proper subspaces $U\in\Spcl(V)$, $U \cap \Vreg (F)$ consists entirely of {\nsr} elements. 
Namely, if $U\in\Spcl(V)$ contains a {\sr} element $x$, then $G^\circ_x$
has to be contained in $Q = \Stab (U)$. But the projection of $G^\circ_x$ to 
any of the two $\SL_3$ factors of $G^{\der}$ is 
an anisotropic maximal torus of this group, which 
implies that $Q$ contains both $\SL_3$ factors. Therefore either $Q=G$ and $U=V$ or the intersection of $Q$ with the $\SL_2$ factor is the standard Borel subgroup.
However, for this maximal parabolic subgroup $Q$ of $G$ the only non-trivial $Q$-invariant subspace of $V$ does not intersect $\Vreg$.

Define $U_1,U_2,U_3\in\Sub(V)$ by
\begin{align*}
\wgts_{U_1}&=\wgts_V\setminus \{\Eroot{00100}0,\Eroot{01100}0,\Eroot{00100}1,\Eroot{01100}1\},\\
\wgts_{U_2}&=\wgts_V\setminus\{\Eroot{00100}0,\Eroot{01100}0,\Eroot{11100}0\},\\
\wgts_{U_3}&=\wgts_V\setminus\{\Eroot{00100}0,\Eroot{01100}0,\Eroot{00110}0,\Eroot{01110}0,\Eroot{00100}1,\Eroot{01100}1\}.
\end{align*}
In terms of matrices, these spaces are described as follows.
\[
U_1=\begin{psmallmatrix}***\\ 0**\\ 0**\end{psmallmatrix},\begin{psmallmatrix}***\\ 0**\\ 0**\end{psmallmatrix},\ \ 
U_2=\begin{psmallmatrix}0**\\ 0**\\ 0**\end{psmallmatrix},\begin{psmallmatrix}***\\ ***\\ ***\end{psmallmatrix},\ \ 
U_3=\begin{psmallmatrix}***\\ 00*\\ 00*\end{psmallmatrix},\begin{psmallmatrix}***\\ 0**\\ 0**\end{psmallmatrix}.
\]
Let $U_i'=\theta(U_i)$, $i=1,2,3$ and $U_i''=U_i\cap U_i'$ so that
\[
U_1'=\begin{psmallmatrix}***\\ ***\\ 00*\end{psmallmatrix},\begin{psmallmatrix}***\\ ***\\ 00*\end{psmallmatrix},\ \ 
U_2'=\begin{psmallmatrix}***\\ ***\\ 000\end{psmallmatrix},\begin{psmallmatrix}***\\ ***\\ ***\end{psmallmatrix},\ \ 
U_3'=\begin{psmallmatrix}***\\ 00*\\ 00*\end{psmallmatrix},\begin{psmallmatrix}***\\ ***\\ 00*\end{psmallmatrix},
\]
and
\[
U_1''=\begin{psmallmatrix}***\\ 0**\\ 00*\end{psmallmatrix},\begin{psmallmatrix}***\\ 0**\\ 00*\end{psmallmatrix},\ \ 
U_2''=\begin{psmallmatrix}0**\\ 0**\\ 000\end{psmallmatrix},\begin{psmallmatrix}***\\ ***\\ ***\end{psmallmatrix},\ \ 
U_3''=\begin{psmallmatrix}***\\ 00*\\ 00*\end{psmallmatrix},\begin{psmallmatrix}***\\ 0**\\ 00*\end{psmallmatrix}.
\]
We write $U_{i,j}=U_i\cap U_j$, $U_{i',j}=U_i'\cap U_j$, $U_{i',j'}=U_i'\cap U_j'$, $U_{i'',j}=U_i''\cap U_j$ and $U_{i'',j'}=U_i''\cap U_j'$.
Then, a straightforward but tedious computation gives the following Hasse diagram for $\Spcl(V)$:

\begin{tikzcd}
\\
& &     &      & V\\
& & U_1 \arrow[urr,dash] & U_1' \arrow[ur,dash] & U_2 \arrow[u,dash] & U_2' \arrow[ul,dash] \\
U_3 \arrow[urr,dash] & U_3' \arrow[urr,dash] & U_1'' \arrow[u,dash] \arrow[ur,dash] & U_{1,2} \arrow[ul,dash] \arrow[ru,dash] &
U_{1',2} \arrow[ul,dash] \arrow[u,dash] & U_{1,2'} \arrow[ulll,dash] \arrow[u,dash] & U_{1',2'} \arrow[ulll,dash] \arrow[ul,dash] &
U_2'' \arrow[ulll,dash] \arrow[ull,dash]\\
&U_3'' \arrow[ul,dash] \arrow[ur,dash] \arrow[u,dash] & U_{3',2} \arrow[ul,dash] \arrow[urr,dash] & U_{3,2'} \arrow[ulll,dash] \arrow[urr,dash] &
U_{1'',2} \arrow[ull,dash] \arrow[ul,dash] \arrow[u,dash] & U_{1'',2'} \arrow[ulll,dash] \arrow[u,dash] \arrow[ur,dash]\\
\end{tikzcd}

Moreover,
\begin{align*}
\Stab(U_1)&=\Eroot{0*\phantom{0}0*}*,\ \Stab(U_2)=\Eroot{**\phantom{0}0*}0,\ \Stab(U_3)=\Eroot{0*\phantom{0}00}0\\
\Stab(U_1')&=\Eroot{*0\phantom{0}*0}*,\ \Stab(U_2')=\Eroot{*0\phantom{0}**}0,\ \Stab(U_3')=\Eroot{00\phantom{0}*0}0\\
\Stab(U_1'')&=\Eroot{00\phantom{0}00}*,\ \Stab(U_2'')=\Eroot{*0\phantom{0}0*}0,\ \Stab(U_3'')=\Eroot{00\phantom{0}00}0.
\end{align*}
The stabilizers of the rest of the elements of $\Spcl(V)$ are given by Lemma \ref{lem: intersections}.
All spaces except $U_2$, $U'_2$ and $V$ are exceptional.

The spaces $U \in \Spcl(V)$ with $\Env(U) = \Stab(U)$ are $V, U_1, U_1', U_2, U_2', U_{1,2}, U_{1',2'}, U_1'', U_2''$.
For $U = U_3, U_{3,2'}, U_{1,2'}$ we have $\Env(U) \star U = U_1$.
For $U = U_3', U_{3',2}, U_{1',2}$ we have $\Env(U) \star U = U_1'$.
For $U = U_3'', U_{1'',2}, U_{1'',2'}$ we have $\Env(U) \star U = U_1''$.

\end{example}

\subsection{}
We can now formulate and prove the main geometric statement about $\lambda(U)$ where $U\in \Sub_{\sreg} (V)$.

Consider the natural surjection
\[
q:\aaa_G=X_*(G^{\ab})\otimes\R\longrightarrow\reducd{\aaa}_G=X_*(G/\centder{G})\otimes\R.
\]
The kernel of $q$ is the image of $X_*(\centder{G}/G^{\der})\otimes\R$ in $\aaa_G$.
The dual space $\reducd{\aaa}_G^*=X^*(G/\centder{G})\otimes\R\hookrightarrow\aaa_G^*$ of $\reducd{\aaa}_G$ has basis $\fundchr$.
Fix a section for $q$, i.e., a subspace $\cmpl$ of $\aaa_G$ such that the restriction of $q$ to $\cmpl$ is an isomorphism.
In other words, $\cmpl$ is a complement of $\Ker q$ in $\aaa_G$.
Let 
\[
W=(\aaa_0^G)\oplus\cmpl\hookrightarrow(\aaa_0^G)\oplus\aaa_G=\aaa_0.
\]

Recall the cone $\cone_U$ defined in \eqref{def: CU}.

\begin{proposition} \label{prop: convregs}
Suppose that $V$ is regular. Let $U \in \Sub_{\sreg} (V)$.
Then, for every $\mu\in\R_{>0}\fundchr$, the restriction of the linear form $\lambda(U)+\mu$ to $W$ is positive on the cone $\cone_U\cap W\setminus\{0\}$.
\end{proposition}

\begin{proof}
Suppose that $x\in\cone_U$ is such that $\sprod{\lambda(U)+\mu}x\le0$.
By Corollary \ref{cor: fundchrS} part \ref{part: nonneg}, we have $\sprod{\chi}x\ge 0$ for all $\chi\in\fundchr$.
By \eqref{eq: lambdaUincone}, we obtain $\sprod{\lambda(U)}x=0$ and $\sprod{\chi}x=0$ for all $\chi\in\fundchr$.
Also, by the definition of $\Env=\Env(U)$, we have $\sprod{\alpha}x=0$ for all $\alpha\in\Delta_0^\Env$.

Let $U'\in\Minset(V)$ be as in Corollary \ref{cor: fundchrS} part \ref{part: psipos}. Then, $\sprod{\beta}x=0$ for all $\beta\in\wgts_{U'}$.

Let $\tilde U=\Env(U)\star U$.
By Lemma \ref{lem: admsets}, $\tilde U\in\Spcl(V)$ and $\Env(\tilde U)=\Stab(\tilde U)=\Env$.

Since $\tilde U\supset U$, and $U\in\Sub_{\sreg}(V)$ we may apply Proposition \ref{prop: stronglyregular} to conclude that $\tilde U$ is not exceptional.
Thus, the set $\wgts_{U'}\cup\Delta_0^\Env\cup\aaa_G^*$ spans $\aaa_0^*$.
Let
\[
p:\aaa_0^*\rightarrow W^*
\]
be the restriction map, which is surjective. We have $\aaa_G^*=\Ker p\oplus\reducd{\aaa}_G^*$. Since $\fundchr$ spans $\reducd{\aaa}_G^*$,
we infer that $p(\wgts_{U'}\cup\Delta_0^\Env\cup\fundchr)$ spans $W^*$.
Restricting to $x \in W$, we deduce $x=0$.
\end{proof}

\subsection{}
We can now finish the proof of Theorem \ref{thm: convergence}.
As we have already noted in the discussion following \eqref{eq: naivezeta}, assuming convergence, the integral
\[
a\mapsto\int_{G(F) \bs G (\A)^1} \theta^{\sreg} (ha) \, dh
\]
is $\centder{G}(\A)$-invariant. It therefore remains to prove the convergence of
\[
\int_{\cmpl}\int_{G(F) \bs G (\A)^1} \theta^{\sreg} (ha) \, dh\ e^{-\sprod{\delta_V+\mu}{a}}\ da
\]
where $\cmpl$ is as in Proposition \ref{prop: convregs} and $\mu\in\R_{>0}\fundchr$.

Choose a Siegel set as in \S\ref{sec: Siegel} such that $G(F)\Siegel=G(\A)$.
For any non-negative function $f$ on $G(F)\bs G(\A)^1$ we have
\begin{equation} \label{eq: intsiegel}
\begin{aligned}
\int_{G(F)\bs G(\A)^1}f(g)\ dg&\le
\int_K\int_{\aaa_0^G}\int_{P_0(F)\bs P_0(\A)^1}f(pe^x k)\delta_{P_0}(e^x)^{-1}\tau_{c_0}(x)\ dp\ dx\ dk\\&\le
\int_{\aaa_0^G}\sup_{y\in\Omega} f(e^x y)\delta_{P_0}(e^x)^{-1}\tau_{c_0}(x)\ dx,
\end{aligned}
\end{equation}
where $\tau_{c_0}$ is the characteristic function of $\{x\in\aaa_0^G\mid\sprod{\alpha}x>c_0\ \forall\alpha\in\Delta_0\}$.

It is therefore enough to show the convergence of
\begin{equation} \label{EqnConvergence1}
\int_W\delta_{P_0}(e^x)^{-1}\sup_{y\in\Omega}\abs{\theta_\phi^{\sreg} (e^x y)} \ e^{-\sprod{\delta_V+\mu}{x}}\tau_{c_0}(x)\ dx,
\end{equation}
where $W=\aaa_0^G\oplus\cmpl\hookrightarrow\aaa_0$ and we extend $\tau_{c_0}$ trivially on $\cmpl$.

As in the proof of Proposition \ref{prop: bndtheta} we bound $\theta^{\sreg} (g)$
by splitting the sum according to support sets, i.e., we write
\[
\theta^{\sreg}_\phi(g)=\sum_{U\in\Sub(V)}\Theta_U^{\sreg}(g)
\]
with
\[
\Theta_U^{\sreg}(g)=\sum_{\xi\in \Vreg(F)^{\sreg}: \,\supp\xi=\wgts_U}\phi(\rho(g)^{-1}\xi).
\]
The integral \eqref{EqnConvergence1} is then majorized by the sum over $U\in\Sub(V)$ of
\begin{equation} \label{EqnConvergence2}
\int_W\delta_{P_0}(e^x)^{-1}\sup_{y\in\Omega}\abs{\Theta_U^{\sreg} (e^x y)} \ e^{-\sprod{\delta_V+\mu}{x}}\tau_{c_0}(x)\ dx.
\end{equation}
Moreover, by Lemma \ref{lem: Latticepoints}
for any $N\ge0$ we have
\[
\Theta_U^{\sreg}(g)\ll^{\rd}_{N,\phi}
\delta_{P_0}(m_0(g))\abs{\det\rho(g)}e^{-\sprod{\lambda(U)}{H_0(g)} - N \sum_{\beta\in\wgts_U} \sprod{\beta}{H_0(g)}_-}, \ \ 
g\in T_0(\A)\Omega.
\]
Therefore, \eqref{EqnConvergence2} is majorized by
\begin{equation} \label{eq: auxint12}
\nu_N (\phi)
\int_W e^{-\sprod{\lambda (U)+\mu}{x}}e^{-N \sum_{\beta\in\wgts_U} \sprod{\beta}{x}_-}\tau_{c_0}(x)\ dx
\end{equation}
for suitable seminorms $\nu_N$ on $C_{\rd}(V(\A))$.
We may assume of course that $U \in \Sub_{\sreg} (V)$, for otherwise $\Theta_U^{\sreg}\equiv0$.
By Proposition \ref{prop: convregs}, the restriction of the linear form $\lambda(U)+\mu$ to $W$ is then positive on the cone $\cone_U\cap W\setminus\{0\}$.
Therefore, the integral \eqref{eq: auxint12} converges by Lemma \ref{lem: convergence}.

This finishes the proof of Theorem \ref{thm: convergence}.

\section{{\PVS}s without non-trivial {\SPCL} subspaces} \label{SectionSimpleCase}

In the section, we analyze the case where a regular {\PVS} $V$ has regular $F$-irreducible components and no non-trivial special subspaces.
Under these conditions we obtain a strengthening of our main result, namely meromorphic continuation of the zeta function to the left of the point $\lambda = 0$
with at most simple poles along the hyperplanes defined by the fundamental characters.

We continue to assume that $G$ is reductive.
In general, a subrepresentation of a regular {\PVS} is not necessarily regular.
We will say that a {\PVS} $V$ is \emph{completely factorizable} (\CF) if all its $F$-irreducible subrepresentations are regular. It is clear that every subrepresentation of a {\CF} {\PVS} is again {\CF}.

For example, every {\BASIC} {\PVS} (see Definition \ref{def: basic}) is {\CF}, since any subrepresentation is {\BASIC}, hence regular.
Examples \ref{ex: GL}--\ref{ex: SOeven} are {\CF} precisely when $n_1=\dots=n_k$.
(However, $n_{k+1}$ may be bigger than $n_k$ in cases \ref{ex: Spnodd} and \ref{ex: SOodd}.)

Note that in general, an $F$-irreducible representation $\pi$ of $G$ does not admit a central character.
However, the restriction of $\pi$ to the torus $T_G$ acts by a character, which we denote by $\omega_\pi$.

\begin{lemma}[cf.~\cite{MR2961852}*{Proposition 2.13}] \label{LemmaRubenthaler}
Let $(G,V,\rho)$ be a {\CF} {\PVS} and let $V=\oplus_{i\in I}(V_i,\rho_i)$ be a decomposition into irreducible subrepresentations over $F$.
Let $\chi_i=\omega_{\rho_i}$, $i\in I$. Then,
\begin{enumerate}
\item \label{part: Vreg} $V$ is regular.
\item \label{part: prodreg} $\Vreg=\prod_{i\in I}\Vreg_i$, where $\Vreg_i$ is the regular part of $V_i$.
\item \label{part: fundinv} The fundamental invariant polynomials of $V$ are those of the $V_i$'s.
\item \label{part: fndk} The restrictions of the fundamental characters of $V$ to $T_G$ are $n_i\chi_i$ for some positive integers $n_i$, $i\in I$.
\item \label{part: linind} The characters $\chi_i$ are linearly independent over $\Z$.
\item \label{part: disunion} The union $\wgts_V=\cup\wgts_{V_i}$ is disjoint.
\item \label{part: isotypic} The irreducible components $V_i$ are the isotypic components of $V$.
\item \label{part: dsuminset} The set $\Minset(V)$ consists of the direct sums 
$U=\oplus_{i\in I} U_i$ with $U_i \in \Minset(V_i)$, $i \in I$. 
\item \label{part: ani} The group $\centder{G}/G^{\der}\Ker\rho$ is $F$-anisotropic.
\end{enumerate}
\end{lemma}

\begin{proof}
First note that every $F$-irreducible regular {\PVS} admits a unique fundamental invariant polynomial (over $F$).

Parts \ref{part: Vreg}, \ref{part: prodreg} and \ref{part: fundinv} are contained in \cite{MR2961852}*{Proposition 2.13}.
(Note that part \ref{part: fundinv} over $\bar F$ easily implies it over $F$.)
Part \ref{part: fndk} immediately follows.
Part \ref{part: linind} follows since the fundamental characters of a {\PVS} are always linearly independent.
Parts \ref{part: disunion}, \ref{part: isotypic} and \ref{part: dsuminset} are easy consequences of part \ref{part: linind}. 

Finally, part \ref{part: ani} is equivalent to the statement that $T_G\cap\Ker\rho$ is of finite index in $T_G\cap\centder{G}$.
Note that $T_G\cap\Ker\rho$ is the intersection of $\Ker\chi_i$, $i\in I$ and by part \ref{part: fndk}, this is of finite index in the common kernel
of the restrictions of the characters in $X^*(G/\centder{G})$ to $T_G$. The lemma follows.
\end{proof}

The following easy consequence about the special subspaces of $V$ and its subrepresentations will be used repeatedly.

\begin{lemma} \label{LemmaSpecialFactor}
Let $V$ be a {\CF} {\PVS} and $V'$ a subrepresentation of $V$ (defined over $F$).
\begin{enumerate} 
\item
Let $V''$ be the (unique) complement to $V'$ in $V$. Then the map
$U \mapsto U \oplus V''$ defines an injection of $\Spcl (V')$ into $\Spcl (V)$.
\item If $\Spcl(V)=\{V\}$, then $\Spcl(V')=\{V'\}$. 
\end{enumerate}
\end{lemma}

\begin{proof} By Lemma \ref{LemmaRubenthaler} part 7, the map 
$U \mapsto U \oplus V''$ defines an injection of $\Minset (V')$ into $\Minset (V)$. The first part follows now from the definition of a special subspace. The second part is clear.
\end{proof}

Let $\swrz(V(\A))$ be the space of Schwartz--Bruhat functions on $V(\A)$.
As a locally convex topological vector space, $\swrz(V(\A))$ can be identified with $\swrz(V(F_\infty))\otimes\swrz(V(\A_{\fin}))$ where $\swrz(V(F_\infty))$ is the usual Schwartz space
and $\swrz(V(\A_{\fin}))$ is the countable-dimensional space of locally constant compactly supported functions on $V(\A_{\fin})$ with the finite topology.
In other words, $\swrz(V(\A_{\fin}))=\varinjlim U$ with the inductive limit topology, where $U$ ranges over the finite-dimensional subspaces of $\swrz(V(\A_{\fin}))$
with their canonical Hausdorff topology.
Equivalently, $\swrz(V(\A))=\varinjlim\swrz(V(F_\infty))\otimes U$ with the locally convex inductive limit topology where $U$ is as before.
The space $\swrz(V(\A))$ is Hausdorff, complete and nuclear.
The embedding $\swrz(V(\A))\rightarrow C_{\rd}(V(\A))$ is continuous.

Assume that $(G,V,\rho)$ is {\CF} with irreducible components $(\rho_i,V_i)$, $i\in I$ over $F$ and let $\chi_i=\omega_{\rho_i}$, $i\in I$.
For any $J\subset I$ let $V^J=\oplus_{j\in J}V_j$ and $V_J=V/V^J=\oplus_{j\notin J}\rho_j$. Let $\rho_J$ be the representation of $G$ on $V_J$.
These are again {\CF} {\PVS}s. 

As usual, let $A_G = T_G (\R)^\circ$ be the complement to $G(\A)^1$ in $G (\A)$ obtained by embedding $\R$ into $F_\infty = \R \otimes F$ via $x \mapsto x \otimes 1$, and let $A_{\Ker \rho} = A_G \cap \Ker \rho$.
Since $\centder{G}/G^{\der}\Ker\rho$ is $F$-anisotropic, we have $G(\A)^1\centder{G}(\A)\bs G(\A) \simeq A_{\Ker\rho} \bs A_G$ and we can write the zeta function of $\phi\in\swrz(V(\A))$ as
\[
Z(\phi,\lambda)=\int_{G(F) A_{\Ker\rho} \bs G (\A)} \theta_\phi^{\sreg} (g) g^{-\lambda}\abs{\det\rho(g)}^{-1} \, dg
\]
for $\lambda\in(\aaa_{G,\C}/\aaa_{\Ker \rho,\C})^* = \C \fundchr$. 
We know from Theorem \ref{thm: convergence} that the integral converges absolutely for $\Re\lambda\in\R_{>0}\fundchr$.
In the following, we will use the isomorphism
$A_{\Ker\rho} \bs A_G \simeq \R^I, a \mapsto (\log \abs{\chi_i (a)})_{i \in I}$
and the Lebesgue measure on the vector space $\R^I$ to normalize the measure on the quotient $G(F) A_{\Ker\rho} \bs G (\A)$.

For any $\phi\in\swrz(V(\A))$, let $\phi_J\in\swrz(V_J(\A))$ be given by
\[
\phi_J(v)=\int_{V^J(\A)}\phi(u+v)\ du.
\]
For $\lambda=\sum_{i\in I}\lambda_i\chi_i$ we write $\lambda_J=\sum_{j\notin J}\lambda_j\chi_j$. We will consider the zeta functions $Z^{V_J}(\phi_J,\lambda_J)$ simultaneously.
In the limiting case $J = I$, where $V_I = 0$, the term $Z^{V_I}(\phi_I,\lambda_I)$ is of course interpreted as the constant $\vol (G(F) \bs G (\A)^1) \hat\phi(0)$.

The main result of this section is the following.

\begin{theorem} \label{thm: smplcase}
Let $(G,V,\rho)$ be a {\CF} {\PVS} such that $\Spcl(V)=\{V\}$. 
Then, there exists $c >0$ such that for any $\phi\in\swrz(V(\A))$ the zeta function $Z(\phi,\lambda)$ admits a meromorphic continuation to the region
\[
\Domain=\{\lambda=\sum_{i\in I}\lambda_i\chi_i\mid\Re\lambda_i>-c\},
\]
and $Z (\phi, \lambda) \prod_{i\in I}\lambda_i$ is holomorphic in $\Domain$.
Moreover, with the notation above,
\begin{equation} \label{eq: diffZ}
\sum_{J\subset I}(-1)^{\# J}\frac{Z^{V_J}(\phi_J,\lambda_J)}{\prod_{j\in J}\lambda_j}
\end{equation}
is holomorphic in $\Domain$.
\end{theorem}

Theorem \ref{thm: smplcase} will be proved in the remainder of this section.
Applying it to $V_J$ for every $J\subset I$ and using inclusion-exclusion we obtain

\begin{corollary} 
Under the above conditions, for every subset $J\subset I$ the function
\[
f_J (\lambda_J) = \sum_{K\subset I \setminus J}(-1)^{\# K}\frac{Z^{V_{J \cup K}}(\phi_{J \cup K},\lambda_{J \cup K})}{\prod_{j\in K}\lambda_j}
\]
is holomorphic in $\{\lambda_J=\sum_{j\notin J}\lambda_j\chi_j\mid\Re\lambda_j>-c\}$ and we have
\[
Z (\phi, \lambda) = \sum_{J\subset I}\frac{f_J (\lambda_J)}{\prod_{j\in J}\lambda_j}
\]
in $\Domain$.
\end{corollary}

Theorem \ref{thm: smplcase} and its corollary apply in particular to any {\BASIC} {\PVS}, as well as to the ``magic square'' examples of \S\ref{sec: magic square} --
see Remark \ref{rem: basic}. Recall that the {\BASIC} {\PVS}s obtained from nilpotent orbits for the classical groups are the following:
\begin{enumerate}
\item $\GL_n^k$ acting on $\Mat_n^{k-1}$ (Example \ref{ex: GL}).
\item $\GL_n^k \times \Sp_{n+m}$ acting on $\Mat_n^{k-1} \oplus \Mat_{n,n+m}$ for $n$, $m$ even (Example \ref{ex: Spnodd}).
\item $\GL_n^k$ acting on $\Mat_n^{k-1} \oplus \Sym_n$ for $n \neq 2$ (Example \ref{ex: Speven}).
\item $\GL_n^k \times \SO_{n+m}$ acting on $\Mat_n^{k-1} \oplus \Mat_{n,n+m}$ for $n$, $m \neq 2$ (Example \ref{ex: SOodd}).
\item $\GL_n^k$ acting on $\Mat_n^{k-1} \oplus \Skew_n$ (Example \ref{ex: SOeven}).
\end{enumerate}

For $F$-irreducible $V$ such that $\Spcl(V)=\{V\}$, we obtain that the zeta function has a simple pole at $\lambda = 0$ with residue $\vol (G(F) \bs G (\A)^1) \hat\phi(0)$,
as in Tate's thesis (cf. Proposition \ref{prop: irrcase} below). Example \ref{ex: G2} (treated in \cite{MR776169}, building on \cite{MR0289428}) and
Example \ref{ex: F4} (treated in \cite{MR1267735}) show that without the assumption $\Spcl(V)=\{V\}$, the structure of the pole of the zeta function at $\lambda = 0$ may be more complicated.

Note that the assumption $\Spcl(V)=\{V\}$ implies (by Lemma \ref{LemmaSpecialFactor} and Proposition \ref{prop: stronglyregular})
that all $V_J$ are {\sr}, i.e. that $\theta_{\phi_J}^{\sreg} (g) = \theta_{\phi_J} (g)$. (In our proof below we will simply work with $\theta_{\phi_J} (g)$ throughout.)

We first consider the case where $\rho$ is irreducible over $F$.
For simplicity we write $x_{\le1}=\min(x,1)$ and $x_{\ge1}=\max(1,x)$ for any $x\in\R_{>0}$.
Denote by $\one_A$ the characteristic function of a set $A$.

For brevity, we will use the notation $A(\cdot,\phi)\ll_{N,\phi} B_N(\cdot)$ to mean that for every $N\ge0$
there exists a continuous seminorm $\nu_N$ on $\swrz(V(\A))$ such that $\abs{A(\cdot,\phi)}\le B_N(\cdot)\nu_N(\phi)$ for every $\phi\in\swrz(V(\A))$.
Recall that we have already used the notation $\ll_{N,\phi}^{\rd}$ in \S\ref{sec: llNphi} with respect to the space $C_{\rd}(V(\A))$.
Clearly $\ll_{N,\phi}^{\rd}$ implies $\ll_{N,\phi}$.

\begin{proposition} \label{prop: irrcase}
Suppose that $(G,V,\rho)$ is an $F$-irreducible {\PVS} such that $\Spcl(V)=\{V\}$ and let $\chi=\omega_\rho$. For any $\phi\in\swrz(V(\A))$ let
\[
\gimel_\phi(g)=\abs{\det\rho(g)}^{-1}\theta_\phi(g)-\one_{\R_{>1}}(\abs{\chi(g)})\hat\phi(0),\ \ g\in G(F)\bs G(\A).
\]
Let $\beta_0$ be the lowest weight of $V$. Then, 
\begin{equation} \label{eq: bndif}
\gimel_\phi(g)\ll_{N,\phi}\delta_{P_0}(m_0(g))e^{-\sprod{\mu}{H_0(g)}-\sprod{\beta_0}{H_0(g)}_+-N\sprod{\chi}{H_0(g)}_-},\ \ g\in\Siegel
\end{equation}
for a suitable $\mu\in\R_{>0}(\Delta_0)$.
Therefore, there exists $c>0$ such that the integral
\begin{equation} \label{eq: convf}
\int_{G(F) A_{\Ker \rho} \bs G(\A)}\gimel_\phi(g)\abs{\chi(g)}^{-s} \, dg
= Z (\phi, s \chi) - \vol (G(F) \bs G (\A)^1) \frac{\hat\phi(0)}{s}
\end{equation}
converges absolutely for $\Re s>-c$.
\end{proposition}

\begin{proof}
We will estimate $\gimel_\phi$ using the decomposition \eqref{eq: thetadecomp} of
$\theta_\phi(g)$. We write
\[
\gimel_\phi (g) = \abs{\det\rho(g)}^{-1} \sum_{U\in\Minset(V), \, U \neq V} \Theta_U(g) + 
\abs{\det\rho(g)}^{-1}\Theta_V (g)-\one_{\R_{>1}}(\abs{\chi(g)})\hat\phi(0)
\]
and separately bound the summands $\Theta_U$, $U\in\Minset(V)$, $U\ne V$ and $\abs{\det\rho}^{-1}\Theta_V-\one_{\R_{>1}}(\abs{\chi})\hat\phi(0)$ on $\Siegel$.

Consider the contribution from $U\in\Minset(V)$, $U\ne V$. 
By Lemma \ref{lem: Latticepoints} we have
\[
\abs{\det\rho(g)}^{-1}\Theta_U(g)\ll_{N,\phi}^{\rd}\delta_{P_0}(m_0(g))e^{-\sprod{\lambda(U)}{H_0(g)} -N\sum_{\beta\in\wgts_U}\sprod{\beta}{H_0(g)}_-},\ \ g\in\Siegel.
\]
Recall that
$\lambda(U)=\delta_0+\sum_{\beta\in\wgts_V\setminus\wgts_U}n_\beta\beta \in \R_{\ge0}(\wgts_U\cup\Delta_0)$ by \eqref{def: lambda(U)}
and \eqref{eq: lambdaUincone}.
By Corollary \ref{cor: spcsets} part \ref{part: spcl} and our assumption on $V$ we have $\Env(U)=G$.
Thus, the {\PS} of $\lambda(U)$ with respect to $\Delta_0\cup\wgts_U$ (see \S\ref{sec: PS}) contains $\Delta_0$ and we have
\[
\lambda(U)=\sum_{\beta\in\wgts_U}c_\beta\beta+\sum_{\alpha\in\Delta_0}d_\alpha\alpha
\]
with suitable coefficients $c_\beta\ge0$ and $d_\alpha > 0$. Moreover, by considering $\sprod{\lambda(U)}{H}$ for $H \in \aaa_G$ we see that
$\sum_{\beta\in\wgts_U}c_\beta = \sum_{\beta\in\wgts_V\setminus\wgts_U}n_\beta \ge 1$ since $U\ne V$. 
Using that on the Siegel domain $\sprod{\beta_0}{H_0(g)}-\sprod{\beta}{H_0(g)}$ is bounded from above for all $\beta\in\rts_{V}$ we obtain
\[
\abs{\det\rho(g)}^{-1}\Theta_U(g)\ll_{N,\phi}^{\rd}\delta_{P_0}(m_0(g))e^{-\sprod{\mu}{H_0(g)} - \sprod{\beta_0}{H_0(g)}_+ - N\sum_{\beta\in\wgts_U}\sprod{\beta}{H_0(g)}_-},\ \ g\in\Siegel.
\]
By Corollary \ref{cor: fundchrS} part \ref{part: nonneg}, $\chi \neq 0$ is a non-negative linear combination of $\wgts_U$. Therefore we can replace here
$\sum_{\beta\in\wgts_U}\sprod{\beta}{H_0(g)}_-$ by a suitable positive multiple of 
$\sprod{\chi}{H_0(g)}_-$, proving the required bound for the contribution of any $U \neq V$.

As for $U=V$, we claim that
\begin{equation} \label{eq: estUV2}
\abs{\det\rho(g)}^{-1}\Theta_V(g)-\one_{\R_{>1}}(\abs{\chi(g)})\hat\phi(0)\ll_{N,\phi}e^{-\sprod{\beta_0}{H_0(g)}_+-N\sprod{\chi}{H_0(g)}_-},\ \ g\in\Siegel.
\end{equation}
We may assume that $g\in T_0(\A)\cap\Siegel$, or even that $g\in T_0(\R)\cap\Siegel$.
We split into cases according to the size of $\sprod{\beta_0}{H_0(g)}$. Fix a constant $c_1$.

If $\sprod{\beta_0}{H_0(g)}\le c_1$, then \eqref{eq: estUV2} holds because
$\abs{\det\rho(g)}^{-1}\Theta_V(g)\ll_{N,\phi}^{\rd}e^{-N\sprod{\chi}{H_0(g)}_-}$
by Lemma \ref{lem: Latticepoints}.

Assume that $\sprod{\beta_0}{H_0(g)}>c_1$. Then $\sprod{\beta}{H_0(g)}$ is bounded away from $0$ for every $\beta\in\wgts_V$, since $g\in\Siegel$.
By a standard argument, e.g., using the Poisson summation formula, we have
\begin{equation} \label{eq: fullthetaest}
\abs{\det\rho(g)}^{-1}\sum_{v\in V(F)}\phi(\rho(g)^{-1}v)-\hat\phi(0)\ll_\phi e^{-\sprod{\beta_0}{H_0(g)}}.
\end{equation}
(Here we cannot replace $\ll_\phi$ by $\ll_\phi^{\rd}$.) On the other hand, 
\[
\abs{\Theta_V(g)-\sum_{v\in V(F)}\phi(\rho(g)^{-1}v)}\le\sum_{v\in V(F)\mid P(v)=0}\abs{\phi(\rho(g)^{-1}v)}
\]
for some non-zero polynomial $P$ on $V$. (For instance, we can take the relative invariant times the product over $\beta\in\wgts_V$
of arbitrary nonzero linear forms on $V_\beta$.) It follows from \cite{MR3534542}*{Lemma 6.4} that
\begin{equation} \label{eq: thetaminfull}
\abs{\det\rho(g)}^{-1}(\Theta_V(g)-\sum_{v\in V(F)}\phi(\rho(g)^{-1}v))\ll_\phi e^{-\sprod{\beta_0}{H_0(g)}}.
\end{equation}
(In fact, the argument gives the stronger assertion with $\ll_\phi^{\rd}$.)

Since $\sprod{\chi}{H_0(g)}-\sprod{\beta_0}{H_0(g)}$ is bounded from below for $g\in\Siegel$,
we may choose $c_1$ such that $\abs{\chi(g)}>1$ whenever $\sprod{\beta_0}{H_0(g)}>c_1$.
Thus, \eqref{eq: estUV2} for $\sprod{\beta_0}{H_0(g)}>c_1$ follows from \eqref{eq: fullthetaest} and \eqref{eq: thetaminfull}. This finishes the proof of \eqref{eq: bndif}.

The convergence of
\[
\int_{P_0 (F) A_{\Ker \rho} \bs \Siegel}\abs{\gimel_\phi(g)} \abs{\chi(g)}^{-\sigma} \, dg
\]
for $\sigma >-c$ for a suitable $c>0$ follows from \eqref{eq: bndif} and Lemma \ref{lem: convergence}, since for any complement $W$ of $\aaa_{\Ker \rho}$ in $\aaa_0$ the
positively-$1$-homogeneous function $\sprod{\mu}{x}+\sprod{\beta_0}{x}_+$ on $W$ is positive on the cone $\{x\in\aaa_{0,+} \cap W \mid\sprod{\chi}{x}\ge0\}\setminus\{0\}$. 

This implies the absolute convergence of \eqref{eq: convf} for $\Re s >-c$.
\end{proof}

\begin{proof}[Proof of Theorem \ref{thm: smplcase}]
By induction on the cardinality of $I \setminus J$, starting with the trivial case $J=I$ (and, if we like, Proposition
\ref{prop: irrcase} for $\# (I \setminus J) = 1$), we may assume holomorphic 
continuation of the functions $Z^{V_J}(\phi_J,\lambda_J)
\prod_{j\notin J}\lambda_j$ 
to $\{\lambda_J=\sum_{j\notin J}\lambda_j\chi_j\mid\Re\lambda_j>-c\}$ 
for all $J\neq\emptyset$. To prove the theorem, we need to show that \eqref{eq: diffZ} is holomorphic on $\Domain$.

For every $i\in I$ let $\beta_i$ be the lowest weight of $V_i$.
For any $J\subset I$ let $\chi_J=\prod_{j\notin J}\chi_j$ and
let $G^J$ be the subset of $G(\A)$ defined by the conditions $\abs{\chi_j(g)}>1$ for all $j\in J$.

Consider
\[
\daleth_\phi(g)=\sum_{J\subset I}(-1)^{\# J}\abs{\det\rho_J(g)}^{-1}\theta_{\phi_J}^{V_J}(g)\one_{G^J}(g).
\]
We claim that there exists $\mu\in\R_{>0} \Delta_0$ such that
\begin{equation} \label{eq: mainbnd2}
\daleth_\phi(g)\ll_{N,\phi}
\delta_{P_0}(m_0(g))e^{-\sprod{\mu}{H_0(g)}-\sum_{i\in I}\sprod{\beta_i}{H_0(g)}_+-N\sum_{i\in I}\sprod{\chi_i}{H_0(g)}_-},\ \ g\in\Siegel.
\end{equation}

We may identify $\swrz(V(\A))$ with the completed tensor product $\hat\otimes_{i\in I}\swrz(V_i(\A))$ with respect to, say, the projective tensor product topology (the spaces $\swrz(V_i(\A))$ are nuclear).
Therefore, it is enough to prove \eqref{eq: mainbnd2} for $\phi$ of the form $\phi=\otimes\psi_i$ where $\psi_i\in\swrz(V_i(\A))$.
We then have
\[
\daleth_\phi(g)=\prod_{i\in I}\big(\abs{\det\rho_i(g)}^{-1}\theta_{\psi_i}^{V_i}(g)-\one_{\R_{>1}}(\abs{\chi_i(g)})\hat\psi_i(0)\big).
\]
If we simply apply \eqref{eq: bndif} to each factor, the bound that we get will not be good enough.
This is not surprising since we have not yet used that $V$ itself is a {\PVS}, only the individual factors.
Instead, we write for each $i\in I$
\[
\theta_{\psi_i}^{V_i}=\Theta_{\psi_i}^{V_i,\fs}+\Theta_{\psi_i}^{V_i,\nfs}
\]
where 
\[
\Theta_{\psi_i}^{V_i,\fs} = \sum_{\xi\in \Vreg_i (F):\,\supp\xi=\wgts_{V_i}}\phi(\rho_i(g)^{-1}\xi)
\]
is the summand corresponding to $U=V_i$ (full support) in the decomposition \eqref{eq: thetadecomp} for $\theta_{\psi_i}^{V_i}$.
Then, we can decompose
\begin{equation} \label{eq: decompdal}
\daleth_\phi=\sum_{J\subset I}\prod_{j\in J}\abs{\det\rho_j}^{-1}\Theta_{\psi_j}^{V_j,\nfs}\prod_{j\notin J}
\big(\abs{\det\rho_j}^{-1}\Theta_{\psi_j}^{V_j,\fs}-\one_{\R_{>1}}(\abs{\chi_j})\hat\psi_j(0)\big).
\end{equation}
For any subset $J\subset I$ let $\phi^J=\otimes_{j\in J}\psi_j$. Then, 
\begin{equation} \label{eq: nfs}
\prod_{j\in J}\Theta_{\psi_j}^{V_j,\nfs} = \sum_{U \in\Minset(V^J), \, U \not\supset V_j \, \forall j\in J} \Theta^{V^J}_{\phi^J;U}
\end{equation}
is the contribution to $\Theta^{V^J}_{\phi^J}$
with respect to the decomposition \eqref{eq: thetadecomp} of the subspaces $U$ not containing any $V_j$, $j \in J$.
As in the proof of Proposition \ref{prop: irrcase}, there exists $\mu\in\R_{>0}\Delta_0$ (which we may take to be independent of $J$) such that
\begin{equation} \label{eq: newbndUprop}
\prod_{j\in J}\abs{\det\rho_j(g)}^{-1}\Theta_{\psi_j}^{V_j,\nfs}(g)\ll_{N,\phi^J}\delta_{P_0}(m_0(g))
e^{-\sprod{\mu}{H_0(g)}-\sum_{j\in J}\sprod{\beta_j}{H_0(g)}_+-N\sum_{j\in J}\sprod{\chi_j}{H_0(g)}_-}
\end{equation}
on $\Siegel$.
Indeed, this estimate holds for each individual contribution
$\Theta^{V^J}_{\phi^J;U}$ to \eqref{eq: nfs}.
Since $\Env(U)=G$ by Lemma \ref{LemmaSpecialFactor} and Corollary \ref{cor: spcsets} part \ref{part: spcl},
the {\PS} of $\lambda(U)$ with respect to $\wgts_U\cup\Delta_0$ contains $\Delta_0$. Writing
\[
\lambda(U)=\sum_{\beta\in\wgts_U}c_\beta\beta+\sum_{\alpha\in\Delta_0}d_\alpha\alpha
\]
with $c_\beta\ge0$ and $d_\alpha > 0$, and considering $\sprod{\lambda(U)}{H}$ for $H \in \aaa_G$, we obtain
\[
\sum_{j \in J} \sum_{\beta\in\wgts_U \cap \wgts_{V_j}} c_\beta \chi_j = \sum_{j \in J} \sum_{\beta\in\wgts_{V_j} \setminus\wgts_U}n_\beta \chi_j, 
\]
and therefore
$\sum_{\beta\in\wgts_U \cap \wgts_{V_j}} c_\beta = \sum_{\beta\in\wgts_{V_j} \setminus\wgts_U}n_\beta \ge 1$ for all $j \in J$,
since the $\chi_j$ are linearly independent over $\Z$ and
$U\not\supset V_j$ for all $j\in J$.

Using Lemma \ref{lem: Latticepoints} together with Corollary \ref{cor: fundchrS} part \ref{part: nonneg} and the fact that
$\sprod{\beta_j}{H_0(g)}-\sprod{\beta}{H_0(g)}$ is bounded from above for any $\beta\in\rts_{V_j}$ and $g\in\Siegel$, we conclude \eqref{eq: newbndUprop} as in the irreducible case.

Combining \eqref{eq: newbndUprop} with the estimate \eqref{eq: estUV2} for every $j\notin J$ and using the decomposition \eqref{eq: decompdal}
we obtain the required estimate \eqref{eq: mainbnd2}.

Note that the expression \eqref{eq: diffZ} is equal to the integral 
\[
\int_{A_{\Ker \rho} G(F)\bs G(\A)}\daleth_\phi(g)e^{-\sprod{\lambda}{H_0(g)}}\ dg.
\]
Therefore, Theorem \ref{thm: smplcase} follows from \eqref{eq: mainbnd2} and Lemma \ref{lem: convergence} by noting that 
for any complement $W$ of $\aaa_{\Ker \rho}$ in $\aaa_0$, the positively-$1$-homogeneous function
$\sprod{\mu}{x}+\sum_{i\in I}\sprod{\beta_i}{x}_+$ on $W$ is positive on the cone $\{x\in\aaa_{0,+} \cap W \mid\sprod{\chi_i}{x}\ge0\ \forall i\in I\}\setminus\{0\}$.
\end{proof}

\begin{remark}
Let $V$ be a regular {\PVS} with $\Spcl(V)=\{V\}$. The example of $V=\Mat_n$ with $G=\GL_n$ acting by matrix multiplication shows that $V$ is not necessarily {\CF}.
One may ask whether disjointness of the sets of weights $\wgts_{V_i}$ of the irreducible components $V_i$ of $V$ is sufficient to imply that $V$ is {\CF}.
This would follow if one could show that an $F$-irreducible {\PVS} $V$ such that $\Spcl(V)=\{V\}$ is necessarily regular.\footnote{This can be easily
checked when $F$ is algebraically closed by classification.}

In the case of {\PVS}s $V$ of DK-type, the construction of \S\ref{sec: examspcl} below shows that the condition $\Spcl(V)=\{V\}$ implies that $V$ is {\CF}. 
\end{remark}

\section{A construction of {\SPCL} subspaces for {\PVS}s of DK-type} \label{sec: examspcl}

We finish the paper by providing a general construction of {\SPCL} subspaces for regular {\PVS}s coming from nilpotent orbits
(see \S\ref{sec: nilPVS}).

Let $G'$ be a reductive group over $F$.
Recall that any nilpotent $e\in\Liealg{G}'$ gives rise to a filtration $\filt(e)=(\filt_i(e))_{i\in\Z}$ of $\Liealg{G}'$
whose stabilizer $\Stab(\filt(e))$ in $G'$ is a parabolic subgroup of $G'$ with Lie algebra $\filt_0(e)$.

Clearly, $\filt(\Ad(g)e)=\Ad(g)\filt(e)$ for all $g\in G'$.
In particular, $e$ and $\Ad(g)e$ define the same filtration if and only if $g\in\Stab(\filt(e))$.

If $\Lorb$ is the nilpotent orbit of $e$, then we say that $\filt(e)$ is an $\Lorb$-filtration. Thus, the variety of $\Lorb$-filtrations
is a $G'$-orbit whose stabilizer is the canonical parabolic subgroup pertaining to $\Lorb$.

\subsection{}
We recall the notion of induction of nilpotent orbits \cite{MR527733}.

Let $Q$ be a parabolic $F$-subgroup of $G'$ with Lie algebra $\Liealg{Q}$.
Let $L$ be the Levi quotient of $Q$, with Lie algebra $\Liealg{L}$.
Let $\pr_Q:Q\rightarrow L$ (resp., $\pr_{\Liealg{Q}}:\Liealg{Q}\rightarrow\Liealg{L}$) be the canonical projections.
Let $\Lorb^L$ be a nilpotent orbit in $\Liealg{L}$ and let $\tilde\Lorb^L$ be its inverse image under $\pr_{\Liealg{Q}}$.
It is an irreducible subvariety of $\Liealg{Q} \subset \Liealg{G}'$ consisting of nilpotent elements.
By definition, the induced orbit $\Ind_Q\Lorb^L$ is the nilpotent orbit in $\Liealg{G}'$ intersecting $\tilde\Lorb^L$ in an open subset.
In fact, it depends only on $L$ (and $\Lorb^L$), and not on $Q$, but it will be important for us to keep track of $Q$.

\begin{definition}
Suppose that $\Lorb=\Ind_Q\Lorb^L$ where
$Q$ is a parabolic $F$-subgroup of $G'$ with Levi quotient $L$ and $\Lorb^L$ is a nilpotent orbit of $\Liealg{L}$.
Let $\filt^L$ be an $\Lorb^L$-filtration of $\Liealg{L}$.
We then say that the triple $\data=(Q,\Lorb^L,\filt^L)$ is an \emph{induced filtration data} (\IFD) for $\Lorb$.
\end{definition}

For instance, $(G',\Lorb,\filt)$ is a (trivial) {\IFD} for any $\Lorb$-filtration $\filt$ of $\Liealg{G}'$.

The variety $\Indata=\Indata_{\Lorb}$ of IFDs for $\Lorb$ comprises finitely many $G'$-orbits.
(If $\Indata_{\Lorb}$ is a single orbit, i.e., if $\Lorb$ is not induced non-trivially, then $\Lorb$ is called \emph{rigid}.)

\begin{lemma} \label{LemmaIfiltration}
Let $\data=(Q,\Lorb^L,\filt^L)\in\Indata$. Define inductively
\[
\tilde\filt_i=\begin{cases}\pr_{\Liealg{Q}}^{-1}(\filt^L_i),&i=0,1,2,\\
[\tilde\filt_1,\tilde\filt_{i-1}] + [\tilde\filt_2,\tilde\filt_{i-2}],&i>2,\end{cases}
\]
and
\[
\tilde\filt_i=\{ x \in \Liealg{G}' \mid [x, \tilde\filt_{-i}] \subset \tilde\filt_0 \},\ \ i < 0.
\]
Then $\tilde\filt = (\tilde\filt_i)_i$ is a filtration of $\Liealg{G}'$ of V-type.
\end{lemma}

Note that $\pr_{\Liealg{Q}} (\tilde\filt_i) = \filt^L_i$ for all $i$ but that the preimage $\pr_{\Liealg{Q}}^{-1}(\filt^L_i)$ strictly contains
$\tilde\filt_i$ for $i>2$.

\begin{proof}
Without loss of generality, we may assume that $Q$ is a standard parabolic subgroup of $G'$ and $\filt^L$ the standard $\Lorb^L$-filtration defined by a dominant element $h^L \in \aaa_0^L$, i.e. 
$\filt^L_i = \bigoplus_{\alpha \in \rts_{L} \cup \{ 0 \}: \, \sprod{\alpha}{h^L} \ge i} \Liealg{G}'_{\alpha}$.

Let $h = h^L + h_L \in \aaa_0^{G'}$ be the unique vector with $h_L \in \aaa^{G'}_L$ and $\sprod{\alpha}{h} = 2$ for all $\alpha \in \Delta_0^{G'} \setminus \Delta_0^Q$.
We claim that the inductively defined spaces $\tilde\filt_i$ coincide with the filtration obtained from the gradation of $\Liealg{G}'$ defined by $h$:
$\tilde\filt_i = \bigoplus_{\alpha \in \rts_{G'} \cup \{ 0 \}: \, \sprod{\alpha}{h} \ge i} \Liealg{G}'_{\alpha}$.

To see this, note that we have $\sprod{\alpha}{h} = \sprod{\alpha}{h^L} \in \{ 0, 1, 2 \}$ for $\alpha \in \Delta_0^Q$ (and therefore $\sprod{\alpha}{h} \in \{ 0, 1, 2 \}$ for all 
$\alpha \in \Delta_0^{G'}$). This implies the claim for $i = 0$, $1$, $2$. The claim for $i \ge 3$ follows,
since any positive root $\alpha \in \rts_{G'}$ with $\sprod{\alpha}{h} \ge 3$ can be written as 
$\alpha = \beta + \gamma$ with positive roots $\beta$ and $\gamma$ and
$\sprod{\beta}{h}  \in \{ 1, 2 \}$.

The relation between $\tilde\filt_i$ and $\tilde\filt_{-i}$ holds for any filtration obtained from a gradation.
\end{proof}

\begin{definition}
We call the filtration $(\tilde\filt_i)_i$ of Lemma \ref{LemmaIfiltration} the \emph{I-filtration} associated to $\data$. 

We say that a filtration $\tilde\filt$ of $\Liealg{G}'$ is an \emph{I-filtration compatible with
$\Lorb$} if it arises from some $\data \in \Indata$ by this construction. 
\end{definition}

The stabilizer $\Stab(\data)$ of $\data=(Q,\Lorb^L,\filt^L)\in\Indata$ is the parabolic subgroup of $Q$ whose image in $L$ is $\Stab_L (\filt^L)$. It is also the stabilizer
of the associated I-filtration $\tilde\filt$.
Its Lie algebra is $\tilde\filt_0$.

$\Lorb$-filtrations are I-filtrations compatible with $\Lorb$, but in general there will be others.
For example, if $\Lorb=\Rich(Q)$, there is an I-filtration compatible with $\Lorb$ with $\tilde\filt_{2} = \Liealg{N}_Q$, which is an $\Lorb$-filtration
only if $\Lorb$ is even and $Q$ is conjugate to the canonical parabolic subgroup of $\Lorb$.
On the other hand, different {\IFD}s may give rise to the same I-filtration.

\begin{lemma} \label{lem: basicind}
Let $\data=(Q,\Lorb^L,\filt^L)$ be an {\IFD} for $\Lorb$. Let $P'=\Stab(\data)$ and $(\tilde\filt_i)_i$ be the associated I-filtration. Then,
\begin{enumerate}
\item $\Lorb_{\tilde\filt} = \Lorb\cap\tilde\filt_2$ is open in $\tilde\filt_2$.
\item $\pr_{\Liealg{Q}}(\Lorb_{\tilde\filt}) = \Lorb^L\cap\filt^L_2$.
\item \label{part: projind} $\Lorb_{\tilde\filt}$ is a $P'$-orbit.
\item \label{part: moveind} Let $v\in\Lorb_{\tilde\filt}$. Then,
\[
\{g\in G'\mid\Ad(g)^{-1}v\in\Lorb_{\tilde\filt}\}=G'_v P'
\]
and $(G'_v)^\circ\subset P'$.
\end{enumerate}
\end{lemma}

\begin{proof}
For simplicity we sometimes write $\bar{\phantom{x}}$ for the image under $\pr_Q$ or $\pr_{\Liealg{Q}}$.

The first part follows since $\pr_{\Liealg{Q}}^{-1}(u)\cap\Lorb$ is open in $\pr_{\Liealg{Q}}^{-1}(u)$ for every $u\in\Lorb^L$ and
$\Lorb^L\cap\filt^L_2$ is open in $\filt_2^L$.

For the second part, suppose that $u\in\Liealg{Q}$ with $\bar u\in\filt_2^L\setminus\Lorb^L$.
Then, $\bar u$ lies in a nilpotent orbit $\Lorb'\ne\Lorb^L$ of $\Liealg{L}$ in the closure of $\Lorb^L$
and $u$ lies in the closure of $\Ind_Q\Lorb'$. The latter has smaller dimension than $\Lorb$.
Thus, $u\notin\Lorb$.

For the last part, suppose that $u=\Ad(g)^{-1}v\in\Lorb_{\tilde\filt}$. Then, by the second part
$\bar u\in\Lorb^L$. By \cite{MR527733}*{Theorem 1.3(d)}, $g\in G'_vQ$.
Write $g=hq$ where $h\in G'_v$, $q\in Q$. Then, $\bar u=\Ad(\bar{q})^{-1}\bar v$. Hence, $\Ad(\bar q)^{-1}\bar v\in\Lorb^L \cap
\filt^L_2$.
By Theorem \ref{thm: stdunip} part \ref{part: moveo} (applied to $\Lorb^L$) we infer that $\bar q\in\pr_Q(P')$, so that $q\in P'$. Hence $g\in G'_v P'$.

Also, $(G'_v)^\circ\subset Q$ again by \cite{MR527733}*{Theorem 1.3(d)}, since $\bar v\in\Lorb^L$ by the second part.
Thus, $(G'_v)^\circ=Q_v^\circ$. Since $\pr_Q(Q_v)\subset L_{\bar v}\subset\pr_Q(P')$, by Theorem \ref{thm: stdunip}
we deduce that $(G'_v)^\circ\subset P'$.

Finally, the third part follows from the last part.
\end{proof}

\subsection{}

We fix $P_0'$, $T_0'$ as in \S\ref{sec: nilPVS}.
The following standard fact is an easy consequence of the Bruhat decomposition.
\begin{remark} \label{rem: unqstd}
Let $P$ be a standard parabolic $F$-subgroup of $G'$ with standard Levi subgroup $M$.
Then, for any parabolic $F$-subgroup $Q$ of $G'$ there exists a unique parabolic subgroup $Q'$
of the form $Q'=Q^p$, $p\in P(F)$ such that $Q'\cap M$ is a standard parabolic subgroup of $M$.
\end{remark}

Namely, if $Q$ is conjugate to the standard parabolic $F$-subgroup $Q''$ of $G$, then we can write $Q = (Q'')^{n p}$ with $n \in N_{G'(F)} (T_0')$ and $p \in P(F)$
by the Bruhat decomposition, and moreover the Weyl group element $\bar{n} \in W = N_{G'(F)} (T_0') / T_0' (F)$ defined by $n$ can be taken of minimal length in its $(W_{Q''}, W_P)$-double coset.
Then $Q'=(Q'')^n$ works by \cite{MR794307}*{Proposition 2.8.9}, and uniqueness follows from \cite{MR794307}*{Proposition 2.3.3, 2.8.1}.

Now let $\Lorb$ be a nilpotent orbit in $\Liealg{G}'$ and let $(G,V,\rho)$ be the regular {\PVS} of DK-type attached to $\Lorb$ (cf. \S\ref{sec: nilPVS}).
Thus, $G$ is the standard Levi subgroup of the (standard) canonical parabolic subgroup $P$ of $\Lorb$.
We write $\Lorb_\filt=\Lorb\cap\filt_2$, where $\filt$ is the standard $\Lorb$-filtration (i.e., with $\Stab(\filt)=P$).

Denote by $[\mathfrak{F}]$ the (finite) set of $G'$-orbits of I-filtrations compatible with $\Lorb$. The elements of $[\mathfrak{F}]$ give rise to {\SPCL} subspaces of $V$. More precisely,

\begin{proposition} \label{prop: specfid}
For any $G'$-orbit $\Cls\in[\mathfrak{F}]$ there exists a unique I-filtration $\tilde\filt \in\Cls$ such that $\Stab(\tilde\filt)\cap G$ is a standard parabolic
subgroup of $G$ and $\Lorb \cap \filt_2 \cap \tilde\filt_2 \ne\emptyset$. Moreover, $U:=\tilde\filt_2\cap V\in\Spcl(V)$ and $\Stab(U)=\Stab(\tilde\filt)\cap G$.
\end{proposition}

\begin{proof}
We first show that the set
\[
\Cls'=\{\tilde\filt\in\Cls\mid \Lorb \cap \filt_2 \cap \tilde\filt_2
\ne\emptyset\}
\]
is a $P$-orbit. Clearly, $\Cls'$ is stable under $P$, and it is non-empty since
$\Lorb \cap \tilde\filt_2 \ne\emptyset$ for all $\tilde\filt \in\Cls$.
Assume that $\tilde\filt\in\Cls'$
and $\Ad(g)\tilde\filt\in\Cls'$ for some $g\in G'$. We need to show that $g\in PP'$ for $P'=\Stab(\tilde\filt)$. By assumption,
$\Lorb_{\tilde\filt}$ intersects both $\Lorb_\filt$ and $\Ad(g)^{-1}\Lorb_\filt$.
Since $\Lorb_{\tilde\filt}$ is a $P'$-orbit by Lemma \ref{lem: basicind} part \ref{part: projind}, there exists $p'\in P'$ such that
\[
\Ad(p')(\Lorb_\filt)\cap\Ad(g)^{-1}(\Lorb_\filt)\ne\emptyset.
\]
By Theorem \ref{thm: stdunip} we conclude that $gp'\in P$, as required.

The first assertion follows now from Remark \ref{rem: unqstd}.
Let $\tilde\filt = (\tilde\filt_i)_i \in\Cls'$ be the unique I-filtration provided by this construction. As before, let $P'=\Stab(\tilde\filt)$.
Since $P'$ contains $T_0'$, we have
\[
\tilde\filt_2\cap\filt_2=(\tilde\filt_2\cap V)+(\tilde\filt_2\cap\filt_3).
\]
Since $\Lorb_{\filt}=\Vreg+\filt_3$ by Theorem \ref{thm: stdunip}, we conclude that
\[
\Lorb_{\tilde\filt}\cap\Lorb_{\filt}=\Lorb\cap\tilde\filt_2\cap\filt_2=(\Lorb_{\tilde\filt}\cap \Vreg)+(\tilde\filt_2\cap\filt_3).
\]
In particular, $\Lorb_{\tilde\filt}\cap \Vreg\ne\emptyset$, so that $U\in\Minset(V)$.

Clearly, $\Stab(U)$ contains the standard parabolic subgroup $P'\cap G$ of $G$.
Let $v_0\in\Lorb\cap U$.
It follows from Lemma \ref{lem: basicind} part \ref{part: moveind} that there exist $\eta_1,\dots,\eta_k\in G'$ such that
\[
\{g\in G\mid \Ad(g)^{-1}v_0\in U\}=\{g\in G\mid \Ad(g)^{-1}v_0\in\Lorb_{\tilde\filt}\}=\bigcup_{1 \le i \le k} \eta_iP'\cap G.
\]
Each non-empty intersection $\eta_iP'\cap G$ is a left coset of $P'\cap G$ in $G$.
The second assertion of the proposition now follows from Proposition \ref{prop: cent} part \ref{part: special}.
\end{proof}

For instance, if $\Lorb=\Rich(Q)$, where $Q$ is a parabolic $F$-subgroup of $G'$, then Proposition \ref{prop: specfid}
says that there is a unique $G(F)$-conjugate $Q'$ of $Q$
such that $\Liealg{N}_{Q'}\cap\Lorb_\filt\ne\emptyset$ and $Q'\cap G$ is a standard parabolic subgroup of $G$.
In general, $Q'\cap P$ is not a parabolic subgroup of $G'$. In other words, if $Q$ is standard, it is not necessary that $Q'=Q$.

Consider the following examples.

\begin{enumerate}
\item Rank two cases (see Example \ref{ex: rank2}).

We may write the subregular orbit $\Lorb$ as $\Rich(Q)$ where $\Delta_0^Q=\{\beta\}$. This {\IFD} gives rise to $U = V \cap \Liealg{N}_{Q}$, the unique non-trivial special subspace.

\item Example \ref{ex: GL2'} (= Example \ref{ex: GL} with $k=3$ and $n_1<n_2$).

We have $\Lorb=\Rich(Q_1)=\Rich(Q_2)$ where $Q_1$ and $Q_2$ are the standard parabolic subgroups of type
$(n_2,n_1,n_1)$ and $(n_1,n_1,n_2)$. These {\IFD}s give rise to the spaces $V_i=V\cap\Liealg{N}_{Q_i}$ in $\Spcl(V)$.

On the other hand, if $k=5$ and $n_1<n_2<n_3$ and $Q$ is the standard parabolic subgroup of type $(n_3,n_2,n_2,n_1,n_1)$,
then $\Lorb=\Rich(Q)$ but $V \cap \Liealg{N}_{Q}\notin\Minset(V)$.

\item Example \ref{ex: F4'} (= Example \ref{ex: F4}).
In the following two examples, for standard parabolic subgroups
$Q_1 \supset Q_2$ we will use the notation 
$\Rich^{Q_1}(Q_2)$ for the Richardson orbit of the standard parabolic subgroup $L_1 
\cap Q_2$ inside the standard Levi subgroup $L_1$ of $Q_1$. 

We have $\Lorb=\Rich(({*}{*}0{*}))=\Rich((0{*}{*}0))$. This will give rise to the subspaces $U_1$ and $U_2$, respectively. 
Using induction by stages, we may also write $\Lorb=\Ind_Q\Lorb'$ where $Q=(0{*}{*}{*})$ and $\Lorb'=\Rich^Q((0{*}{*}0))$.

The weighted Dynkin diagram of $\Lorb'$ is
\begin{tikzcd}[sep=small]
\overset{0}{\bullet}\arrow[r,Rightarrow]&\overset{1}{\bullet}\arrow[r,dash]&\overset{0}{\bullet}.
\end{tikzcd}
The standard $\Lorb'$-filtration satisfies $\Phi_{\filt'_2}=\{(0120),(0121),(0122)\}$, and the associated I-filtration
$\Phi_{\tilde\filt_2}= \Phi^+_{G'} \setminus 
\{(0100),(0010),(0001),(0110),(0011),(0111)\}$.
Therefore, this {\IFD} gives rise to the space $U_3$.

Note that writing $\Lorb=\Ind_{({*}{*}{*}0)}\Lorb''$ where $\Lorb''=\Rich^{({*}{*}{*}0)}((0{*}{*}0))$ will give rise to $U_2$ again, since
the weighted Dynkin diagram of $\Lorb''$ is
\begin{tikzcd}[sep=small]
\overset{2}{\bullet}\arrow[r,dash]&\overset{0}{\bullet}\arrow[r,Rightarrow]&\overset{0}{\bullet}
\end{tikzcd}
and therefore $\tilde\filt_2 = \Liealg{N}_{(0{*}{*}0)}$ again.

\item Example \ref{ex: E6'} (= Example \ref{ex: E6})

We have $\Lorb=\Rich(Q'_1)=\Rich(Q'_2)$, where $Q'_1=\Eroot{0**0*}*$, $Q'_2=\Eroot{***0*}0$ (which are associate).
These give rise to $U_1$ and $U_2$, respectively. 

They also give three presentations $\Lorb=\Ind_{Q_i}\Lorb'_i$, $i=1,2,3$ where
\begin{align*}
Q_1=\Eroot{***0*}*,\ \ \Lorb'_1&=\Rich^{Q_1}(Q_1')=\Rich^{Q_1}(Q_2'),\\
Q_2=\Eroot{*****}0,\ \ \Lorb'_2&=\Rich^{Q_2}(Q_2'),\\
Q_3=\Eroot{0****}*,\ \ \Lorb'_3&=\Rich^{Q_3}(Q_1').
\end{align*}
Note that $L_1^{\der}=\SL_5\times\SL_2$ and $\Lorb'_1$ is the minimal orbit of the $\SL_5$-factor.
The weighted Dynkin diagram is $\Eroot{100\phantom{0}0}1$. We have $\Phi_{\filt'_2}=\{\Eroot{11100}1\}$. We get the space $U_1\cap U_2$.

The weighted Dynkin diagram of $\Lorb'_2$ is $\Eroot{01010}{}$ and
\[
\Phi_{\filt'_2}=\{\Eroot{01110}0,\Eroot{11110}0,\Eroot{01111}0,\Eroot{11111}0\}.
\]
We get the space $U_2''$.

The weighted Dynkin diagram of $\Lorb'_3$ is $\Eroot{\phantom{0}0020}0$, and we get $U_1$ again.

Of course, we can apply the involution $\theta$ to all of the above, obtaining in addition $U_1'$, $U_2'$ and $U_1' \cap U_2'$.

Finally, we also have $\Lorb=\Ind_Q\Lorb'$ where $Q=\Eroot{0***0}*$ and $\Lorb'$ is the minimal orbit of the Levi subgroup of $Q$ (whose derived group is $\Spin_8$).
The weighted Dynkin diagram of $\Lorb'$ is $\Eroot{\phantom{0}010\phantom{0}}0$ and $\Phi_{\filt'_2}= \{ \Eroot{01210}1 \}$.
Here, $\tilde\filt_2\cap\Lorb_{\filt}=\emptyset$ for the standard I-filtration
$\tilde\filt$.
However, it is easy to see that if $w$ is the simple reflection with respect to $\Eroot{00100}0$ (the simple root defining $G$),
then $\Stab(w\tilde\filt) \cap G$ is the standard Borel subgroup of $G$ and $w\tilde\filt_2\cap\Lorb_{\filt}\ne\emptyset$. The intersection $w\tilde\filt_2\cap V$ is $U_3''$.  

Note that $U_3$ is not obtained from this construction (either directly or by intersection).
\end{enumerate}

%\bibliographystyle{amsalpha}
%\bibliographystyle{plainnat}
%\bibliography{../Bibfiles/all}

\begin{bibdiv}
\begin{biblist}

\bib{MR2017446}{book}{
       title={Rev\^{e}tements \'{e}tales et groupe fondamental ({SGA} 1)},
      series={Documents Math\'{e}matiques (Paris) [Mathematical Documents
  (Paris)]},
   publisher={Soci\'{e}t\'{e} Math\'{e}matique de France, Paris},
        date={2003},
      volume={3},
        ISBN={2-85629-141-4},
        note={S\'{e}minaire de g\'{e}om\'{e}trie alg\'{e}brique du Bois Marie
  1960--61. [Algebraic Geometry Seminar of Bois Marie 1960-61], Directed by A.
  Grothendieck, With two papers by M. Raynaud, Updated and annotated reprint of
  the 1971 original [Lecture Notes in Math., 224, Springer, Berlin; MR0354651
  (50 \#7129)]},
      review={\MR{2017446}},
}

\bib{MR2275598}{incollection}{
      author={Bhargava, Manjul},
       title={Higher composition laws and applications},
        date={2006},
   booktitle={International {C}ongress of {M}athematicians. {V}ol. {II}},
   publisher={Eur. Math. Soc., Z\"{u}rich},
       pages={271\ndash 294},
      review={\MR{2275598}},
}

\bib{2107.12819}{article}{
      author={Bhargava, Manjul},
      author={Taniguchi, Takashi},
      author={Thorne, Frank},
       title={Improved error estimates for the {D}avenport-{H}eilbronn
  theorems},
        date={2023},
     journal={Math. Ann.},
      volume={online},
         url={https://doi.org/10.1007/s00208-023-02684-w},
}

\bib{MR0294349}{article}{
      author={Borel, A.},
      author={Tits, J.},
       title={\'{E}l\'ements unipotents et sous-groupes paraboliques de groupes
  r\'eductifs. {I}},
        date={1971},
        ISSN={0020-9910},
     journal={Invent. Math.},
      volume={12},
       pages={95\ndash 104},
      review={\MR{MR0294349 (45 \#3419)}},
}

\bib{MR1102012}{book}{
      author={Borel, Armand},
       title={Linear algebraic groups},
     edition={Second},
      series={Graduate Texts in Mathematics},
   publisher={Springer-Verlag},
     address={New York},
        date={1991},
      volume={126},
        ISBN={0-387-97370-2},
      review={\MR{MR1102012 (92d:20001)}},
}

\bib{MR2109105}{book}{
      author={Bourbaki, Nicolas},
       title={Lie groups and {L}ie algebras. {C}hapters 7--9},
      series={Elements of Mathematics (Berlin)},
   publisher={Springer-Verlag},
     address={Berlin},
        date={2005},
        ISBN={3-540-43405-4},
        note={Translated from the 1975 and 1982 French originals by Andrew
  Pressley},
      review={\MR{MR2109105 (2005h:17001)}},
}

\bib{MR794307}{book}{
      author={Carter, Roger~W.},
       title={Finite groups of {L}ie type},
      series={Pure and Applied Mathematics (New York)},
   publisher={John Wiley \& Sons Inc.},
     address={New York},
        date={1985},
        ISBN={0-471-90554-2},
        note={Conjugacy classes and complex characters, A Wiley-Interscience
  Publication},
      review={\MR{MR794307 (87d:20060)}},
}

\bib{MR3625130}{article}{
      author={Chaudouard, Pierre-Henri},
       title={Sur la contribution unipotente dans la formule des traces
  d'{A}rthur pour les groupes g\'{e}n\'{e}raux lin\'{e}aires},
        date={2017},
        ISSN={0021-2172},
     journal={Israel J. Math.},
      volume={218},
      number={1},
       pages={175\ndash 271},
         url={https://doi.org/10.1007/s11856-017-1464-4},
      review={\MR{3625130}},
}

\bib{MR1251060}{book}{
      author={Collingwood, David~H.},
      author={McGovern, William~M.},
       title={Nilpotent orbits in semisimple {L}ie algebras},
      series={Van Nostrand Reinhold Mathematics Series},
   publisher={Van Nostrand Reinhold Co.},
     address={New York},
        date={1993},
        ISBN={0-534-18834-6},
      review={\MR{MR1251060 (94j:17001)}},
}

\bib{MR936994}{article}{
      author={Datskovsky, Boris},
      author={Wright, David~J.},
       title={Density of discriminants of cubic extensions},
        date={1988},
        ISSN={0075-4102},
     journal={J. Reine Angew. Math.},
      volume={386},
       pages={116\ndash 138},
         url={https://doi.org/10.1515/crll.1988.386.116},
      review={\MR{936994}},
}

\bib{MR3969874}{incollection}{
      author={Finis, Tobias},
      author={Hoffmann, Werner},
      author={Wakatsuki, Satoshi},
       title={The subregular unipotent contribution to the geometric side of
  the {A}rthur trace formula for the split exceptional group {$G_2$}},
        date={2018},
   booktitle={Geometric aspects of the trace formula},
      series={Simons Symp.},
   publisher={Springer, Cham},
       pages={163\ndash 182},
      review={\MR{3969874}},
}

\bib{MR3534542}{article}{
      author={Finis, Tobias},
      author={Lapid, Erez},
       title={On the continuity of the geometric side of the trace formula},
        date={2016},
        ISSN={0251-4184},
     journal={Acta Math. Vietnam.},
      volume={41},
      number={3},
       pages={425\ndash 455},
         url={https://doi.org/10.1007/s40306-016-0176-x},
      review={\MR{3534542}},
}

\bib{MR63358}{article}{
      author={Freudenthal, Hans},
       title={Beziehungen der {$E_7$} und {$E_8$} zur {O}ktavenebene. {I}},
        date={1954},
     journal={Indag. Math.},
      volume={16},
       pages={218\ndash 230},
        note={Nederl. Akad. Wetensch. Proc. Ser. A {{\bf{5}}7}},
      review={\MR{63358}},
}

\bib{MR1125214}{article}{
      author={Gyoja, Akihiko},
       title={Invariants, nilpotent orbits, and prehomogeneous vector spaces},
        date={1991},
        ISSN={0021-8693},
     journal={J. Algebra},
      volume={142},
      number={1},
       pages={210\ndash 232},
         url={https://doi.org/10.1016/0021-8693(91)90226-X},
      review={\MR{1125214}},
}

\bib{MR463157}{book}{
      author={Hartshorne, Robin},
       title={Algebraic geometry},
      series={Graduate Texts in Mathematics, No. 52},
   publisher={Springer-Verlag, New York-Heidelberg},
        date={1977},
        ISBN={0-387-90244-9},
      review={\MR{463157}},
}

\bib{MR1438546}{book}{
      author={Hilton, P.~J.},
      author={Stammbach, U.},
       title={A course in homological algebra},
     edition={Second},
      series={Graduate Texts in Mathematics},
   publisher={Springer-Verlag, New York},
        date={1997},
      volume={4},
        ISBN={0-387-94823-6},
         url={https://doi.org/10.1007/978-1-4419-8566-8},
      review={\MR{1438546}},
}

\bib{MR620024}{book}{
      author={Hochschild, Gerhard~P.},
       title={Basic theory of algebraic groups and {L}ie algebras},
      series={Graduate Texts in Mathematics},
   publisher={Springer-Verlag, New York-Berlin},
        date={1981},
      volume={75},
        ISBN={0-387-90541-3},
      review={\MR{620024}},
}

\bib{MR903631}{article}{
      author={Hoffmann, Werner},
       title={The nonsemisimple term in the trace formula for rank one
  lattices},
        date={1987},
        ISSN={0075-4102},
     journal={J. Reine Angew. Math.},
      volume={379},
       pages={1\ndash 21},
         url={https://doi.org/10.1515/crll.1987.379.1},
      review={\MR{903631}},
}

\bib{MR3675167}{incollection}{
      author={Hoffmann, Werner},
       title={The trace formula and prehomogeneous vector spaces},
        date={2016},
   booktitle={Families of automorphic forms and the trace formula},
      series={Simons Symp.},
   publisher={Springer, [Cham]},
       pages={175\ndash 215},
      review={\MR{3675167}},
}

\bib{MR3843151}{article}{
      author={Hoffmann, Werner},
      author={Wakatsuki, Satoshi},
       title={On the geometric side of the {A}rthur trace formula for the
  symplectic group of rank 2},
        date={2018},
        ISSN={0065-9266},
     journal={Mem. Amer. Math. Soc.},
      volume={255},
      number={1224},
       pages={v+88},
         url={https://doi.org/10.1090/memo/1224},
      review={\MR{3843151}},
}

\bib{MR0447140}{article}{
      author={Igusa, Jun-ichi},
       title={Exponential sums associated with a {F}reudenthal quartic},
        date={1977},
        ISSN={0040-8980},
     journal={J. Fac. Sci. Univ. Tokyo Sect. IA Math.},
      volume={24},
      number={1},
       pages={231\ndash 245},
      review={\MR{0447140 (56 \#5455)}},
}

\bib{MR910424}{article}{
      author={Igusa, Jun-ichi},
       title={On a certain class of prehomogeneous vector spaces},
        date={1987},
        ISSN={0022-4049},
     journal={J. Pure Appl. Algebra},
      volume={47},
      number={3},
       pages={265\ndash 282},
         url={http://dx.doi.org/10.1016/0022-4049(87)90051-X},
      review={\MR{910424 (88m:20089)}},
}

\bib{MR2138067}{article}{
      author={Kable, Anthony~C.},
      author={Yukie, Akihiko},
       title={On the number of quintic fields},
        date={2005},
        ISSN={0020-9910},
     journal={Invent. Math.},
      volume={160},
      number={2},
       pages={217\ndash 259},
         url={https://doi.org/10.1007/s00222-004-0391-2},
      review={\MR{2138067}},
}

\bib{MR575790}{article}{
      author={Kac, V.~G.},
       title={Some remarks on nilpotent orbits},
        date={1980},
        ISSN={0021-8693},
     journal={J. Algebra},
      volume={64},
      number={1},
       pages={190\ndash 213},
         url={http://dx.doi.org/10.1016/0021-8693(80)90141-6},
      review={\MR{575790 (81i:17005)}},
}

\bib{MR1944442}{book}{
      author={Kimura, Tatsuo},
       title={Introduction to prehomogeneous vector spaces},
      series={Translations of Mathematical Monographs},
   publisher={American Mathematical Society, Providence, RI},
        date={2003},
      volume={215},
        ISBN={0-8218-2767-7},
        note={Translated from the 1998 Japanese original by Makoto Nagura and
  Tsuyoshi Niitani and revised by the author},
      review={\MR{1944442}},
}

\bib{MR527733}{article}{
      author={Lusztig, G.},
      author={Spaltenstein, N.},
       title={Induced unipotent classes},
        date={1979},
        ISSN={0024-6107},
     journal={J. London Math. Soc. (2)},
      volume={19},
      number={1},
       pages={41\ndash 52},
         url={http://dx.doi.org/10.1112/jlms/s2-19.1.41},
      review={\MR{527733}},
}

\bib{MR1361168}{book}{
      author={M\oe~glin, C.},
      author={Waldspurger, J.-L.},
       title={Spectral decomposition and {E}isenstein series},
      series={Cambridge Tracts in Mathematics},
   publisher={Cambridge University Press, Cambridge},
        date={1995},
      volume={113},
        ISBN={0-521-41893-3},
         url={https://doi.org/10.1017/CBO9780511470905},
        note={Une paraphrase de l'\'{E}criture [A paraphrase of Scripture]},
      review={\MR{1361168}},
}

\bib{MR1413009}{article}{
      author={Morishita, Masanori},
      author={Watanabe, Takao},
       title={A note on the mean value theorem for special homogeneous spaces},
        date={1996},
        ISSN={0027-7630},
     journal={Nagoya Math. J.},
      volume={143},
       pages={111\ndash 117},
         url={http://projecteuclid.org/getRecord?id=euclid.nmj/1118771971},
      review={\MR{1413009 (97h:11042)}},
}

\bib{MR0092928}{article}{
      author={Mostow, G.~D.},
       title={Fully reducible subgroups of algebraic groups},
        date={1956},
        ISSN={0002-9327},
     journal={Amer. J. Math.},
      volume={78},
       pages={200\ndash 221},
      review={\MR{MR0092928 (19,1181f)}},
}

\bib{MR0230726}{article}{
      author={Ono, Takashi},
       title={A mean value theorem in adele geometry},
        date={1968},
        ISSN={0025-5645},
     journal={J. Math. Soc. Japan},
      volume={20},
       pages={275\ndash 288},
      review={\MR{0230726}},
}

\bib{MR357399}{article}{
      author={Popov, V.~L.},
       title={Picard groups of homogeneous spaces of linear algebraic groups
  and one-dimensional homogeneous vector fiberings},
        date={1974},
        ISSN={0373-2436},
     journal={Izv. Akad. Nauk SSSR Ser. Mat.},
      volume={38},
       pages={294\ndash 322},
      review={\MR{357399}},
}

\bib{MR0320232}{article}{
      author={Ranga~Rao, R.},
       title={Orbital integrals in reductive groups},
        date={1972},
        ISSN={0003-486X},
     journal={Ann. of Math. (2)},
      volume={96},
       pages={505\ndash 510},
      review={\MR{MR0320232 (47 \#8771)}},
}

\bib{MR260758}{book}{
      author={Raynaud, Michel},
       title={Faisceaux amples sur les sch\'{e}mas en groupes et les espaces
  homog\`enes},
      series={Lecture Notes in Mathematics, Vol. 119},
   publisher={Springer-Verlag, Berlin-New York},
        date={1970},
      review={\MR{260758}},
}

\bib{MR1247502}{incollection}{
      author={R\"{o}hrle, Gerhard~E.},
       title={On extraspecial parabolic subgroups},
        date={1993},
   booktitle={Linear algebraic groups and their representations ({L}os
  {A}ngeles, {CA}, 1992)},
      series={Contemp. Math.},
      volume={153},
   publisher={Amer. Math. Soc., Providence, RI},
       pages={143\ndash 155},
         url={https://doi-org.ezproxy.weizmann.ac.il/10.1090/conm/153/01310},
      review={\MR{1247502}},
}

\bib{MR2961852}{article}{
      author={Rubenthaler, Hubert},
       title={Decomposition of reductive regular prehomogeneous vector spaces},
        date={2011},
        ISSN={0373-0956},
     journal={Ann. Inst. Fourier (Grenoble)},
      volume={61},
      number={5},
       pages={2183\ndash 2218 (2012)},
         url={https://doi.org/10.5802/aif.2670},
      review={\MR{2961852}},
}

\bib{MR1731463}{article}{
      author={Saito, Hiroshi},
       title={Explicit form of the zeta functions of prehomogeneous vector
  spaces},
        date={1999},
        ISSN={0025-5831},
     journal={Math. Ann.},
      volume={315},
      number={4},
       pages={587\ndash 615},
         url={http://dx.doi.org/10.1007/s002080050330},
      review={\MR{1731463 (2001b:11111)}},
}

\bib{MR1928799}{article}{
      author={Saito, Hiroshi},
       title={Global zeta functions of {F}reudenthal quartics},
        date={2002},
        ISSN={0129-167X},
     journal={Internat. J. Math.},
      volume={13},
      number={8},
       pages={797\ndash 820},
         url={http://dx.doi.org/10.1142/S0129167X02001496},
      review={\MR{1928799 (2003i:11051)}},
}

\bib{MR1994885}{article}{
      author={Saito, Hiroshi},
       title={Convergence of the zeta functions of prehomogeneous vector
  spaces},
        date={2003},
        ISSN={0027-7630},
     journal={Nagoya Math. J.},
      volume={170},
       pages={1\ndash 31},
         url={http://projecteuclid.org/getRecord?id=euclid.nmj/1114631874},
      review={\MR{1994885 (2005d:11170)}},
}

\bib{MR676121}{article}{
      author={Sat\B{o}, Fumihiro},
       title={Zeta functions in several variables associated with
  prehomogeneous vector spaces. {I}. {F}unctional equations},
        date={1982},
        ISSN={0040-8735},
     journal={Tohoku Math. J. (2)},
      volume={34},
      number={3},
       pages={437\ndash 483},
         url={https://doi.org/10.2748/tmj/1178229205},
      review={\MR{676121}},
}

\bib{MR695661}{article}{
      author={Sat\B{o}, Fumihiro},
       title={Zeta functions in several variables associated with
  prehomogeneous vector spaces. {II}. {A} convergence criterion},
        date={1983},
        ISSN={0040-8735},
     journal={Tohoku Math. J. (2)},
      volume={35},
      number={1},
       pages={77\ndash 99},
         url={https://doi.org/10.2748/tmj/1178229103},
      review={\MR{695661}},
}

\bib{MR430336}{article}{
      author={Sato, M.},
      author={Kimura, T.},
       title={A classification of irreducible prehomogeneous vector spaces and
  their relative invariants},
        date={1977},
        ISSN={0027-7630},
     journal={Nagoya Math. J.},
      volume={65},
       pages={1\ndash 155},
         url={http://projecteuclid.org/euclid.nmj/1118796150},
      review={\MR{430336}},
}

\bib{MR1086566}{article}{
      author={Sato, Mikio},
       title={Theory of prehomogeneous vector spaces (algebraic part)---the
  {E}nglish translation of {S}ato's lecture from {S}hintani's note},
        date={1990},
        ISSN={0027-7630},
     journal={Nagoya Math. J.},
      volume={120},
       pages={1\ndash 34},
         url={https://doi.org/10.1017/S0027763000003214},
        note={Notes by Takuro Shintani, Translated from the Japanese by
  Masakazu Muro},
      review={\MR{1086566}},
}

\bib{MR296079}{article}{
      author={Sato, Mikio},
      author={Shintani, Takuro},
       title={On zeta functions associated with prehomogeneous vector spaces},
        date={1972},
        ISSN={0027-8424},
     journal={Proc. Nat. Acad. Sci. U.S.A.},
      volume={69},
       pages={1081\ndash 1082},
         url={https://doi.org/10.1073/pnas.69.5.1081},
      review={\MR{296079}},
}

\bib{MR0344230}{article}{
      author={Sato, Mikio},
      author={Shintani, Takuro},
       title={On zeta functions associated with prehomogeneous vector spaces},
        date={1974},
        ISSN={0003-486X},
     journal={Ann. of Math. (2)},
      volume={100},
       pages={131\ndash 170},
      review={\MR{MR0344230 (49 \#8969)}},
}

\bib{MR1466966}{book}{
      author={Serre, Jean-Pierre},
       title={Galois cohomology},
   publisher={Springer-Verlag},
     address={Berlin},
        date={1997},
        ISBN={3-540-61990-9},
        note={Translated from the French by Patrick Ion and revised by the
  author},
      review={\MR{1466966 (98g:12007)}},
}

\bib{MR0292755}{article}{
      author={Shintani, Takuro},
       title={On {D}irichlet series whose coefficients are class numbers of
  integral binary cubic forms},
        date={1970},
        ISSN={0021-4280},
     journal={Proc. Japan Acad.},
      volume={46},
       pages={909\ndash 911},
      review={\MR{MR0292755 (45 \#1837)}},
}

\bib{MR0289428}{article}{
      author={Shintani, Takuro},
       title={On {D}irichlet series whose coefficients are class numbers of
  integral binary cubic forms},
        date={1972},
     journal={J. Math. Soc. Japan},
      volume={24},
       pages={132\ndash 188},
      review={\MR{MR0289428 (44 \#6619)}},
}

\bib{MR0384717}{article}{
      author={Shintani, Takuro},
       title={On zeta-functions associated with the vector space of quadratic
  forms},
        date={1975},
     journal={J. Fac. Sci. Univ. Tokyo Sect. I A Math.},
      volume={22},
       pages={25\ndash 65},
      review={\MR{MR0384717 (52 \#5590)}},
}

\bib{MR0430168}{article}{
      author={Vinberg, {\`E}.~B.},
       title={The {W}eyl group of a graded {L}ie algebra},
        date={1976},
        ISSN={0373-2436},
     journal={Izv. Akad. Nauk SSSR Ser. Mat.},
      volume={40},
      number={3},
       pages={488\ndash 526, 709},
      review={\MR{0430168 (55 \#3175)}},
}

\bib{MR776169}{article}{
      author={Wright, David~J.},
       title={The adelic zeta function associated to the space of binary cubic
  forms. {I}. {G}lobal theory},
        date={1985},
        ISSN={0025-5831},
     journal={Math. Ann.},
      volume={270},
      number={4},
       pages={503\ndash 534},
         url={https://doi.org/10.1007/BF01455301},
      review={\MR{776169}},
}

\bib{MR1185585}{article}{
      author={Wright, David~J.},
      author={Yukie, Akihiko},
       title={Prehomogeneous vector spaces and field extensions},
        date={1992},
        ISSN={0020-9910},
     journal={Invent. Math.},
      volume={110},
      number={2},
       pages={283\ndash 314},
         url={https://doi.org/10.1007/BF01231334},
      review={\MR{1185585}},
}

\bib{MR1149040}{article}{
      author={Yukie, Akihiko},
       title={On the {S}hintani zeta function for the space of binary quadratic
  forms},
        date={1992},
        ISSN={0025-5831},
     journal={Math. Ann.},
      volume={292},
      number={2},
       pages={355\ndash 374},
         url={https://doi.org/10.1007/BF01444626},
      review={\MR{1149040}},
}

\bib{MR1267735}{book}{
      author={Yukie, Akihiko},
       title={Shintani zeta functions},
      series={London Mathematical Society Lecture Note Series},
   publisher={Cambridge University Press, Cambridge},
        date={1993},
      volume={183},
        ISBN={0-521-44804-2},
      review={\MR{1267735}},
}

\bib{MR1369477}{article}{
      author={Yukie, Akihiko},
       title={On the convergence of the zeta function for certain
  prehomogeneous vector spaces},
        date={1995},
        ISSN={0027-7630},
     journal={Nagoya Math. J.},
      volume={140},
       pages={1\ndash 31},
         url={https://doi.org/10.1017/S0027763000005390},
      review={\MR{1369477}},
}

\end{biblist}
\end{bibdiv}

\def\cprime{$'$}

\end{document}